\documentclass[12pt]{amsart}
\usepackage[cp1251]{inputenc}
\usepackage{a4wide,amsmath,eucal,verbatim,amsopn,amsthm,amsfonts,amssymb,mathdots,mathrsfs,esint,mathtools,enumitem,mnsymbol,bbm,hhline,url}
\usepackage{amsmath,amssymb,latexsym}
\usepackage{amscd}
\usepackage[stable]{footmisc}
\usepackage[hyperfootnotes=false,linktocpage=true,colorlinks=true,allcolors=blue]{hyperref}
\theoremstyle{plain}
\newtheorem{theorem}{Theorem}[section]
\newtheorem{proposition}[theorem]{Proposition}
\newtheorem{corollary}[theorem]{Corollary}
\newtheorem{lemma}[theorem]{Lemma}
\newtheorem{definition}[theorem]{Definition}

\newtheorem{remark}[theorem]{Remark}
\newtheorem{example}[theorem]{Example}

\newcommand{\complexSbrevity}{s_{k,c}}
\newtheorem{apptheorem}{Theorem}[section]
\newtheorem{appproposition}[apptheorem]{Proposition}
\newtheorem{applemma}[apptheorem]{Lemma}

\newtheorem{appexample}[apptheorem]{Example}
\newtheorem{appremark}[apptheorem]{Remark}
\DeclareMathOperator{\Span}{\mathrm{Span}}

\numberwithin{equation}{section}

\newcommand{\SUParts}{\mathcal{S}^{\mathrm{lin}}(U_R)}
\newcommand{\SchwartzC}{\mathcal{S}_{\bold{c}}}
\newcommand{\SUPartsc}{\SchwartzC^{\mathrm{lin}}(U_R)}

\newcommand{\faithful}{r_G}
\newcommand{\faithfulG}[1]{r_{#1}}
\newcommand{\Umacro}{V}
\newcommand{\orb}{\mathfrak{C}}
\newcommand{\Irr}{\mathrm{Irr}}
\newcommand{\IrrGen}{\mathrm{Irr}_{\mathrm{gen}}}
\newcommand{\IrrGenUni}{\mathrm{Irr}_{\mathrm{gen,u}}}
\newcommand{\IrrUnrForPoles}{\mathrm{Irr}_{\mathrm{rel}}}
\newcommand{\IrrUnrUniForPoles}{\mathrm{Irr}_{\mathrm{rel,a.u.}}}
\newcommand{\Specialjmath}{\jmath}
\newcommand{\WeylElement}{\mathfrak{w}}

\newcommand{\LeadingCoeff}{\mathscr{E}}
\newcommand{\localfield}{\mathfrak{f}}

\newcommand{\transfer}{\mathfrak{T}}

\newcommand{\ZInnerInner}{Z^2}
\def\dintegrallocal{d(o_1,o_2)}
\def\dintegrallocalVarG{o}

\newcommand{\Dpt}{\mathfrak{S}}
\def\Z{\mathbb{Z}}
\def\R{\mathbb{R}}
\def\A{\mathbb{A}}

\def\C{\mathbb{C}}
\def\Mat{\mathrm{Mat}}

\newcommand{\emptyblockforGL}{\oslash}
\newcommand{\circforgspin}{\bullet}
\newcommand{\ParabolicCT}{R}

\newcommand{\embedding}{\operatorname{\mathfrak{e}}}
\newcommand{\embeddingL}{\operatorname{\mathfrak{e}_1}}
\newcommand{\embeddingR}{\operatorname{\mathfrak{e}_2}}

\newcommand{\appembeddingOne}{\operatorname{\mathfrak{e}_1}}
\newcommand{\appembeddingTwo}{\operatorname{\mathfrak{e}_2}}

\DeclareMathOperator{\tr}{tr}

\DeclareMathOperator{\Hom}{Hom}

\DeclareMathOperator{\gcdzzz}{\mathfrak{L}}

\newcommand{\absdet}{|\cdot|}

\newcommand{\intertInduced}[2]{M_{#1}^{\square}(#2)}

\newcommand{\LeviZ}{r_z}
\newcommand{\LeviY}{r_{({}^{\WeylElement_0}y)}}

\newcommand{\LeviA}{m_a}
\newcommand{\Levi}[1]{\ell(#1)}
\newcommand{\Uni}[1]{\upsilon(#1)}
\newcommand{\Unim}[1]{\upsilon_{-}(#1)}
\newcommand{\Es}[1]{\boldsymbol{#1}}

\DeclareMathOperator{\BigLevi}{\mathscr{L}}
\DeclareMathOperator{\BigUni}{\mathscr{U}}

\newcommand{\temp}{\operatorname{temp}}
\newcommand{\disc}{\operatorname{disc}}

\newcommand{\autparam}{\phi}

\DeclareMathOperator{\diag}{diag}
\newcommand{\Dual}[1]{\widehat{#1}}

\DeclareMathOperator{\Ind}{Ind}
\DeclareMathOperator{\GL}{GL}
\DeclareMathOperator{\SL}{SL}
\DeclareMathOperator{\SO}{SO}
\DeclareMathOperator{\Spin}{Spin}
\DeclareMathOperator{\GSpin}{GSpin}
\DeclareMathOperator{\GPin}{GPin}

\DeclareMathOperator{\GSO}{GSO}
\DeclareMathOperator{\GO}{GO}
\DeclareMathOperator{\GSp}{GSp}
\DeclareMathOperator{\Orth}{O}
\DeclareMathOperator{\Sp}{Sp}
\DeclareMathOperator{\Real}{Re}

\newcommand{\bs}{\backslash}

\begin{document}
\title[Doubling constructions and functoriality]{Doubling constructions: Global functoriality for non-generic cuspidal representations}
\author{Yuanqing Cai}
\author{Solomon Friedberg}
\author{Eyal Kaplan}
\address{Cai: Faculty of Mathematics and Physics, Institute of Science and Engineering, Kanazawa University, Kakumamachi, Kanazawa, Ishikawa, 920-1192, Japan}
\email{cai@se.kanazawa-u.ac.jp}
\address{Friedberg:  Department of Mathematics, Boston College, Chestnut Hill, MA 02467-3806, USA}
\email{solomon.friedberg@bc.edu}
\address{Kaplan: Department of Mathematics, Bar Ilan University, Ramat Gan 5290002, Israel}
\email{kaplaney@gmail.com}
\thanks{This research was supported by the ERC, StG grant number 637912 (Cai),
by the JSPS KAKENHI grant number 19F19019 (Cai), by MEXT Leading Initiative for Excellent Young Researchers Grant Number JPMXS0320200394 (Cai),
by JSPS KAKENHI Grant Number 23K12951 (Cai), by the BSF, grant number 2012019 (Friedberg), by the NSF, grant numbers 1500977, 1801497 and 2100206 (Friedberg), and by the Israel Science Foundation, grant numbers 376/21 and 421/17 (Kaplan).}
\begin{abstract}
We study the generalized doubling method for pairs of representations of
$G\times\GL_k$ where $G$ is a symplectic group, split special orthogonal group or split general spin group. We analyze the poles of the local integrals, and prove that the global completed $L$-function with a cuspidal representation of $\GL_k$ twisted by a highly ramified Hecke character
is entire. We obtain a new proof of the weak functorial transfer of cuspidal automorphic representations of $G$ to the natural general linear group, which is independent of the trace formula and its prerequisites, by combining our results with the Converse Theorem.
\end{abstract}
\maketitle
\addtocontents{toc}{\protect\setcounter{tocdepth}{2}}
\section*{Introduction}\label{intro}

Let $F$ be a number field with a ring of adeles $\A$. Let $G$ be either a symplectic group or a quasi-split orthogonal group defined over $F$. In his monumental work, Arthur described the discrete part $\Pi_{\disc}(G)$ of
$L^2(G(F)\bs G(\A))$ in terms of automorphic representations of $\GL_N(\A)$ \cite[\S~1.5]{Arthur2013}, where $N$ is the dimension of the natural faithful representation $\faithful$ of the connected component of the Langlands dual group of $G$. More precisely, $\Pi_{\disc}(G)$ is the disjoint union of sets $\Pi_{\autparam}(G)$ ($A$-packets) of automorphic representations of $G(\A)$, where $\autparam$ ranges over a certain set $\tilde{\Psi}_2(G)$ of automorphic representations of
$\GL_N(\A)$ that appear in the spectral decomposition of $L^2(\GL_N(F)\bs\GL_N(\A))$.
In particular the representation $\autparam$ is isobaric. Moreover for every $\pi\in\Pi_{\autparam}(G)$, $\autparam$ is a weak transfer of $\pi$ in the sense that $\autparam_{\nu}$ is the local unramified transfer of $\pi_{\nu}$ to $\GL_N(F_{\nu})$ at almost all finite places $\nu$ of $F$, and at the infinite places $\nu$ the infinitesimal character of $\autparam_{\nu}$ is determined by that of $\pi_{\nu}$ via $r_G$. In contrast with $\Pi_{\disc}(G)$, the set $\tilde{\Psi}_2(G)$ satisfies strong multiplicity one, i.e., each $\autparam\in\tilde{\Psi}_2(G)$ is determined by its local components almost everywhere.

Arthur's strategy is based on a comparison of a stable trace formula for classical groups and
a twisted stable trace formula for $\GL_N$. The stabilization of the trace formula and the twisted one
are therefore indispensable prerequisites. The former was carried out by Arthur \cite{Arthur2001,Arthur2002,Arthur2003prt3} and the latter
by M{\oe}glin and Waldspurger \cite{MoeglinWaldspurger2016a,MoeglinWaldspurger2016b}. Additional crucial ingredients are the fundamental lemma
proved by Ng\^{o} \cite{Ngo2010} and many other important aspects of local harmonic analysis due to Langlands, Kottwitz, Shelstad, Waldspurger, Chaudouard and Laumon among others.
(At the time of writing this paper, several references in \cite{Arthur2013} remain unpublished.)
See \cite{Mok2015,Kalethaetal2014} for implementations of Arthur's strategy to unitary groups and to non-quasi-split groups.

In this paper we present a new proof of the weak functorial transfer from split classical groups to general linear groups,
one that is independent of the trace formula and its prerequisites. We also establish the transfer from split general spin groups.
In principle, Arthur's strategy is applicable to these groups as well. However, to the best of our knowledge this has not been carried out to obtain a global transfer. In the local setup see M{\oe}glin \cite{Moeglin2014} for a simplification of Arthur's argument for the local Langlands correspondence for classical groups and also for general spin groups.

Let $G$ be either a split symplectic group or a split special orthogonal group of rank $n$, or a split general spin group of rank $n+1$, for $n\geq1$.
Define $\faithful$ as above. Then $N=2n$ except in the symplectic case where $N=2n+1$. Here is our main result.
\begin{theorem}\label{theo:globl functorial lift}
Any irreducible cuspidal automorphic representation $\pi$ of $G(\A)$ has a weak functorial transfer $\Pi$ to $\GL_N(\A)$.
\end{theorem}

By a basic result of Langlands \cite{BJ1977}, any irreducible automorphic representation of $G(\A)$ is a constituent of the parabolic induction of an irreducible cuspidal automorphic representation of a Levi subgroup. Thus Theorem~\ref{theo:globl functorial lift} immediately extends to any irreducible automorphic representation.

Let $\pi$ be an irreducible cuspidal automorphic representation of $G(\A)$ and let $\Pi$ be a weak functorial transfer of $\pi$.
In general, $\Pi$ is not uniquely determined by $\pi$. However, one may specify $\Pi$ uniquely by requiring it to be an isobaric sum
$\Pi_1\boxplus\ldots\boxplus\Pi_l$ of (not necessarily unitary) cuspidal automorphic representations $\Pi_i$ of $\GL_{n_i}(\A)$ with
$n_1+\ldots+n_l=N$. By the classification of Jacquet and Shalika \cite[Theorem~4.4]{JS2}, the representations $\Pi_i$ are uniquely determined by $\pi$, up to permutation.

As opposed to the case where $\pi$ is globally generic, the isobaric transfer is not necessarily compatible
with the local transfer predicted by the local Langlands conjectures, even when the latter are known.
This phenomenon, first observed by Langlands \cite{LanglandsEin1979}, was the point of departure of Arthur's notion of $A$-packets (\cite{Arthur1984}).

Theorem~\ref{theo:globl functorial lift} is proved using the Converse Theorem (see below).
An artifact of the proof is that with our method $\Pi_{\infty}$ is the transfer of $\pi_{\infty}$
via the local Langlands correspondence for real groups (\cite{La3}). On the other hand, this means that $\Pi$
obtained with our method is not necessarily the isobaric one.

In the local setup we define a ``coarse transfer" from irreducible representations of $G(F_{\nu})$ to irreducible representations of $\GL_N(F_{\nu})$, in terms of twisted $\gamma$-factors (see \S~\ref{coarse transfer}). In Theorem~\ref{theorem:weak transfer fixes gamma} we observe that if $\Pi$ is a weak functorial transfer of $\pi$, then $\Pi_{\nu}$ is a coarse transfer of $\pi_{\nu}$ for all places.
In particular, for any weak transfer $\Pi$ of $\pi$, the supercuspidal support of $\Pi_{\nu}$ is determined by $\pi_{\nu}$ for all finite places $\nu$.
This is in accordance with a result of M{\oe}glin \cite[Corollaire~4.2]{Moeglin2009} stating that the parameters of two local
$A$-packets that intersect have the same supercuspidal support. Another consequence of Theorem~\ref{theo:globl functorial lift} is that a coarse transfer always exists (Corollary~\ref{corollary:local converse thm}).

As noted above, for symplectic and quasi-split orthogonal groups, the weak functorial transfer of automorphic representations follows directly from \cite[\S~1.5]{Arthur2013}. In fact, as explained by Shin in \cite{Shin2023} as far as the weak transfer is concerned one can simplify the argument of
\cite{Arthur2013} and extend it to cover all (possibly non-quasi-split) classical groups.
To do so one still needs the stabilization of the trace formula and the twisted one
and, consequently, some variants of the fundamental lemma, including ones which are not yet available in the literature.

Our proof is arguably simpler, at least as far as its prerequisites are concerned.
On the downside, our results and techniques do not provide information about the image of the weak functorial transfer
or the structure of $A$-packets, nor do they address the trace identities for local $A$-packets.

We note that Theorem~\ref{theo:globl functorial lift} for general spin groups already implies the same result for special orthogonal groups, because representations of special orthogonal groups
are simply representations of general spin groups with a trivial central character.

We expect our techniques to be applicable to quasi-split or even non-quasi-split classical groups as well, see below.

In the important special case of globally generic cuspidal automorphic representations of quasi-split classical groups, i.e.,
those representations admitting a nonzero Whittaker--Fourier coefficient, Cogdell \textit{et al.} \cite{CKPS2,CKPS,CPSS} proved the existence of a (strong) functorial transfer to $\GL_N(\A)$. (In the globally generic case the isobaric transfer is compatible with the local transfer predicted by the local Langlands correspondence.)
In their proof they used the Converse Theorem of Cogdell and Piatetski-Shapiro \cite{CPS3,CPS1999} (more precisely, the version in \cite[\S~2]{CKPS2}), in order to reduce the existence of a transfer to proving the
analytic properties of $L$-functions for pairs of cuspidal automorphic representations of $G(\A)$ and $\GL_k(\A)$. The pertinent $L$-functions in those works were studied using the Whittaker--Fourier coefficients of Eisenstein series, that is, using the so-called Langlands--Shahidi method, developed in \cite{La2,La5,Sh2,Sh4,Shahidi1985,Sh3}; see also
Shahidi \cite{ShahidiBook2010} and the references therein.
This method is limited to generic representations and accordingly, their results were limited to globally generic representations of $G(\A)$.
Asgari and Shahidi \cite{AsgSha,AsgSha2} extended functoriality in the generic case to quasi-split general spin groups.

As in those works, we use the Converse Theorem to reduce the proof of Theorem~\ref{theo:globl functorial lift} to the study of $L$-functions.
But in our case we do not assume that the representation $\pi$ of $G(\A)$ is generic. To use the Converse Theorem, we analyze the $L$-functions via the generalized doubling method developed in \cite{CFGK2,CFK2022,CFKmodels,DimaKaplan}, which
is applicable to all cuspidal automorphic representations of $G(\A)$.

For a discussion of the Converse Theorem and how we apply it see \S~\ref{The converse theorem}. In a nutshell,
given an irreducible cuspidal (generic or otherwise) automorphic representation $\pi=\otimes'_{\nu}\pi_{\nu}$ of $G(\A)$, one may define an
irreducible admissible representation $\Pi=\otimes'_\nu \Pi_\nu$ of $\GL_N(\A)$, where for almost all places $\nu$,
$\Pi_{\nu}$ is the local transfer of $\pi_{\nu}$. The main problem is to prove that there is an automorphic representation $\Pi'$ of $\GL_N(\A)$ which agrees with $\Pi$ at almost  all places.
The Converse Theorem provides criteria for the existence of $\Pi'$ in terms of certain twisted $L$-functions $L(s,\Pi\times\tau)$, where
$\tau$ varies over a class of cuspidal automorphic representations of $\GL_k(\A)$, $1\leq k < N$.

In order to verify these criteria we define a completed $L$-function $L(s,\pi\times\tau)$ such that $L(s,\pi\times\tau)=L(s,\Pi\times\tau)$.
This was accomplished in \cite{CFK2022} via the generalized doubling method, which constructs an integral representation for the (partial) $L$-function.
The remaining obstacle is to prove that $L(s,\pi\times\tau)$ is entire, under suitable conditions.
Using the Eulerian form of the global generalized doubling integral and the computation of the local integrals with unramified data, this task can be divided into two parts.

The first is to prove that, under certain local assumptions, each pole in $\Real(s)\geq1/2$ of the local $L$-factor $L(s,\pi_{\nu}\times\tau_{\nu})$ is accounted for by a pole of a local generalized doubling integral. This is our main local result, and it requires a new family of local integrals which are, in light of
\cite{me14,GinzburgSoudry2023}, interesting in their own right. See \S~\ref{Producing poles} and Appendix~\ref{appendix:GL2 poles}.

The second is to prove that the global generalized doubling integral is holomorphic in
$\Real(s)\geq1/2$. This follows by showing that the Eisenstein series which appears in this integral is holomorphic there,
under the condition that $\absdet^{it}\chi_{\pi}\tau\not\cong{\tau}^{\vee}$ for all $t\in\R$. Here $\chi_{\pi}$ is, essentially, the restriction of the central character of $\pi$ to the identity component of the center of $G(\A)$.
This Eisenstein series was studied by Jiang, Liu and Zhang \cite{JiangLiuZhang2013} (for classical groups) who described its poles in the self-dual case.
The method of \cite{JiangLiuZhang2013} was to use induction in order to reduce the question to one on the discrete spectrum, at which point they
employed Arthur's classification \cite{Arthur2013}. In order to prove that there are no poles in the non-self-dual case we use some basic
observations of Langlands on the discrete spectrum, and avoid using \cite{Arthur2013}. See \S~\ref{section:complete L functions for twists}.

In the generic case, Ginzburg, Rallis and Soudry (see e.g., \cite{Soudry6,RGS}) developed the descent method and used it to describe the image of
functoriality. Together with the aforementioned works of Cogdell \textit{et al.} \cite{CKPS2,CKPS,CPSS}, they provided a correspondence between the generic part of $\Pi_{\disc}(G)$ (which is cuspidal) and a certain subset $\tilde{\Psi}_{\temp}(G)$ of $\tilde{\Psi}_2(G)$. Moreover, this correspondence is compatible
with the local Langlands correspondence. Refer to \cite[Ch.~11]{RGS} for precise statements.
We also mention that based on the generalized doubling method, Ginzburg and Soudry \cite{SoudryGinzburg} developed the ``double descent" which
provides an explicit construction of $\Pi_{\autparam}(G)$ independently of Arthur's work, at least for cuspidal $\autparam\in\tilde{\Psi}_2(G)$.

One earlier result on functoriality for non-generic representations was obtained by
Pitale, Saha and Schmidt  \cite{PitaleSahaSchmidt2014}. In that work the authors used the integral representation of Furusawa \cite{Furusawa1993} and the Converse Theorem of Cogdell and Piatetski-Shapiro \cite{CogdellPS1996}, to produce a strong functorial transfer of certain non-generic cuspidal automorphic representations of $\GSp_4(\A)$ to cuspidal automorphic representations of $\GL_4(\A)$ (and $\GL_5(\A)$).
However, it was previously unknown whether an $L$-functions approach could be used to study functoriality in the generality presented here.

The classical doubling method, by which we mean the case of $k=1$, was first introduced by Piatetski-Shapiro and Rallis \cite{PSR}, with the local theory obtained by Lapid and Rallis \cite{LR}. That construction was applicable to all classical groups, including the symplectic group, orthogonal groups, and unitary groups,
and was later extended to the metaplectic group (the double cover of the symplectic group) by Gan \cite{Gan} and to unitary groups of Hermitian or skew-Hermitian forms over quaternion algebras by Yamana \cite{Yamana} (see also \cite{Kakuhama2020}). We expect that the generalized doubling method, that is, the generalization to arbitrary $k$ described in \cite{CFGK2,CFK2022,CFKmodels,DimaKaplan} can be extended to these groups. Granted that, the techniques presented here should lead to additional cases of weak functorial transfers, for these groups.

The doubling method is not limited to the study of functoriality. Thus far, the case $k=1$ has had an important role in a wide range of problems, including the theta correspondence, e.g., \cite{KudlaRallis1994,HKS,GanIchino2014,Yamana}, and $p$-adic $L$-functions \cite{HarrisLiSkinner2006,EischenHarrisLiSkinner}. It would be interesting to see if similar applications can be found for $k>1$; for example, see the work of Ginzburg and Soudry \cite{GinzburgSoudry2023} towards a new regularized Siegel--Weil type formula.

Historically, the idea to combine a converse theorem with an integral representation in order to obtain a lift (of modular forms) goes back
to Hecke, Weil and Shimura. Beginning in the 1980s, the success of the theory of Rankin--Selberg integrals for general linear groups by Jacquet, Piatetski-Shapiro and Shalika \cite{JPSS2,JPSS3,JPSS}, and by Jacquet and Shalika \cite{JS2,JS1}, and the emerging results on converse theorems, inspired several works on Rankin--Selberg integrals for $G\times\GL_k$, by authors including Gelbart, Ginzburg, Piatetski-Shapiro, Rallis, and Soudry. See for example \cite{GPS,G,Soudry,Soudry3} (more recent works include \cite{BS,GJRS,JZ}). Eventually the results on functoriality for globally generic cuspidal automorphic representations of classical groups were obtained via the Langlands--Shahidi method, while integral representations have found other applications. In this work we show that the approach to transfer problems via integral representations combined with the Converse Theorem is also effective for non-generic automorphic representations.

\subsection*{Acknowledgments}
We are very happy to thank Jeffrey Adams, Mahdi Asgari, Laurent Clozel, Jim Cogdell, Gal Dor, Jan Frahm, David Ginzburg, Dmitry Gourevitch,
Joseph Hundley, Avner Kiro, Zemer Kosloff, Baiying Liu, Goran Mui{\'c}, Dipendra Prasad, Freydoon Shahidi, David Soudry, David Vogan, and Lei Zhang for numerous valuable and inspiring discussions.
The authors would like to thank Erez Lapid for extremely useful advice throughout this project, and in particular for his help with the proof of Theorem~\ref{theorem:series}.
Finally, we are grateful to the referees for their interest in this work and helpful remarks, which helped improve the presentation.

\tableofcontents

\section{The groups}\label{the groups}

Throughout, $F$ is a field of characteristic $0$. If $\mathcal{G}$ is a connected reductive linear algebraic group defined and split over $F$, we fix in $\mathcal{G}$ a
maximal torus $T_{\mathcal{G}}$ and a Borel subgroup $B_{\mathcal{G}}=T_{\mathcal{G}}\ltimes N_{\mathcal{G}}$ where $N_{\mathcal{G}}$ is the unipotent radical. The Weyl group of $T_{\mathcal{G}}$ in $\mathcal{G}$ is denoted $W_{\mathcal{G}}$, and $\WeylElement_{\mathcal{G}}$ is the longest Weyl element in $W_{\mathcal{G}}$.
For a standard parabolic subgroup $R<\mathcal{G}$ we denote by $M_R$ its unique Levi subgroup containing $T_{\mathcal{G}}$, then
$R=M_R\ltimes U_R$ where $U_R$ is the unipotent radical of $R$. The modulus character of $R$ is denoted $\delta_R$
and the unipotent subgroup opposite to $U_R$ is denoted $U_R^-$. Also set $R^-=M_R\ltimes U_R^-$, the parabolic subgroup opposite to $R$ with a Levi part $M_R$. Let $C_{\mathcal{G}}$ be the center of $\mathcal{G}$; similarly for any $X<\mathcal{G}$, let $C_{X}$ be the center of $X$. We denote the complex dual group of $\mathcal{G}$ by $\Dual{\mathcal{G}}$, e.g., $\Dual{\GL}_l=\GL_l(\C)$.

For the group $\GL_{l}$, $T_{\mathcal{\GL}_l}$ is taken to be the diagonal torus and $B_{\mathcal{\GL}_l}$ is the subgroup of upper triangular invertible matrices.
If $\beta$ is a composition of $l$, that is, an ordered tuple of positive integers whose sum is $l$, we let $R_{\beta}=M_{\beta}\ltimes U_{\beta}$ denote the standard parabolic subgroup of $\GL_l$ corresponding to $\beta$. Also let $\Mat_{a\times b}$ be the space of $a\times b$ matrices, $\Mat_{a}=\Mat_{a\times a}$; the transpose of $x\in\Mat_{a\times b}$ is denoted ${}^tx$, and when $a=b$, $\tr$ is the trace map. For $g\in\GL_l(F)$, put $g^*=J_l{}^tg^{-1}J_l$ where $J_l$ is the permutation matrix whose $(i,l-i+1)$-th coordinates are $1$. A block-diagonal matrix in $\GL_l$ will be denoted $\diag(a_1,\ldots,a_d)$ where $a_i\in\GL_{k_i}$, $k_1+\ldots+k_d=l$.

We take $\mathcal{G}_l$ to be one of the groups: $\Sp_{l}$, $\SO_l$ or $\GSpin_l$, defined and split over $F$, where for $\Sp_l$, $l$ must be even.
First assume that $\mathcal{G}_l$ is either $\Sp_{l}$ or $\SO_l$.
For definiteness, we take $J_{\Sp_l}=J_l\diag(-I_{l/2},I_{l/2})$ and $J_{\SO_l}=J_l$, and define
$\mathcal{G}_l=\{g\in\SL_{l}:{}^tgJ_{\mathcal{G}_l}g=J_{\mathcal{G}_l}\}$.
We then fix $T_{\mathcal{G}_l}=\mathcal{G}_l\cap T_{\GL_l}$ and $B_{\mathcal{G}_l}=\mathcal{G}_l\cap B_{\GL_l}$. Note that $\Sp_{0}$, $\SO_{0}$ and $\SO_{1}$ are trivial.

For $\mathcal{G}_l=\GSpin_{l}$  we follow \cite{AsgSha,HS} and define it using based root datum. Let $n=\lfloor l/2\rfloor$,
$X=\bigoplus_{i=0}^n\Z e_i$, $X^{\vee}=\bigoplus_{i=0}^n\Z e_i^{\vee}$ and fix the standard $\Z$-pairing $\langle,\rangle$ on $X\times X^{\vee}$. Assume $l>2$. If $l$ is even define
\begin{align*}
&\Delta=\{e_1-e_2,\ldots,e_{n-1}-e_n,\,e_{n-1}+e_n\},\\
&\Delta^{\vee}=\{e_1^{\vee}-e_2^{\vee},\ldots,e_{n-1}^{\vee}-e_n^{\vee},\,e_{n-1}^{\vee}+e_n^{\vee}-e_0^{\vee}\}.
\end{align*}
If $l$ is odd define
\begin{align*}
&\Delta=\{e_1-e_2,\ldots,e_{n-1}-e_n,\,e_n\},\\
&\Delta^{\vee}=\{e_1^{\vee}-e_2^{\vee},\ldots,e_{n-1}^{\vee}-e_n^{\vee},\,2e_n^{\vee}-e_0^{\vee}\}.
\end{align*}
Otherwise $l\leq2$ and the sets $\Delta$ and $\Delta^{\vee}$ are taken to be empty.
The group $\mathcal{G}_l=\GSpin_{l}$ is the unique $F$-split $F$-group whose based root datum is $(X,\Delta,X^{\vee},\Delta^{\vee})$ (see \cite{AsgSha,HS} and, e.g., \cite[Ch.~16]{Springer1998}). This definition already fixes $T_{\mathcal{G}_l}$ and $B_{\mathcal{G}_l}$. Note that $\GSpin_0=\GSpin_1=\GL_1$ and $\GSpin_2=\GL_1\times\GL_1$. Denote $C_{\mathcal{G}_l}^{\circforgspin}=e_0^{\vee}(\GL_1)$, it is a subgroup of $C_{\mathcal{G}_l}$. When $l>2$, $C_{\mathcal{G}_l}^{\circforgspin}$ is in fact the identity component of $C_{\mathcal{G}_l}$, see e.g., \cite[Proposition~2.3(a)]{AsgSha}. Note also that $C_{\mathcal{G}_l}^{\circforgspin}$ is the kernel of a projection $\mathcal{G}_l\rightarrow\SO_l$ (see \cite[\S~4.3]{HS}). Dually, $e_0$ is the similitude character of $\Dual{\mathcal{G}}_l$.

For uniformity of notation, we take $C_{\mathcal{G}_l}^{\circforgspin}$ to be the trivial group when $\mathcal{G}_l\ne\GSpin_l$.

For the groups $\mathcal{G}_{2l}=\SO_{2l}$ and $\GSpin_{2l}$, when $B_{\mathcal{G}_{2l}}$ is fixed, we let $\Specialjmath$ denote the outer involution corresponding to the permutation of the last two simple roots in the Dynkin diagram. For uniformity, if $\mathcal{G}_{2l}$ is not one of these groups we take $\Specialjmath$ to be trivial.
There is a unique standard Siegel parabolic subgroup in $\mathcal{G}_l$ except for $\mathcal{G}_{2l}=\SO_{2l}$ and $\GSpin_{2l}$, for $l>1$, in which case there are $2$ such subgroups and if $P$ denotes one, ${}^{\Specialjmath}P$ is the other.

We list the groups $\Dual{\mathcal{G}}_l$:
\begin{align*} \renewcommand*{\arraystretch}{1.4}
  \begin{array}{|c||cccccc|}
    \hline
    \mathcal{G}_l&\Sp_{2n}&\SO_{2n}&\SO_{2n+1}&\GSpin_{2n}&\GSpin_{2n+1}&\\ \hline\
    \Dual{\mathcal{G}}_l&\SO_{2n+1}(\C)&\SO_{2n}(\C)&\Sp_{2n}(\C)&\GSO_{2n}(\C)&\GSp_{2n}(\C)&\\ \hline
  \end{array}
\end{align*}

Let $M$ be a standard Levi subgroup of $\mathcal{G}_{l}$. Then $M\cong M_{\beta}\times\mathcal{G}_{l-2r}$ where $\beta=(\beta_1,\ldots,\beta_d)$ is a composition of $r\geq0$ and $2r\leq l$. We fix an isomorphism $i_M:M\to M_{\beta}\times \mathcal{G}_{l-2r}$ as follows. First, when $\mathcal{G}_l\ne\GSpin_l$,
for $a_i\in\GL_{\beta_i}$, $1\leq i\leq d$, and $g\in\mathcal{G}_{l-2r}$,
\begin{align*}
i_M(\diag(a_1,\ldots,a_d,g,a_d^*,\ldots,a_1^*))=(a_1,\ldots,a_d,g).
\end{align*}

For $\mathcal{G}_l=\GSpin_l$ and $l>2$, starting with $d=1$, the root datum of $M$ can be written as a direct sum of the root data for $\GL_r$ and for $\mathcal{G}_{l-2r}$, and we note that the character and co-character lattices of $\mathcal{G}_{l-2r}$ are identified with $\Span_{\Z}\{e_0,e_{r+1},\ldots,e_n\}$
and $\Span_{\Z}\{e_0^{\vee},e^{\vee}_{r+1},\ldots,e_n^{\vee}\}$, resp., where for $r=n$ these lattices become $\Z e_0$ and $\Z e_0^{\vee}$. This fixes $i_M$ for $d=1$. The case $d>1$ is now described using the standard embedding of $M_{\beta}$ in $\GL_r$. We note that in the dual picture, when $l$ is even
$\Dual{\mathcal{G}}_l=\GSO_{l}(\C)$ is the identity component of $\GO_l(\C)=\{g\in\GL_l(\C):{}^tgJ_{\SO_l}g=e_0(g)J_{\mathcal{\SO}_l}\}$ and when $2\nmid l$,
$\GSp_{l-1}(\C)=\{g\in\GL_{l-1}(\C):{}^tgJ_{\Sp_{l-1}}g=e_0(g)J_{\Sp_{l-1}}\}$. We then have the isomorphism $\Dual{i}_M:\widehat{M}\to\widehat{M}_{\beta}\times \widehat{\mathcal{G}}_{l-2r}$: for $a_i\in\Dual{\GL}_{\beta_i}$ and $g'\in \Dual{\mathcal{G}}_{l-2r}$,
\begin{align*}
\Dual{i}_M(\diag(a_1,\ldots,a_d,g',e_0(g')a_d^*,\ldots,e_0(g')a_1^*))=(a_1,\ldots,a_d,g').
\end{align*}

In addition, let $\Upsilon_l$ be the character of $\GSpin_l$ normalized by $\langle \Upsilon_l,e_0^{\vee}\rangle=2$. The kernel of $\Upsilon_l$ is $\Spin_l$. We note that
\begin{align*}
\Upsilon_l\circ i_M^{-1}(a_1,\ldots,a_d,g)=\Upsilon_{l-2k}(g)\prod_{i=1}^d\det(a_i).
\end{align*}
When $\mathcal{G}_l\ne\GSpin_l$ we take $\Upsilon_l$ to be trivial.

Let $G=\mathcal{G}_c$ for an integer $c>1$. Let ${}^LG=\Dual{G}\rtimes \Gamma_F$ be the Langlands dual group of $G$, where $\Gamma_F$ is the absolute Galois group of $F$.
Since $G$ is split over $F$, the action of $\Gamma_F$ on $\Dual{G}$ is trivial. Thus we may work with the simpler group $\Dual{G}$. We let $\faithful:\Dual{G}\to\Dual{\GL}_N$ be the natural embedding, where $N=c+1$ if $G=\Sp_c$ and otherwise $N=2\lfloor c/2\rfloor$.

\section{The Converse Theorem}\label{The converse theorem}

In order to obtain the weak functorial transfer we will use the Converse Theorem for $\GL_N$, as in \cite{CKPS2,CKPS,AsgSha,CPSS,AsgSha2}. The roots of this method go back to
Hecke, who formulated a criterion for a Dirichlet series $D(s)$ to arise from a full level cusp form (see \cite{CPS3} and the references therein).
The conditions, now familiarly called ``niceness properties", were stated in terms of the associated completed $L$-function.

For automorphic representations, the case of $\GL_2$ was obtained by Jacquet and Langlands \cite{JL}. The general case of $\GL_N$ was proved by Cogdell and Piatetski-Shapiro \cite{CPS3,CPS1999}, and we will use the following minor modification of their theorem from \cite{CKPS2}. This version incorporated a twist of the representations of $\GL_k(\A)$
by a fixed Hecke character $\eta$ (see below). While incorporating this character did not complicate the proof of the Converse Theorem, it did play a
key role in the application of the theorem in \textit{ibid.}~ and in the follow-up works mentioned above. Namely, the twists by $\eta$ were important for the stability results and in order to prove that the $L$-functions were entire. We will use $\eta$ in a similar way, as we explain below, and also in order to simplify the analysis of the poles of the local $L$-factors.

Let $F$ be a number field with its ring of adeles $\A$.
Let $\eta$ be a unitary Hecke character of $\A^*$ and fix a nonempty finite set $S$ of finite places of $F$. For each $k\geq1$, let $\mathscr{A}'(S,k)$ denote the set
of irreducible unitary cuspidal automorphic representations $\tau'$ of $\GL_k(\A)$ such that for all $\nu\in S$, the local component $\tau'_{\nu}$ of $\tau'$ at $\nu$ is unramified. For
any automorphic representation $\sigma$ of $\GL_k(\A)$, set $\eta\sigma(g)=\eta(\det g)\sigma(g)$. Define
\begin{align*}
\mathscr{A}(S,k,\eta)=\{\eta\tau':\tau'\in \mathscr{A}'(S,k)\},\qquad
\mathscr{A}(S,\eta)=\coprod_{k=1}^{N-1}\mathscr{A}(S,k,\eta).
\end{align*}
(While the Converse Theorem of \cite[\S~2]{CKPS2} was stated without restricting $\mathscr{A}(S,\eta)$ to unitary representations,
it is clearly sufficient to consider only the unitary ones.)

Let $\psi=\otimes'_{\nu}\psi_{\nu}$ be a nontrivial additive character of $\A$ which is trivial on $F$.
Assume we have, for each place $\nu$ of $F$, an irreducible admissible representation $\Pi_{\nu}$ of $\GL_N(F_{\nu})$, such that for almost all $\nu$,
$\Pi_{\nu}$ is unramified. For any $\tau\in\mathscr{A}(S,\eta)$,
the local factors $L(s,\Pi_{\nu}\times\tau_{\nu})$ and $\epsilon(s,\Pi_{\nu}\times\tau_{\nu},\psi_{\nu})$ are defined by \cite{JPSS,JS3} (whether $\Pi_{\nu}$ is generic or not).
We can then define the formal Euler products $L(s,\Pi\times\tau)=\prod_{\nu}L(s,\Pi_{\nu}\times\tau_{\nu})$ and
$\epsilon(s,\Pi\times\tau,\psi)=\prod_{\nu}\epsilon(s,\Pi_{\nu}\times\tau_{\nu},\psi_{\nu})$. In particular, $L(s,\Pi)$ is defined formally as an Euler product, but if we assume
that this Euler product is absolutely convergent in some right half plane, then by \cite[Lemma~2.2]{CPS3} so are the Euler products $L(s,\Pi\times\tau)$. Moreover, if the central character of $\Pi$ is trivial on $F^*$, by \cite[Lemma~2.1]{CPS3} $\epsilon(s,\Pi\times\tau,\psi)$ is independent of $\psi$ and we can denote
$\epsilon(s,\Pi\times\tau)=\epsilon(s,\Pi\times\tau,\psi)$.

\begin{theorem}\cite[\S~2]{CKPS2}\label{theorem:cnv}
Assume $\Pi=\otimes'_{\nu}\Pi_{\nu}$ is an irreducible admissible representation of $\GL_N(\A)$, whose central character is invariant under $F^*$, and such that
$L(s,\Pi)$ is absolutely convergent in $\Real(s)\gg0$. Furthermore assume that for each $\tau\in\mathscr{A}(S,\eta)$, the $L$-function
$L(s,\Pi\times\tau)$ is nice, that is,
\begin{enumerate}[leftmargin=*]
  \item\label{It:analytic cont} $L(s,\Pi\times\tau)$ and $L(s,\Pi^{\vee}\times\tau^{\vee})$ admit analytic continuation to $\C$.
  \item\label{It:BVS} $L(s,\Pi\times\tau)$ and $L(s,\Pi^{\vee}\times\tau^{\vee})$ are bounded in vertical strips of finite width.
  \item\label{It:functional eq} $L(s,\Pi\times\tau)=\epsilon(s,\Pi\times\tau)L(1-s,\Pi^{\vee}\times\tau^{\vee})$.
\end{enumerate}
Then there is an irreducible automorphic representation $\Pi'$ of $\GL_N(\A)$ such that $\Pi'_{\nu}\cong\Pi_\nu$ for all $\nu\not\in S$.
\end{theorem}

We explain how to apply this theorem to obtain a weak functorial transfer from the group $G$ to $\GL_N$, following the paradigm of \cite{CKPS2}
(for notation see \S~\ref{the groups}).
\begin{definition}\label{Def:weak transfer}
Let $\pi$ be an irreducible automorphic representation of $G(\A)$ and let $\Pi$ be an
irreducible automorphic representation of $\GL_N(\A)$. We say that $\Pi$ is a weak functorial transfer of $\pi$ if for almost all finite places $\nu$ of $F$ where $\pi_{\nu}$ is unramified, $\Pi_{\nu}$ is the local functorial transfer of $\pi_{\nu}$ dictated by $\faithful$, and at the infinite places $\nu$, the infinitesimal
character of $\Pi_{\nu}$ is determined, via $\faithful$, by that of $\pi_{\nu}$\footnote{The infinitesimal character of $\pi_{\nu}$ is encoded into the $L$-parameter of $\pi_{\nu}$, see Nair and Prasad \cite[Lemma~1]{NairPrasad2021}.}. We then say that $\pi$ has a weak functorial transfer to $\GL_N(\A)$.
\end{definition}
Let $\pi=\otimes'_{\nu}\pi_{\nu}$ be an irreducible unitary cuspidal automorphic representation of $G(\A)$. Our aim is to construct a weak functorial transfer
of $\pi$. Let $S_{\pi}$ be the (finite) set of finite places $\nu$ such that $\pi_{\nu}$ is not unramified; however
if $\pi$ is unramified at all finite places, then we select an arbitrary finite place to include in $S_\pi$ so that it is not empty.
We can certainly define $\Pi_{\nu}$ at all finite places $\nu\notin S_{\pi}$, and at the infinite places we take $\Pi_{\nu}$ to be the archimedean transfer (see
\S~\ref{local transfer results}). For each $\nu\in S_{\pi}$, we take an arbitrary irreducible $\Pi_{\nu}$, with a certain condition on the central character
if $G=\GSpin_c$ (see Theorem~\ref{theorem:pi and Pi for ramified twisted} below). Define $\Pi=\otimes'_{\nu}\Pi_{\nu}$.
Our goal is to show the existence of an automorphic representation $\Pi'$ of $\GL_N(\A)$ such that $\Pi'_{\nu}\cong\Pi_\nu$ for all $\nu\not\in S_{\pi}$. To this end we will
choose $\eta$ which is sufficiently ramified at the places of $S_{\pi}$.
Now we need to prove the conditions of the theorem for all $\tau\in\mathscr{A}(S_{\pi},\eta)$, and the existence of $\Pi'$ will follow.

The generalized doubling method provides a framework for the $L$-functions $L(s,\pi\times\tau)$ via an integral representation. It also gives rise to local factors. The definition is compatible with the local transfer outside $S_{\pi}$, which means that (for any $\eta$, even the trivial)
\begin{align*}
L(s,\pi_{\nu}\times\tau_{\nu})=L(s,\Pi_{\nu}\times\tau_{\nu}),\qquad
\epsilon(s,\pi_{\nu}\times\tau_{\nu},\psi_{\nu})=\epsilon(s,\Pi_{\nu}\times\tau_{\nu},\psi_{\nu}),\qquad\forall\nu\not\in S_{\pi}.
\end{align*}
Thus, at least $L^{S_{\pi}}(s,\pi\times\tau)=L^{S_{\pi}}(s,\Pi\times\tau)$. This is almost enough, except we need to understand the completed $L$-functions.
Here stability kicks in, and the twists by $\eta$ imply that
\begin{align*}
L(s,\pi_{\nu}\times\tau_{\nu})=L(s,\Pi_{\nu}\times\tau_{\nu})=1,\qquad
\epsilon(s,\pi_{\nu}\times\tau_{\nu},\psi_{\nu})=\epsilon(s,\Pi_{\nu}\times\tau_{\nu},\psi_{\nu}),\qquad\forall\nu\in S_{\pi}.
\end{align*}
Thus $L(s,\pi\times\tau)=L(s,\Pi\times\tau)$ and $\epsilon(s,\pi\times\tau)=\epsilon(s,\Pi\times\tau)$, and it remains to show that $L(s,\pi\times\tau)$ is nice.

Building on the results of \cite{CFK2022},
it remains to prove that $L(s,\pi\times\tau)$ is entire and in fact, since by \textit{ibid.}~ we already know the global functional equation,
it is enough to show that $L(s,\pi\times\tau)$ is holomorphic in $\Real(s)\geq1/2$.

Our starting point is the Eulerian form of the global integral, which can be written as
\begin{align}\label{eq:func intro}
b^{S}(s,c,\tau\otimes\chi_{\pi})Z(s,\varphi_1,\varphi_2,f)=
L(s,\pi\times\tau)\prod_{\nu\in S}\frac{Z(s,\omega_{\nu},f_{\nu})}{L(s,\pi_{\nu}\times\tau_{\nu})}.
\end{align}
Here $S$ is a finite set of places, containing $S_{\pi}$ and $S_{\infty}$, outside of which all data are unramified; on the l.h.s.~ (left hand side) the factor $b^S$ is a product of partial $L$-functions depending only on $\tau$ and $\chi_{\pi}$, and $Z(\cdots)$ is the global integral; and on the r.h.s.~ $Z(\cdots)$ is the local integral. See \S~\ref{The local integral}, \S~\ref{the local factors}, \S~\ref{global groups and notation} and \S~\ref{The gbl integral} for notation and precise definitions.
The $L$-functions appearing in $b^S$ are holomorphic in $\Real(s)\geq1/2$ for $\tau\in\mathscr{A}(S,\eta)$, because of the twist by $\eta$.

The local result, Theorem~\ref{theorem:archimedean producing poles}, states that for each $s$ with $\Real(s)\geq1/2$, there are data
$(\omega_{\nu},f_{\nu})$ such that
$Z(s,\omega_{\nu},f_{\nu})/L(s,\pi_{\nu}\times\tau_{\nu})$ is nonzero. Note that
the local $L$-factors are trivial for $\nu\in S_{\pi}$, which means that the challenging case of this result is for $\nu\in S-S_{\pi}$.

The global result, Theorem~\ref{theorem:series}, states that the global integral is holomorphic in
$\Real(s)\geq1/2$. This follows from a general argument on Eisenstein series, because
$\absdet^{it}\chi_{\pi}\tau\not\cong{\tau}^{\vee}$ for all $t\in\R$,
a condition which in our setting holds because $\tau\in\mathscr{A}(S_{\pi},\eta)$ ($\chi_{\pi}$ is defined in \S~\ref{global groups and notation}).

Now for each $s$ such that $\Real(s)\geq1/2$, since the l.h.s.~ of \eqref{eq:func intro} is holomorphic and the finite product of quotients can be made nonzero, we deduce that $L(s,\pi\times\tau)$ is holomorphic in $\Real(s)\geq1/2$ and thereby entire (see Theorem~\ref{theorem:twisting to obtain entire L function}).

For the complete details of our application of the Converse Theorem see \S~\ref{section:Constructing the lift}.

\section{Local theory}\label{Local theory}
We start with the local theory.
In \S~\ref{The local integral} we define the generalized doubling integral and recall its basic properties.
In \S~\ref{the local factors} we describe the functional equation of the integral, and the corresponding $\gamma$-factor.
The main properties of the $\gamma$-factor from \cite{CFK2022} are summarized in Theorem~\ref{theorem:ten commendments}, then we recall the
definitions of the $L$- and $\epsilon$-factors. The ``coarse transfer" is defined in \S~\ref{coarse transfer}.
In \S~\ref{Producing poles} we prove our main local result, Theorem~\ref{theorem:archimedean producing poles}, that
(under certain assumptions) the poles of the $L$-factor in $\Real(s)\geq1/2$ can be obtained using the integral.

\subsection{Notation}\label{local groups and notation}
In this section $F$ is a local field. By convention, for a linear algebraic group $\mathcal{G}$ defined and split over $F$, we denote its $F$-points also by $\mathcal{G}$.
If $F$ is non-archimedean, $\mathcal{O}$ will be its ring of integers, $q$ will denote the
cardinality of its residue field, and an entire (resp., meromorphic) function $\phi(s):\C\rightarrow\C$ will always be an element of
$\C[q^{-s},q^s]$ (resp., $\C(q^{-s})$). For any group $X$, $x,y\in X$ and $Y<X$, set ${}^xy=xyx^{-1}$ and ${}^xY=\{{}^xy:y\in Y\}$.

All representations under consideration here are complex and smooth.
We use the smooth and normalized induction functor.
Occasionally we will use the standard notation $\rtimes$ for the parabolic induction functor from a standard maximal parabolic subgroup
of any group $\mathcal{G}_l$. In detail, if $R<\mathcal{G}_l$ is such a subgroup, then in \S~\ref{the groups} we defined $i_{M_R}:M_R\to\GL_r\times \mathcal{G}_{l-2r}$ ($r\geq0$, $2r\leq l$) and for a representation $\sigma\otimes\pi$ of $\GL_r\times\mathcal{G}_{l-2r}$ we denote $\sigma\rtimes\pi=\Ind_R^{\mathcal{G}_l}((\sigma\otimes\pi)\circ i_{M_R})$. We also use the standard notation $\times$ for the induction functor from any standard parabolic subgroup of $\GL_l$.

An admissible representation over an archimedean field is understood to be admissible Fr\'{e}chet of moderate growth. If $F$ is archimedean, by a supercuspidal representation we always mean a quasi-character of $F^*$. For an admissible representation $\pi$ of $\mathcal{G}$, $\pi^{\vee}$ denotes its contragredient representation. In addition, the unramified setting will only be considered when the field is non-archimedean, and we say that a representation $\pi$ of $\mathcal{G}$ is unramified if it has a nonzero vector fixed by $\mathcal{G}(\mathcal{O})$.

Let $\Irr(\mathcal{G})$ be the set of irreducible admissible representations of $\mathcal{G}$, and when $\mathcal{G}=\GL_l$ we denote
by $\IrrGen(\GL_l)\subset \Irr(\GL_l)$ the subset of generic ones, i.e., those representations admitting a Whittaker model.
Supercuspidal or tempered representations are assumed to be irreducible and admissible.

For an admissible representation $\pi$ of $\GSpin_l$ which admits a central character, we let $\chi_{\pi}$ be the pullback of $\pi$ by $e_0^{\vee}$
viewed as a quasi-character of $F^*$ (which is essentially $\pi|_{C_{\mathcal{G}_{l}}^{\circforgspin}}$).
Using the above notation, if $(\sigma\otimes\pi)\circ i_{M_R}\in\Irr(M_R)$, then $\chi_{\sigma\rtimes\pi}=\chi_{\pi}$. If $\mathcal{G}_l\ne\GSpin_l$ we simply take, for any
admissible representation $\pi$ of $\mathcal{G}_l$, $\chi_{\pi}^d$ to be trivial for any $d\in\tfrac12\Z$.

For a representation $\tau$ of $\GL_k$ and a quasi-character $\eta$ of $F^*$, denote by $\eta\tau$ the twist of $\tau$ by $\eta$, namely the representation of $\GL_k$ on the same space of $\tau$ defined by $\eta\tau(g)=\eta(\det g)\tau(g)$. E.g., $\absdet\tau(g)=|\det{g}|\tau(g)$.

Let $\psi$ be a nontrivial additive character of $F$.

Let $c\geq1$. For $l\geq1$, a generic character of $U_{(c^l)}$ is a character whose stabilizer in $M_{(c^l)}$ is isomorphic to $\GL_c$.
We fix the following generic character $\psi_l$ of $U_{(c^l)}$. Write $u\in U_{(c^l)}$ in the form $u=(u_{i,j})_{1\leq i,j\leq l}$ with $u_{i,j}\in\Mat_c$.
Then $\psi_{l}(u)=\psi(\sum_{i=1}^{l-1}\tr(u_{i,i+1}))$.

Let $k\geq1$ and $\tau\in\IrrGen(\GL_k)$. We recall the definitions of the representations $\rho_c(\tau)$ and their $(k,c)$ models from \cite{CFKmodels,DimaKaplan}.
If $\tau$ is unitary, let $\rho_c(\tau)$ be the generalized Speh representation of $\GL_{kc}$, i.e., the unique irreducible quotient of
$\Ind_{R_{(k^c)}}^{\GL_{kc}}((\tau\otimes\ldots\otimes\tau)\delta_{R_{(k^c)}}^{1/(2k)})$ (\cite{Jac4}). In general if
$\tau=\times_{i=1}^d\absdet^{a_i}\tau_i$ where $\tau_i$ are tempered and $a_1>\ldots>a_d$, define
\begin{align*}
\rho_c(\tau)=\times_{i=1}^d\absdet^{a_i}\rho_c(\tau_i).
\end{align*}
Then by \cite[Theorem~4]{CFKmodels} and \cite[\S~1.4]{DimaKaplan}, $(k^c)$ is the unique maximal orbit in the wave-front set of $\rho_c(\tau)$ and
\begin{align*}
\dim\Hom_{U_{(c^k)}}(\rho_c(\tau),\psi_k)=1.
\end{align*}
In general, an admissible finite-length representation of $\GL_{kc}$ satisfying these properties is called a $(k,c)$ representation (see \cite[\S~1.4]{DimaKaplan}).
A nonzero $\lambda\in\Hom_{U_{(c^k)}}(\rho_c(\tau),\psi_k)$ is called a $(k,c)$ functional on $\rho_c(\tau)$.
The unique $(k,c)$ model of $\rho_c(\tau)$, $W_{\psi}(\rho_c(\tau))$, is the space of functions $W:\GL_{kc}\to\C$ of the form
\begin{align*}
W(g)=\lambda(\rho_c(\tau)(g)\xi),
\end{align*}
where $\xi$ varies in the space of $\rho_c(\tau)$, and $\lambda$ is a fixed $(k,c)$ functional on $\rho_c(\tau)$.
The space $W_{\psi}(\rho_c(\tau))$ is a representation of $\GL_{kc}$ with the action given by right-translations, and as such it is a quotient of $\rho_c(\tau)$. For example if $c=1$, then $\rho_c(\tau)=\tau$ and $W_{\psi}(\tau)$ is the usual $\psi_k$-Whittaker model of $\tau$.

\subsection{Functorial transfer}\label{local transfer results}
Recall the representation $\faithful:\Dual{G}\to\Dual{\GL}_N$ defined in \S~\ref{the groups}. Suppose that $F$ is non-archimedean and let $\pi\in\Irr(G)$ be unramified. By the Satake isomorphism, $\pi$ is parameterized by
a semi-simple conjugacy class $[t_{\pi}]$ in $\Dual{G}$, and the semi-simple conjugacy class $[\faithful(t_{\pi})]$ in $\Dual{\GL}_N$ determines an unramified $\Pi\in\Irr(\GL_N)$, which we call the unramified transfer of $\pi$ (see \cite[\S~16.2]{Bo} and \cite[\S~5.2]{CKPS}). When $F$ is archimedean, by Langlands \cite{La3}, any $\pi\in\Irr(G)$ is parameterized by an admissible homomorphism $\phi:W_F\to\Dual{G}$, where $W_F$ is the Weil group of $F$, and then $\faithful\circ\phi:W_F\to\Dual{\GL}_N$ parameterizes a representation $\Pi\in\Irr(\GL_N)$, which is the archimedean functorial transfer of $\pi$ (see \cite[\S~5.1]{CKPS}).
For brevity, in both cases we denote
\begin{align}\label{eq:t pi}
\transfer(\pi)=\Pi.
\end{align}

The central character of $\transfer(\pi)$ is $\chi_{\pi}^{N/2}$, which is trivial when $G\ne\GSpin_c$, and makes sense for $\GSpin_c$ because
by the description in \S~\ref{the groups}, $\det\faithfulG{\GSpin_c}=e_0^{N/2}$.

In the non-archimedean case the statement on the central character follows from the concrete description of
$[\faithful(t_{\pi})]$ in \cite[\S~p.~189]{CKPS} and \cite[pp.~177--178]{AsgSha}, this description is also applicable to non-generic representations.
In the archimedean case the compatibility of the central character with the local Langlands correspondence was explained in \cite{La3}
for an arbitrary connected reductive group (see also \cite[pp.~178--179]{AsgSha} for $\GSpin_c$).

\subsection{The integral}\label{The local integral}
Let $c>1$ and $k\geq1$ be integers. Recall $G=\mathcal{G}_c$, and define $H=\mathcal{G}_{2kc}$.
Let $Q=M_Q\ltimes U_Q$ be the standard parabolic subgroup of $H$ with the Levi part $M_Q\cong\GL_c\times\ldots\times\GL_c\times\mathcal{G}_{2c}$ ($\GL_c$ occurs $k-1$ times). For brevity, put $U=U_Q$. For a character $\psi'$ of $U$, denote its stabilizer in $M_Q$ by $\mathrm{St}_{M_Q}(\psi')$.

If $k>1$, let $\psi_U$ be a character of $U$ which is generic with respect to the unipotent orbit $((2k-1)^c1^c)$ associated with $H$
(see \cite{G2} and \cite[\S~5]{CM} for these notions). This means that $\mathrm{St}_{M_Q}(\psi_U)(\C)$ and $G(\C)\times G(\C)$ are of the same Lie type. We can choose $\psi_U$ such that we have a map $\embedding:G\times G\rightarrow \mathrm{St}_{M_Q}(\psi_U)$, which is an embedding unless $G=\GSpin_c$ in which case its kernel is the diagonal embedding $(C_G^{\circforgspin})^{\triangle}$ of $C_G^{\circforgspin}$ in $G\times G$. If $k=1$, then $U$ is trivial and $\psi_U=1$, and we have a map $\embedding:G\times G\rightarrow M_Q$ with the same properties. Denote the identity element of $G$ by $e_G$. For brevity set $\embeddingL(g)=\embedding(g,e_G)$ and $\embeddingR(g)=\embedding(e_G,g)$, for $g\in G$.

Let $P=M_P\ltimes U_P<H$ be a standard Siegel parabolic subgroup. Using the isomorphism
$i_{M_P}:M_P\to\GL_{kc}\times\mathcal{G}_0$, we fix $T_{\GL_{kc}}$ and $B_{\GL_{kc}}$ to be the projections of $i_{M_P}(T_H\cap i_{M_P}^{-1}(\GL_{kc},1))$ and
$i_{M_P}(B_H\cap i_{M_P}^{-1}(\GL_{kc},1))$, resp., into $\GL_{kc}$.
Fix a maximal compact subgroup $K_G$ of $G$ and a similar subgroup $K_H$ of $H$. We can assume that $\embedding(K_G,K_G)<K_H$ and in addition $H=B_HK_H$.

Denote $P'={}^{\Specialjmath^{kc}}P$. Take a representative $\delta_0\in H$ for $\WeylElement\in W_H$ such that ${}^{\delta_0}U_{P'}=U_{P}^-$.
Let $\delta_1\in i_{M_Q}^{-1}(I_c,\ldots,I_c,\mathcal{G}_{2c})\cap U_{P'}$ be of maximal rank as a matrix in $\Mat_c$ (the rank is $2\lfloor c/2\rfloor$). Put $\delta=\delta_0\delta_1$. There is a subgroup $\Umacro<U$ such that $i_{M_P}({}^{\delta}\Umacro)=(U_{(c^k)},1)$ and $\psi_U$ restricts to a generic character of $\Umacro(\cong U_{(c^k)})$. For $k>1$, once $\delta_0$ is fixed, $\delta_1$ is uniquely specified by this character.
Define $U_0=U\cap U_{P'}$. Then $U=\Umacro\rtimes U_0$. The character $\psi_U$ is a character of $U_0$ by restriction. Also let $\iota$ be an involution of $G$ such that ${}^{\delta}\{\embedding(g,{}^{\iota}g):g\in G\}<M_P$ (such an involution exists).

When $k>1$, we can define $\psi_U$ and $\delta_0$ such that $\psi_U({}^{\delta^{-1}}u)=\psi_k^{-1}(u)$ for all $u\in U_{(c^k)}$ (thereby also fixing $\delta_1$); we assume this is the case. For concrete definitions of all objects above: the groups, the character $\psi_U$, the map $\embedding$, $\delta$ and $\iota$, see \cite{CFK2022,DimaKaplan}.
We give one example for the convenience of the reader.
\begin{example}\label{example:Sp}
Let $c=2n$, $G=\Sp_c$ and $k>1$. Then $H=\Sp_{2kc}$ and $\psi_U$ is given by
\begin{align*}
\psi_U(\left(\begin{smallmatrix}v&u&z\\&I_{2c}&u'\\&&v^*\end{smallmatrix}\right))=\psi_{k-1}(v)\psi(\tr(u^{1,1}+u^{2,2})),
\end{align*}
where $v\in U_{(c^{k-1})}$, $u^{1,1},u^{2,2}\in\Mat_n$ and $\left(\begin{smallmatrix}u^{1,1}&*&*\\ * & * &u^{2,2}\end{smallmatrix}\right)$ denotes the bottom $c$ rows of $u$. The map $\embedding:G\times G\rightarrow \mathrm{St}_{M_Q}(\psi_U)$ is the embedding given as follows: for $g_1=\left(\begin{smallmatrix}g_{1,1}&g_{1,2}\\g_{1,3}&g_{1,4}\end{smallmatrix}\right)\in G$, $g_{1,i}\in \Mat_{n}$, and $g_2\in G$,
\begin{align*}
\embedding(g_1,g_2)=
\diag(g_1,\ldots,g_1,\left(\begin{smallmatrix} g_{1,1}&&g_{1,2}\\ &g_2&\\ g_{1,3}&&g_{1,4}\end{smallmatrix}\right),g_1^*,\ldots,g_1^*),
\end{align*}
where $g_1^*$ appears $k-1$ times. Finally
\begin{align*}
\delta_0=\left(\begin{smallmatrix} &I_{kc}\\ -I_{kc}\end{smallmatrix}\right),\quad
\delta_1=\diag(I_{(k-1)c},\left(\begin{smallmatrix}I_c & I_c \\ &I_c\end{smallmatrix}\right),I_{(k-1)c}),\quad
{}^{\iota}g=\left(\begin{smallmatrix}&I_{n}\\I_{n}\end{smallmatrix}\right)g\left(\begin{smallmatrix}&I_{n}\\I_{n}\end{smallmatrix}\right)\quad(g\in G).
\end{align*}
\end{example}
Let $\pi\in\Irr(G)$. Recall that $\chi_{\pi}$ was defined in \S~\ref{local groups and notation}. Also let $\pi^{\iota}$ be the
representation on the space of $\pi$ given by $\pi^{\iota}(g)=\pi({}^{\iota}g)$.

Let $\tau\in\IrrGen(\GL_k)$ and recall the representation $\rho_c(\tau)$ and its $(k,c)$ model $W_{\psi}(\rho_c(\tau))$ defined in \S~\ref{local groups and notation}.
For $s\in\C$, let $V(s,\tau,c)$ be the space of the representation\footnote{For $H=\GSpin_{2kc}$ this is the representation $\Ind_{P}^{H}(\absdet^{s-1/2}W_{\psi}(\rho_c(\tau))\otimes
\chi_{\pi})$ of \cite{CFK2022} under the identification of $M_P$ with $\GL_{kc}\times\GSpin_{0}$ given in \cite[\S~2.5]{CFK2022}.}
\begin{align*}
\Ind_{P}^{H}((\absdet^{s-1/2}\chi_{\pi}W_{\psi}(\rho_c(\tau))\otimes\chi_{\pi})\circ i_{M_P}).
\end{align*}
By the Iwasawa decomposition we can realize these representations,
as $s$ varies, all in the space $V^{K_H}(\tau,c)$ of
\begin{align*}
\Ind_{K_H\cap P}^{K_H}((\chi_{\pi}W_{\psi}(\rho_c(\tau))\otimes\chi_{\pi})\circ i_{M_P}).
\end{align*}
Consider the space $V(\tau,c)$ of holomorphic functions $f:\C\to V^{K_H}(\tau,c)$, where if $F$ is non-archimedean we say that $f$ is holomorphic if
$s\mapsto f(s)(y)(I_{kc})$ is entire for all $y\in K_H$, and if $F$ is archimedean $f$ is holomorphic if its composition with
every continuous functional on $V^{K_H}(\tau,c)$ is entire. This space admits a natural representation of $H$.

We regard $f\in V(\tau,c)$ as a complex-valued function $f(s,h)$ on $\C\times H$. Note that for all $s\in\C$, $h\mapsto f(s,h)\in V(s,\tau,c)$,
and for each $h\in H$, the function $s\mapsto f(s,h)$ is entire. The function $f$ is also called an entire section. When $F$ is archimedean, we say that
$f$ is $K_H$-finite if $h\mapsto f(s,h)$ is $K_H$-finite for all $s$.

The integral is defined for a matrix coefficient $\omega$ of $\pi^{\vee}$ and $f\in V(\tau,c)$ by
\begin{align*}
Z(s,\omega,f)=\int\limits_{C_G^{\circforgspin}\backslash G}\int\limits_{U_0}
\omega(g)f(s,\delta u_0\embeddingR({}^{\iota}g))\,\psi_U(u_0)\,du_0\,dg.\notag
\end{align*}
\begin{theorem}(\cite{CFK2022})\label{theorem:all local props}
The integral $Z(s,\omega,f)$ satisfies the following properties.
\begin{enumerate}[leftmargin=*]
  \item It is absolutely convergent in $\Real(s)\gg0$ independently of the data (\cite[Proposition~2.5]{CFK2022}).
  \item\label{it:cont} It extends to a meromorphic function of $s$ which is continuous as a trilinear map on
  $V(s,\tau,c)\times\pi\times(\pi^{\iota})^{\vee}$ (\cite[\S~4, \S~6.10]{CFK2022}).
 \item\label{it:nonzero} For any given $s$ there is a matrix coefficient $\omega$ and a $K_{H}$-finite $f$ for which the integral is finite and nonzero in a neighborhood of $s$ (\cite[Proposition~2.6, Corollary~6.9]{CFK2022}).
\end{enumerate}
\end{theorem}

When data are normalized and unramified, and in addition $K_G=G(\mathcal{O})$, $K_H=H(\mathcal{O})$ and the measure is normalized so that $du_0(U_0\cap K_H)=1$ and $dg(K_G)=1$, by \cite[\S~6.4]{CFK2022}
\begin{align}\label{int:unr}
Z(s,\omega,f)=\frac{L(s,\pi\times\tau)}{b(s,c,\tau\otimes\chi_{\pi})},
\end{align}
where
\begin{align}\label{eq:b}
b(s,c,\tau\otimes\chi_{\pi})=
\begin{dcases}
[L(s+c/2,\tau)]\prod\limits_{j=1}^{c/2}L(2s+2j-2,\tau,\vee^2\otimes\chi_{\pi})L(2s+2j-1,\tau,\wedge^2\otimes\chi_{\pi})&2\mid c,\\
\prod\limits_{j=1}^{\lfloor c/2\rfloor}L(2s+2j-1,\tau,\vee^2\otimes\chi_{\pi})
\prod\limits_{j=1}^{\lceil c/2\rceil}L(2s+2j-2,\tau,\wedge^2\otimes\chi_{\pi})&2\nmid c.
\end{dcases}
\end{align}
Here $L(s+c/2,\tau)$ appears only when $G=\Sp_c$. See \cite[\S~3]{CFK2022}.

We will also need the following property of the inner $du_0$-integral.
\begin{lemma}\cite[Corollary~2.3]{CFK2022}\label{corollary:inv of U on g giota}
In $\Real(s)\gg0$, for any $f\in V(\tau,c)$ and $g_0\in G$,
\begin{align*}
\int\limits_{U_0}f(s,\delta u_0\embedding(g_0,{}^{\iota}g_0))\,\psi_U(u_0)\,du_0
=\int\limits_{U_0}f(s,\delta u_0)\,\psi_U(u_0)\,du_0.
\end{align*}
\end{lemma}

\subsection{\texorpdfstring{$\gamma$-}{gamma}, $L$- and \texorpdfstring{$\epsilon$}{epsilon}-factors}\label{the local factors}
Let $\pi\in\Irr(G)$ and $\tau\in\IrrGen(\GL_k)$.
By \cite[Proposition~2.2]{CFK2022} and Theorem~\ref{theorem:all local props} \eqref{it:cont}, the integral $Z(s,\omega,f)$ can be regarded as a morphism in the space
\begin{align}\label{eq:uniqueness space}
\Hom_{U\rtimes \embedding(G,G)}(V(s,\tau,c),\psi_U^{-1}\otimes\pi^{\vee}\otimes\pi^{\iota}).
\end{align}
By Gourevitch and the third named author \cite[Theorem~2.1]{DimaKaplan} the dimension of this space is at most $1$, except
for finitely many values of $q^{-s}$ over a non-archimedean field, and except for a discrete set of values of $s$ over an archimedean field. As we explain below, this leads to the definition of the local factors via a functional equation.

Let $w_P\in H$ be a representative for $\WeylElement_H\WeylElement_{M_P}$ (see \S~\ref{the groups} for notation). We then have the standard intertwining operator
\begin{align*}
M(s,c,\tau\otimes\chi_{\pi},w_P):V(s,\tau,c)\rightarrow V(1-s,\chi_{\pi}^{-1}\tau^{\vee},c).
\end{align*}
See \cite[\S~3]{CFK2022} for details; note that the $(k,c)$ model is defined with respect to $\psi$ on both sides, and on the r.h.s.~ when $kc$ is odd we induce from ${}^{\Specialjmath}P$ and
the representation
$(\absdet^{1/2-s}W_{\psi}(\rho_c(\tau^{\vee}))\otimes\chi_{\pi})\circ i_{{}^{\Specialjmath}M_P}$ of ${}^{\Specialjmath}M_P$.

According to the uniqueness theorem \cite[Theorem~3.2]{CFK2022}, the representation $V(s,\tau,c)$ affords a certain unique degenerate Whittaker model,
which we can use in order to write a functional equation similar to Shahidi's functional equation \cite[(1.2)]{Sh3}, relating between $f$ and $M(s,c,\tau\otimes\chi_{\pi},w_P)f$.
Specifically, denote the corresponding degenerate Whittaker functional by $\lambda(s,c,\tau\otimes\chi_{\pi},\psi)$. This functional depends on the choices of $\delta_0$ and
$\delta_1$; see \cite[(3.3)]{CFK2022}. Then there is a meromorphic function $C(s,c,\tau\otimes\chi_{\pi},\psi)$ such that for any $f\in V(\tau,c)$,
\begin{align}\label{eq:func equation for normalization}
&\lambda(s,c,\tau\otimes\chi_{\pi},\psi)f=
C(s,c,\tau\otimes\chi_{\pi},\psi)\lambda(1-s,c,\chi_{\pi}^{-1}\tau^{\vee}\otimes\chi_{\pi},\psi)M(s,c,\tau\otimes\chi_{\pi},w_P)f.
\end{align}
A minor modification is needed in \eqref{eq:func equation for normalization} when $kc$ is odd. See \cite[\S~3]{CFK2022} for precise definitions of the functional (in all cases) and of \eqref{eq:func equation for normalization} in the odd case. Define the normalized intertwining operator
\begin{align*}
M^*(s,c,\tau\otimes\chi_{\pi},\psi)=C(s,c,\tau\otimes\chi_{\pi},\psi)M(s,c,\tau\otimes\chi_{\pi},w_P).
\end{align*}
The integral $Z(1-s,\omega,M^*(s,c,\tau\otimes\chi_{\pi},\psi)f)$, which is absolutely convergent in a left half plane,
satisfies the same equivariance properties as $Z(s,\omega,f)$ and thus also belongs to \eqref{eq:uniqueness space}. The following equation then defines a meromorphic function $\gamma(s,\pi\times\tau,\psi)$: for any matrix coefficient $\omega$ of $\pi^{\vee}$ and $f\in V(\tau,c)$,
\begin{align}\label{eq:Gamma def}
&\gamma(s,\pi\times\tau,\psi)Z(s,\omega,f)\\&=\pi(\mathfrak{i}_G)^k[\gamma(s,\tau,\psi)]\vartheta(s,c,\tau\otimes\chi_{\pi},\psi)Z(1-s,\omega,M^*(s,c,\tau\otimes\chi_{\pi},\psi)f).\nonumber
\end{align}
Here $\mathfrak{i}_G=(-1)^{c+1}e_G$ unless $G=\GSpin_{c}$ with $2|c$, then
$\mathfrak{i}_G=i_{M}^{-1}(-I_{c/2},-1)\in C_G$ for a standard Siegel Levi subgroup $M$ of $G$;
the factor $\gamma(s,\tau,\psi)$ appears only when $G=\Sp_c$; and
\begin{align}\label{eq:vartheta}
\vartheta(s,c,\tau\otimes\chi_{\pi},\psi)=\chi_{\pi}(2)^{-kn}\tau(-I_k)^n\tau(2I_k)^{-2n}|2|^{-2kn(s-1/2)},\qquad n=\lfloor c/2\rfloor.
\end{align}
In addition, if $\pi'\in\Irr(\Sp_0)$ we formally set $\gamma(s,\pi'\times\tau,\psi)=\gamma(s,\tau,\psi)$, and for
$\pi'\in\Irr(\GSpin_c)$ with $c\leq1$, $\gamma(s,\pi'\times\tau,\psi)=1$.
\begin{remark}\label{remark:apology}
The function $\vartheta(s,c,\tau\otimes\chi_{\pi},\psi)$ is non-canonical and depends on the choices
made in the construction: the definitions of $G$ and $H$; the character $\psi_U$; the representative $\delta_0$; and the choice of $\delta_1$ when $k=1$. Different choices will multiply $\vartheta(s,c,\tau\otimes\chi_{\pi},\psi)$ by a factor which is entire and nowhere vanishing as a function of $s$. The formula \eqref{eq:vartheta} was computed in \cite{CFK2022} based on the concrete construction described there. Cf. \cite[(25)]{LR}.
\end{remark}

We recall that the Rankin--Selberg $\gamma$-, $L$- and $\epsilon$-factors for pairs of irreducible admissible representations of general linear groups are
defined by \cite{JPSS,JS3} even when the representations are not generic (see \cite[\S~9.4]{JPSS}). Throughout, when we refer to
$\gamma$-, $L$- or $\epsilon$-factors for $\Pi\times\sigma$ where $\Pi\in\Irr(\GL_l)$ and $\sigma\in\Irr(\GL_r)$, we always refer to these factors.

The following summarizes \cite[Theorem~4.2]{CFK2022}.
\begin{theorem}\label{theorem:ten commendments}
Let $\pi\in\Irr(G)$ and $\tau\in\IrrGen(\GL_k)$.
The $\gamma$-factor $\gamma(s,\pi\times\tau,\psi)$ satisfies and is characterized by the following properties.
\begin{itemize}[leftmargin=*]
\item Unramified twisting: for any $s_0,s_1\in\C$,
\begin{align}
\label{eq:unramified twisting}
\begin{cases}\gamma(s,\pi\times|\det|^{s_1}\tau,\psi)=\gamma(s+s_1, \pi\times\tau,\psi)& G\ne\GSpin_c,\\
\gamma(s,|\Upsilon_c|^{s_0}\pi\times|\det|^{s_1}\tau,\psi)=\gamma(s+s_0+s_1,\pi\times\tau,\psi)& G=\GSpin_c.
\end{cases}
\end{align}
\item
Multiplicativity: Let $\pi$ be a quotient of $(\sigma_1\times\ldots\times\sigma_d)\rtimes\pi'$ where each $\sigma_i$ and $\pi'$ are irreducible admissible.
Let $\tau$ be a quotient of $\times_{i=1}^{d'}\tau_i$, where each $\tau_i$ is irreducible admissible generic. Then
    \begin{align}\label{eq:multiplicativity I}
    &\gamma(s,\pi\times\tau,\psi)=\gamma(s,\pi'\times\tau,\psi)\prod_{i=1}^{d}
    \gamma(s,\sigma_i\times\tau,\psi)\gamma(s,\sigma_i^{\vee}\times\chi_{\pi}\tau,\psi),\\
    &\gamma(s,\pi\times\tau,\psi)=\prod_{i=1}^{d'}\gamma(s,\pi\times\tau_i,\psi).\label{eq:multiplicativity II}
    \end{align}
\item
Unramified factors: When all data are unramified,
    \begin{align*}
    \gamma(s,\pi\times\tau,\psi)=\frac{L(1-s,\pi^{\vee}\times\tau^{\vee})}{L(s,\pi\times\tau)}.
    \end{align*}

\item
Duality:
\begin{align*}
\gamma(s,\pi^{\vee}\times\tau,\psi)=\gamma(s,\pi\times\chi_{\pi}^{-1}\tau,\psi).
    \end{align*}

\item
Functional equation:
    \begin{align}\label{eq:functional equation}
    \gamma(s,\pi\times\tau,\psi)\gamma(1-s,\pi^{\vee}\times\tau^{\vee},\psi^{-1})=1.
    \end{align}

\item Dependence on $\psi$: Denote $\psi_b(x)=\psi(bx)$, for $b\in F^*$. Then
    \begin{align}\label{eq:dependence on psi}
    \gamma(s,\pi\times\tau,\psi_b)=\chi_{\pi}^{k\lfloor c/2\rfloor}(b)\tau(bI_k)^{N}|b|^{kN(s-1/2)}\gamma(s,\pi\times\tau,\psi).
    \end{align}

\item
Archimedean property: Assume $F$ is archimedean and let $\transfer(\pi)$ be defined by \eqref{eq:t pi}. Then
\begin{align}\label{eq:Archimedean property}
\gamma(s,\pi\times\tau,\psi)=\epsilon(s,\transfer(\pi)\times\tau,\psi)\frac{L(1-s,\transfer(\pi)^{\vee}\times\tau^{\vee})}{L(s,\transfer(\pi)\times\tau)}.
\end{align}

\item
Crude functional equation:  Let $F_0$ be a number field with ring of adeles $\A$, $\psi=\otimes'_{\nu}\psi_{\nu}$ be a nontrivial additive character of $\A$ which is trivial on $F_0$, and assume $\pi=\otimes'_{\nu}\pi_{\nu}$ and $\tau=\otimes'_{\nu}\tau_{\nu}$ are irreducible cuspidal automorphic representations of $G(\A)$ and $\GL_{k}(\A)$, resp. Let $S$ be a finite set of places of $F_0$ such that for any $\nu\notin S$, all data are unramified. Then
    \begin{align}\label{eq:crude}
    L^S(s,\pi\times\tau)=\prod_{\nu\in S}\gamma(s,\pi_{\nu}\times\tau_{\nu},\psi_{\nu})L^S(1-s,\pi^{\vee}\times\tau^{\vee}).
    \end{align}
\end{itemize}
\end{theorem}

We recall the definitions of the $L$- and $\epsilon$-factors from \cite[\S~7]{CFK2022}.

\underline{Tempered case}: When $\pi$ and $\tau$ are tempered, in the non-archimedean case take $P(X)\in\C[X]$ such that
the zeros of $P(q^{-s})$ are those of $\gamma(s,\pi\times\tau,\psi)$ and $P(0)=1$, and define
$L(s,\pi\times\tau)=P(q^{-s})^{-1}$ (which is independent of $\psi$ by \eqref{eq:dependence on psi}). The $\epsilon$-factor is then defined by
\begin{align}\label{eq:eps}
\epsilon(s,\pi\times\tau,\psi)=\gamma(s,\pi\times\tau,\psi)\frac{L(s,\pi\times\tau)}{L(1-s,\pi^{\vee}\times\tau^{\vee})},
\end{align}
and \eqref{eq:functional equation} implies that $\epsilon(s,\pi\times\tau,\psi)\in\C[q^{-s},q^s]^*$.
In the archimedean case define
\begin{align*}
L(s,\pi\times\tau)=L(s,\transfer(\pi)\times\tau),\qquad
\epsilon(s,\pi\times\tau,\psi)=\epsilon(s,\transfer(\pi)\times\tau,\psi).
\end{align*}

\underline{Essentially tempered}: When $\pi=|\Upsilon_c|^{r_0}\pi_0$ and $\tau=|\det|^{r_1}\tau_0$ for $r_0,r_1\in\R$ and tempered representations $\pi_0$ and $\tau_0$, with
$r_0=0$ for $G\ne\GSpin_c$, define
\begin{align*}
L(s,\pi\times\tau)=L(s+r_0+r_1,\pi_0\times\tau_0),\qquad \epsilon(s,\pi\times\tau,\psi)=\epsilon(s+r_0+r_1,\pi_0\times\tau_0,\psi).
\end{align*}
The relation \eqref{eq:eps} still holds by \eqref{eq:unramified twisting}.

\underline{Langlands quotient}: In general write $\pi$ as the Langlands quotient of
\begin{align}\label{LQpi}
\sigma\rtimes\pi',
\end{align}
where $\sigma\in\Irr(\GL_l)$, $l>0$, $\sigma$ is the unique irreducible quotient of $\times_{i=1}^{d}\sigma_i$, with $\sigma_i=\absdet^{a_i}\sigma_{i,0}$ for a tempered representation $\sigma_{i,0}$, $\pi'\in \Irr(\mathcal{G}_{c-2l})$ is essentially tempered, and
$a_1>\ldots>a_{d}>r/2$ where $r$ is such that $|\Upsilon_{c-2l}|^{-r}\pi'$ is tempered ($r=0$ if $G\ne\GSpin_c$). Also write $\tau=\times_{i=1}^{d'}\tau_i$ where each $\tau_i$ is essentially tempered. We then have
the Rankin--Selberg $L$-factors $L(s,\sigma\times\tau)=\prod_{i,j}L(s,\sigma_i\times\tau_j)$ and $L(s,\sigma^{\vee}\times\chi_{\pi}\tau)$. Define
\begin{align}\label{eq:gen L def}
&L(s,\pi\times\tau)=L(s,\sigma\times\tau)L(s,\sigma^{\vee}\times\chi_{\pi}\tau)\prod_jL(s,\pi'\times\tau_j),\\
&\epsilon(s,\pi\times\tau,\psi)=\epsilon(s,\sigma\times\tau,\psi)\epsilon(s,\sigma^{\vee}\times\chi_{\pi}\tau,\psi)
\prod_j\epsilon(s,\pi'\times\tau_j,\psi).\nonumber
\end{align}
Note that if $\pi'\in\Irr(\Sp_0)$, by the definition of $\gamma(s,\pi'\times\tau_j)$ we have $L(s,\pi'\times\tau_j)=L(s,\tau_j)$, while if $\pi'\in\Irr(\GSpin_c)$ with $c\leq1$,
by definition $\gamma(s,\pi'\times\tau_j)=1$ hence $L(s,\pi'\times\tau_j)=1$.
We then have \eqref{eq:eps} in general.
\begin{remark}
The formulations of \eqref{eq:unramified twisting} and \eqref{eq:multiplicativity I}
for $\GSpin_c$ are slightly different from \cite{CFK2022} because our conventions for $\Upsilon_l$ and $i_M$ here are chosen to be compatible with \cite{AsgSha}.
In turn, for the $L$-factor we have $L(s,\sigma\times\tau)L(s,\sigma^{\vee}\times\chi_{\pi}\tau)$ in \eqref{eq:gen L def} instead of
$L(s,\sigma\times\chi_{\pi}\tau)L(s,\sigma^{\vee}\times\tau)$ in \textit{ibid.}, with a similar change in the $\epsilon$-factor.
\end{remark}

The following result of \cite{CFK2022} shows that our definitions are reasonable.
\begin{theorem}\label{theorem:pi and Pi for unramified pi or archimedean}
Let $\pi\in\Irr(G)$ and $\tau\in\IrrGen(\GL_k)$. If $F$ is non-archimedean we further suppose $\pi$ is unramified.
Then, the $\gamma$-, $L$- and $\epsilon$-factors of $\pi\times\tau$ and $\transfer(\pi)\times\tau$ coincide.
\end{theorem}
\begin{proof}
In the non-archimedean case this was proved in \cite[Lemma~7.1]{CFK2022}. In the archimedean case
this follows from the definitions and from \eqref{eq:Archimedean property}.
\end{proof}

We also need the following results.
\begin{theorem}(\cite[Corollary~7.2]{CFK2022})\label{theorem:L for unr temp and unitary is holomorphic in half plane}
Let $\pi\in\Irr(G)$ and $\tau\in\IrrGen(\GL_k)$. Suppose $\pi$ is tempered, and if $F$ is non-archimedean also assume $\pi$ is unramified. Then,
\begin{enumerate}[leftmargin=*]
  \item If $\tau$ is tempered, $L(s,\pi\times\tau)$ is holomorphic in $\Real(s)>0$.
  \item If $\tau$ is unitary, $L(s,\pi\times\tau)$ is holomorphic in $\Real(s)\geq1/2$.
\end{enumerate}
\end{theorem}

\begin{corollary}\label{corollary:L of pi vs L of sigma in half plane}
Let $\pi\in\Irr(G)$ and $\tau\in\IrrGen(\GL_k)$. Assume that $\tau$ and $\chi_{\pi}$ are unitary, and if $F$ is non-archimedean also suppose $\pi$ is unramified.
Write $\pi$ as the Langlands quotient of $\sigma\rtimes\pi'$ with the assumptions as in \eqref{LQpi} and in particular, $l>0$. Then,
$L(s,\pi\times\tau)/L(s,\sigma^{\vee}\times\chi_{\pi}\tau)$ is holomorphic and nonzero in $\Real(s)\geq1/2$.
\end{corollary}
\begin{proof}
Since $\chi_{\pi}$ is unitary, the essentially tempered representation $\pi'$ of $\mathcal{G}_{c-2l}$ is unitary hence tempered.
Note that when $F$ is non-archimedean, $\pi$ is unramified whence so is $\pi'$.

Since $a_i>0$, $L(s,\sigma\times\tau)$ is holomorphic in $\Real(s)\geq1/2$.
In addition $L(s,\pi'\times\tau)$ is holomorphic in $\Real(s)\geq1/2$ by Theorem~\ref{theorem:L for unr temp and unitary is holomorphic in half plane}; note also that
if $l=n$ and $G=\Sp_c$, $L(s,\pi'\times\tau)=L(s,\tau)$ by the definitions. Looking at \eqref{eq:gen L def} we deduce that
the poles of $L(s,\pi\times\tau)$ in $\Real(s)\geq1/2$ coincide with the poles of $L(s,\sigma^{\vee}\times\chi_{\pi}\tau)$.
\end{proof}

\begin{theorem}(\cite[Lemma~7.3]{CFK2022}, Stability)\label{theorem:pi and Pi for ramified twisted}
Assume $F$ is non-archimedean. Let $\pi\in\Irr(G)$ and $\Pi\in\Irr(\GL_N)$. Suppose the central character of $\Pi$ is $\chi_{\pi}^{N/2}$.
If $\eta$ is a sufficiently highly ramified character of $F^*$, depending on $\pi$ and $\Pi$, then for any unramified $\tau\in\IrrGen(\GL_k)$, the $\gamma$-, $L$- and $\epsilon$-factors of $\pi\times\eta\tau$ and $\Pi\times\eta\tau$ coincide and moreover,
$L(s,\pi\times\eta\tau)=L(s,\Pi\times\eta\tau)=1$.
\end{theorem}
\begin{remark}\label{remark:stability}
\begin{enumerate}[leftmargin=*]
\item In \cite[Lemma~7.3]{CFK2022} it was assumed for $G=\GSpin_c$ that $\Pi$ is unramified, but this assumption was never used in the proof.
\item The representation $\Pi$ in \cite[Lemma~7.3]{CFK2022} was taken to be generic. However, the generic case already implies the result in general:
Indeed $\Pi\in\Irr(\GL_N)$ is a constituent of $\times_{i=1}^d\sigma_i$ where each $\sigma_i$ is essentially tempered. Then if $\Pi'$ is the generic constituent of $\times_{i=1}^d\sigma_i$, under the assumptions of the theorem we have $\gamma(s,\pi\times\eta\tau,\psi)=\gamma(s,\Pi'\times\eta\tau,\psi)$, and the definition in \cite[\S~9.4]{JPSS} implies $\gamma(s,\Pi'\times\eta\tau,\psi)=\gamma(s,\Pi\times\eta\tau,\psi)$. Moreover if $\eta$ is sufficiently highly ramified such that $L(s,\sigma_i\times\eta\tau)=1$ for all $i$, then
    $L(s,\Pi\times\eta\tau)=1$.
\end{enumerate}
\end{remark}

\subsection{Coarse transfer}\label{coarse transfer}
Motivated by Theorem~\ref{theorem:pi and Pi for unramified pi or archimedean} we introduce the following definition.
Let $\pi\in\Irr(G)$. We call $\Pi\in\Irr(\GL_N)$ a ``coarse transfer" of $\pi$ if for every $k\geq1$ and every $\tau\in\Irr(\GL_k)$,
\begin{align}\label{local-FE-gamma}
\gamma(s,\pi\times\tau,\psi)=\gamma(s,\Pi\times \tau,\psi).
\end{align}
For example, if $\pi$ is unramified or $F$ is archimedean, by Theorem~\ref{theorem:pi and Pi for unramified pi or archimedean}, $\transfer(\pi)$ is a coarse transfer of $\pi$.
According to \eqref{eq:multiplicativity II} and the similar multiplicativity of the Rankin--Selberg $\gamma$-factors
(see \cite{JPSS,JS3}), it is sufficient to verify \eqref{local-FE-gamma} for the supercuspidal representations $\tau$.

Assume $F$ is non-archimedean. By \cite[Theorem~3.1]{JPSS}, if $\Pi$ is a coarse transfer of $\pi$, then so is any $\Pi'\in\Irr(\GL_N)$ which has the same supercuspidal support as $\Pi$. Conversely, by Henniart \cite[Proposition~1.9]{Henniart2002}, the supercuspidal support of a coarse transfer (if it exists)
depends only on $\pi$. By contrast, an irreducible \textit{generic} representation of a split classical group is already determined by its twisted
$\gamma$-factors, see for example \cite{JSd1,Morimoto2018,Zhang2018,Zhang2019}.

\subsection{The poles of the $L$-factor}\label{Producing poles}
We study the poles in $\Real(s)\geq1/2$ of the $L$-factor, under certain assumptions (which will be sufficient for the local aspects of the proof of Theorem~\ref{theorem:twisting to obtain entire L function} below). If $F$ is non-archimedean let $\IrrUnrForPoles(G)\subset\Irr(G)$ be the subset of unramified representations, and if $F$ is archimedean take $\IrrUnrForPoles(G)=\Irr(G)$. Let $\IrrUnrUniForPoles(G)\subset\IrrUnrForPoles(G)$ be the subset of representations $\pi$ for which $\chi_{\pi}$ is unitary. Also denote by $\IrrGenUni(\GL_k)\subset\IrrGen(\GL_k)$ the subset of unitary (generic) representations.

\begin{theorem}\label{theorem:archimedean producing poles}
Let $\pi\in\IrrUnrUniForPoles(G)$ and $\tau\in\IrrGenUni(\GL_k)$.
Then for each $s$ with $\Real(s)\geq1/2$, there is a matrix coefficient $\omega$ of $\pi^{\vee}$
and a $K_H$-finite $f\in V(\tau,c)$ such that
\begin{align*}
Z(s,\omega,f)/L(s,\pi\times\tau)
\end{align*}
is nonzero at $s$ (but may have a pole).
\end{theorem}

\underline{Summary of the proof}: Write $\pi$ as the Langlands quotient of $\sigma\rtimes\pi'$ with the notation of \eqref{LQpi}. In light of Corollary~\ref{corollary:L of pi vs L of sigma in half plane}, the main difficulty is to show that the poles of
$L(s,\sigma^{\vee}\times\chi_{\pi}\tau)$ are accounted for by $Z(s,\omega,f)$. The result is obtained by a series of reductions
from $Z(s,f,\omega)$ to a $\GL_l\times\GL_k$ integral, which admits the poles.
Our first step is to substitute a standard integral formula for $\omega$ into $Z(s,f,\omega)$, see
\eqref{eq:matrix coefficient omega formula}. Then we arrive at an iterated integral \eqref{eq:integral 1}.
The inner integral is denoted $Z^1$, see \eqref{eq:integral 1.7}. In Proposition~\ref{proposition:integral over the open cell} we present
$Z^1$ as an integral over the open cell $PU_P^-$. Next, in Proposition~\ref{proposition:Pole Inner Z1}, we compute $Z^1$ for a section supported on $PU_P^-$, and obtain an
inner integral $\ZInnerInner$, see \eqref{eq:mult inner start 2}. In Proposition~\ref{proposition:inner 2} we reduce $\ZInnerInner$ to the $\GL_l\times\GL_k$ integral sought after.

For $k=1$ Theorem~\ref{theorem:archimedean producing poles} already follows from the proof of \cite[Theorem~7.1]{Yamana}, with minor adjustments for the groups
(in \textit{ibid.} the full orthogonal groups $\Orth_c$ were considered, and $\GSpin_c$ was not considered). Thus we assume $k>1$.

For any algebraic group $X$ defined and split over $F$, denote by $\mathcal{S}(X)$ the space of Schwartz functions on $X$ and by
$\SchwartzC(X)\subset \mathcal{S}(X)$ the subspace of compactly supported ones. The following properties of functions in $W_{\psi}(\rho_c(\tau))$ will be used to establish convergence.
\begin{lemma}\label{lemma:main convergence}
Let $\tau\in\IrrGen(\GL_k)$ ($k>1$). There is some constant $d_0>0$ such that for any
$W\in W_{\psi}(\rho_c(\tau))$, there is $\phi\in\mathcal{S}(\Mat_c)$ such that for all $g\in\GL_c$,
\begin{align*}
|W(\diag(g,I_{(k-1)c}))|\leq |\det g|^{-d_0}\phi(g).
\end{align*}
\end{lemma}
\begin{proof}
First observe that the support of $g\mapsto W(\diag(g,I_{(k-1)c}))$ is contained in the support of a Schwartz function. This follows as in
\cite[Lemma~(4.1.5)]{JPSS2} and \cite[Proposition~6.1]{Jac5}, using the $\psi_k$-equivariance of $W$ under the subgroup of elements of $U_{(c^k)}$ of the form $\diag(\left(\begin{smallmatrix}I_c&u\\&I_c\end{smallmatrix}\right),I_{(k-2)c})$.
Then in the non-archimedean case we argue as in Jacquet and Rallis \cite[\S~6]{JR2} by writing an asymptotic expansion of $W$ on the subgroup $\diag(T_{\GL_c},I_{(k-1)c})$.
In the archimedean case we argue as in Aizenbud, Gourevitch and Jacquet \cite[Lemmas 3.2--3.4]{AGJ}, and note that
\cite[Lemma~3.4]{AGJ} is immediately seen to be applicable.
\end{proof}
Denote the space of $\rho_c(\tau)$ by $V_{\rho_c(\tau)}$. We generalize \cite[Lemma~3.3]{AGJ}.
\begin{lemma}\label{lemma:main N invariance}
Assume $F$ is archimedean and let $\tau\in\IrrGen(\GL_k)$. Fix a $(k,c)$ functional $\lambda$ on $\rho_c(\tau)$ (see \S~\ref{local groups and notation}). There is a constant $d_0>0$ and a continuous seminorm $\gamma$ on $V_{\rho_c(\tau)}$ such that for any $\xi\in V_{\rho_c(\tau)}$, $g=\diag(g_1,\ldots,g_k)\in M_{(c^k)}$, $u\in U_{(c^k)}$ and $o\in K_{\GL_{kc}}$,
\begin{align}\label{eq:conv main N 1}
|\lambda(\rho_c(\tau)(ugo)\xi)|\leq\gamma(\xi)|\tau((\det g_k)I_k)|\cdot\prod_{i=1}^{k-1}|\det (g_k^{-1}g_i)|^{-d_0}.
\end{align}
Moreover, for any $\xi$ there are $\phi_1,\ldots,\phi_{k-1}\in\mathcal{S}(\Mat_c)$ such that for all $g,u$ and $o$ as above,
\begin{align}\label{eq:conv main N 2}
|\lambda(\rho_c(\tau)(ugo)\xi)|\leq|\tau((\det g_k)I_k)|\cdot\prod_{i=1}^{k-1}|\det (g_k^{-1}g_i)|^{-d_0}\phi_i(g_{i+1}^{-1}g_i).
\end{align}
\end{lemma}
\begin{proof}
As above we follow the proofs of \cite[Lemmas 3.2--3.4]{AGJ}, but a minor modification to \cite[Lemma~3.4]{AGJ} is required.
For $h=\diag(h_1,\ldots,h_{k-1})\in M_{(c^{k-1})}$, let
\begin{align*}
h^{\times}=\diag(h_{k-1}\dots h_1,h_{k-1}\dots h_2,\ldots,h_{k-1}h_{k-2},h_{k-1},I_c)\in M_{(c^k)}.
\end{align*}
Let $\Mat_c^{k-1}=\Mat_{c}\times\ldots \times\Mat_{c}$ ($k-1$ copies), regarded as an $\R$-space.
We show that there is a constant $D>0$ such that for any polynomial $P$ on $\Mat_c^{k-1}$, there is a continuous seminorm $\gamma_P$ on $V_{\rho_c(\tau)}$ such that
for all $\xi\in V_{\rho_c(\tau)}$ and all $h^{\times}$ as above,
\begin{align}\label{eq:3.4}
|\lambda(\rho_c(\tau)(h^{\times})\xi)|\leq \gamma_P(\xi)|\det{h^{\times}}|^{-2D}/|P(h)|.
\end{align}
(Cf. \cite[Lemma~3.4]{AGJ}.)
First observe that for any $u\in U_{(c^k)}$,
\begin{align*}
\lambda(\rho_c(\tau)(h^{\times}u)\xi)=\psi(\tr(\sum_{i=1}^{k-1}h_iu_{i,i+1}))\lambda(\rho_c(\tau)(h^{\times})\xi).
\end{align*}
It then follows that for any polynomial $Q$ on $\Mat_c^{k-1}$, there is an element $X_Q$ in the complexification of the universal enveloping algebra of the Lie algebra of $\GL_{kc}$, such that
\begin{align}\label{eq:3.4 2}
\lambda(\rho_c(\tau)(h^{\times})d\rho_c(\tau)(X_Q)\xi)=Q(h)\lambda(\rho_c(\tau)(h^{\times})\xi),\qquad\forall \xi,\, h^{\times}.
\end{align}
Let $\gamma_0$ be a continuous seminorm on $V_{\rho_c(\tau)}$ and $D>0$ be an integer such that for all $\xi$ and $b\in\GL_{kc}$,
$|\lambda(\rho_c(\tau)(b)\xi)|\leq\gamma_0(\xi)||b||^{D}$, and let also $P_0$ be a polynomial on $\Mat_c^{k-1}$ such that
$||h^{\times}||^{D}\leq P_0(h)|\det{h^{\times}}|^{-2D}$. Here $||\cdot||$ is the norm on $\GL_{kc}$ of \cite{AGJ}. Then by \eqref{eq:3.4 2},
\begin{align*}
|Q(h)\lambda(\rho_c(\tau)(h^{\times})\xi)|\leq \gamma_0(d\rho_c(\tau)(X_Q)\xi)P_0(h)|\det{h^{\times}}|^{-2D},\qquad\forall \xi,\, h^{\times}.
\end{align*}
We infer \eqref{eq:3.4}. Then \eqref{eq:conv main N 1} for $g=\diag(g_1,\ldots,g_{k-1},I_c)$ follows from this, and the extension to $g=\diag(g_1,\ldots,g_k)$ follows from \cite[Lemma~12]{CFKmodels}. Regarding \eqref{eq:conv main N 2}, as in \cite{AGJ} this follows essentially from the proof of Friedberg and Jacquet \cite[Lemma~3.1]{FJ}. Briefly,
by the Dixmier--Malliavin Theorem \cite{DM} it is sufficient to consider $\xi$ of the form $\phi(\xi)=\int_{\GL_{kc}}\rho_c(\tau)(m)\xi\phi(m)\,dm$ where $\phi\in\SchwartzC(\GL_{kc})$. We proceed as in \cite{FJ} and use \eqref{eq:conv main N 1}.
\end{proof}

We introduce the following notation. For any $h\in P'$ (resp., $h\in {P'}^-$), denote $h=\Levi{h}\Uni{h}$ (resp.,
$h=\Levi{h}\Unim{h}$) where $\Levi{h}\in M_{P'}$ and $\Uni{h}\in U_{P'}$ (resp., $\Unim{h}\in U_{P'}^-$).
Also denote, for any $g\in G$, $\Es{g}=\embeddingR({}^{\iota}g)\in H$. Put $n=\lfloor c/2\rfloor$.

The center of $U_R$, $C_{U_R}$, can be identified with a subgroup of matrices in $\Mat_l$ satisfying a certain symmetry condition, depending on $G$ (see e.g., $z_3$ in \eqref{eq:z for Sp example} below). There is a homeomorphism, or a diffeomorphism over archimedean fields, of $U_R$ and $\Mat_{l\times n-l}\times\Mat_{l\times c-n-l}\times C_{U_R}$.
Denote
\begin{align*}
&\SUParts=\mathcal{S}(\Mat_{l\times n-l})\otimes\mathcal{S}(\Mat_{l\times c-n-l})\otimes \mathcal{S}(C_{U_R})
\end{align*}
and define $\SUPartsc$ similarly with $\SchwartzC(\cdots)$. Then $\SUParts\subset\mathcal{S}(U_R)$ and $\SUPartsc\subset\SchwartzC(U_R)$.

For $z\in U_R$, we have $\embeddingR(z)\in P'$, $\Es{z}\in{P'}^-$ and ${}^{\Es{z}^{-1}}\delta_1\in P'$. Thus ${}^{\Es{z}^{-1}}\delta_1=\Levi{{}^{\Es{z}^{-1}}\delta_1}\Uni{{}^{\Es{z}^{-1}}\delta_1}$. Define
\begin{align}\label{eq:LeviZ}
\LeviZ=i_{M_P}(\,{}^{\delta_0}(\Levi{\Es{z}}\Levi{{}^{\Es{z}^{-1}}\delta_1})\,)\in\GL_{kc}.
\end{align}
(On the r.h.s.~ the $\mathcal{G}_0$ factor is trivial.) By computation $\LeviZ\in \diag(U_{(l,c-2l,l)}\cap U_{(l,c-n-l,n)},I_{(k-1)c})$.
\begin{example}For $G=\Sp_c$,
\begin{align}\label{eq:z for Sp example}
z=&\left(\begin{smallmatrix}
  I_l & z_{1} & & \\
   & I_{n-l} &  &  \\
   &  & I_{n-l} & -J_{n-l}{}^tz_{1}J_l \\
   &  &  & I_l
\end{smallmatrix}\right)\left(\begin{smallmatrix}
  I_l &  & z_{2} & z_{3} \\
   & I_{n-l} &  & J_{n-l}{}^tz_{2}J_l \\
   &  & I_{n-l} &  \\
   &  &  & I_l
\end{smallmatrix}\right)\in U_R,\qquad J_lz_3={}^tz_3J_l,
\\\LeviZ=&
\diag(\left(\begin{smallmatrix}
        I_l & z_{1} & -z_{2} & \quad-z_3-z_{2}J_{n-l}{}^tz_{1}J_l\\
         & I_{n-l} &  & \quad \\
         &  & I_{n-l} & \quad \\
         &  & \quad & I_l
\end{smallmatrix}\right)
\left(\begin{smallmatrix}
        I_l &  &  & 0\\
         & I_{n-l} &  & -J_{n-l}{}^tz_{2}J_l  \\
         &  & I_{n-l} & 0 \\
         &  &  & I_l
\end{smallmatrix}\right),I_{(k-1)c}).\notag
\end{align}
\end{example}
For $a\in\GL_l$, set
\begin{align*}
&\LeviA=\diag(a,I_{kc-l})\in\GL_{kc},\qquad
\LeviA'=\diag(a,I_{(k-1)l})\in\GL_{kl}.
\end{align*}
\begin{lemma}\label{lemma:reduce c to l}
Assume $\tau\in\IrrGenUni(\GL_k)$. Define for $W\in W_{\psi}(\rho_c(\tau))$ and $\mu\in \SUParts$,
\begin{align}\label{Eq:integral W and mu}
&\mathcal{T}(W,\mu)=\int\limits_{U_R}W(\LeviZ)\mu(z)\,dz.
\end{align}
Then, the integral \eqref{Eq:integral W and mu} is absolutely convergent and defines a functional $\mathcal{T}$ on $W_{\psi}(\rho_c(\tau))\otimes\SUParts$.
Moreover, the map $a\mapsto\mathcal{T}(\LeviA\cdot W\otimes\mu)$ on $\GL_l$ is a surjection into the space of functions $a\mapsto W_l(\LeviA')$ where
$W_l$ varies in $\absdet^{(kc-kl-c+l)/2}W_{\psi}(\rho_l(\tau))$.
\end{lemma}
\begin{proof}
The proof follows, essentially, by a careful inspection of the ``root eliminations" in the proof of \cite[Lemma~9]{CFKmodels}. We provide the details.
We use the realization of $W_{\psi}(\rho_c(\tau))$ from \cite[\S~3.2]{CFKmodels}, which is applicable when $\tau\in\IrrGenUni(\GL_k)$, inductively, for the composition
$(l,c-2l,l)$ of $c$. The representation $\rho_c(\tau)$ is a subrepresentation of
\begin{align*}
\Ind_{R_{(kl,k(c-2l),kl)}}^{\GL_{kc}}(\Big(W_{\psi}(\rho_l(\tau))\otimes W_{\psi}(\rho_{c-2l}(\tau))
\otimes W_{\psi}(\rho_{l}(\tau))\Big)\delta_{P_{(kl,k(c-2l),kl)}}^{-1/(2k)}).
\end{align*}
Let $V<U_{(c^k)}$ be the subgroup of $v=(v_{i,j})\in U_{(c^k)}$ such that for all $i<j$,
\begin{align*}
v_{i,j}=\left(\begin{smallmatrix}0_l \\ v_{i,j}^1 & 0_{c-2l} \\ v_{i,j}^{2,1} & v_{i,j}^{2,2} &  0_l\end{smallmatrix}\right)\in\Mat_c,\qquad
v_{i,j}^2=\left(\begin{smallmatrix}v_{i,j}^{2,1} & v_{i,j}^{2,2}\end{smallmatrix}\right),
\qquad\text{$v_{i,j}^1$ and $v_{i,j}^2$ are arbitrary}.
\end{align*}
For integers $l_1,l_2\geq0$ consider the permutation matrix $\kappa_{l_1,l_2}\in\GL_{k(l_1+l_2)}$ such that
${}^{\kappa_{l_1,l_2}}M_{((l_1,l_2)^k)}=M_{(l_1^k,l_2^k)}$ and, regarded as a permutation, $\kappa_{l_1,l_2}$ is of minimal length and preserves the order of the integers of the same size in the composition $(l_1,l_2)^k$ (see \eqref{kappa} in the appendix). Put
$\kappa=\diag(I_{kl},\kappa_{c-2l,l})\kappa_{l,c-l}\in\GL_{kc}$.
By \cite[Lemma~9]{CFKmodels} (applied twice), we can realize the $(k,c)$ functional $\lambda$ on $\rho_c(\tau)$ using the absolutely convergent integral
$\lambda(\xi)=\int\limits_V\xi(\kappa v)\,dv$, where $\xi\in V_{\rho_c(\tau)}$. See also \cite[\S~5.3]{CFK2022}.

For each $v\in V$, ${}^{\LeviZ^{-1}}v=uv_z$ for some $u\in U_{(c,(k-1)c)}$, where $v_z\in V$ depends on $v$ and $z$, but we can change variables in $v_{1,2}$ to remove the dependency on $z$. Since $\psi_k$ is trivial on $V$, it follows that $\psi_k({}^{\LeviZ^{-1}}v)$ depends only on $z$ and $v_{1,2}$ (because these specify the block $u_{1,2}$ of $u$). Also observe that ${}^{\kappa}\LeviZ\in U_{(kl,k(c-2l),kl)}$, so that $\xi(\kappa\LeviZ)=\xi(\kappa)$. We obtain
\begin{align}\label{eq:intergal with z}
\int\limits_V\xi(\kappa v\LeviZ)\,dv=\int\limits_V\xi(\kappa v)\psi(\tr(<z,v>))\,dv,
\end{align}
where $\psi(\tr(<z,v>))$ is a bilinear pairing when we identify $z=(z_1,z_2,z_3)\in \Mat_{l\times n-l}\times\Mat_{l\times c-n-l}\times C_{U_R}$ as explained above;
this pairing is non-degenerate on the left and degenerate on the right (there is $v\ne0$ such that $\psi(\tr(<z,v>))=0$ for all $z$).

Thus, the absolute convergence of the integral~\eqref{Eq:integral W and mu} follows from that of $\lambda(\xi)$.

Define a total order $\succ$ on the triplets $\{(i,j,d)\}_{1\leq i<j\leq k,1\leq d\leq 2}$ by $(i,j,2)\succ(i',j',1)$ for all $i,j,i',j'$, and for a fixed $d\in\{1,2\}$
order the triplets by regarding them as blocks of $V$, starting from $(k-1,k)$ and proceeding along the diagonals, bottom to top, left to right:
\begin{align*}
(k-1,k,d)\succ (k-2,k-1,d)\succ\ldots\succ(1,2,d)\succ(k-2,k,d)\succ\ldots\succ(1,k,d).
\end{align*}
Let $(i,j,d)_-$ be the maximal $(i',j',d')$ such that $(i,j,d)\succ(i',j',d')$, or the empty set if $(i,j,d)=(1,k,1)$. Let
\begin{align*}
&V_{\succ(i,j,d)}=\{v\in V:v^{d'}_{i',j'}=0,\forall (i',j',d')\succ(i,j,d)\}<V.
\end{align*}
Note that $V_{\succ((i,j,d)_-)}$, which is understood to be the trivial group for $(i,j,d)=(1,k,1)$, is a normal subgroup of $V_{\succ(i,j,d)}$, and the quotient, which is also a subgroup of
$V_{\succ(i,j,d)}$, is identified with the block $v_{i,j}^d$. For a fixed $z\in U_R$ and $(i,j,d)\succ(1,2,2)$ define
\begin{align*}
\mathcal{T}_{(i,j,d)}^{\,\LeviZ}(\xi)=\int\limits_{V_{\succ(i,j,d)}}\xi(\kappa v)\psi(\tr(<z,v>))\,dv.
\end{align*}
By definition $V_{\succ(k-1,k,2)}=V$. By \eqref{eq:intergal with z}, and writing $W=W_{\xi}$, $W_{\xi}(h)=\lambda(h\cdot\xi)$,
\begin{align}\label{eq:TW first}
\mathcal{T}(W_{\xi}\otimes\mu)=\int\limits_{U_R}\mathcal{T}_{(k-1,k,2)}^{\,\LeviZ}(\xi)\mu(z)\,dz.
\end{align}
Note also that $\lambda=\mathcal{T}_{(k-1,k,2)}^{\,I_{kc}}$.

Let $X_{i,j}^+<\Mat_{kc}$ be the subgroup of matrices $x=(x_{i,j})_{1\leq i,j\leq k}$ such that $x_{i',j'}=0$ for all $(i',j')\ne(i,j)$ and
\begin{align*}
x_{i,j}=\left(\begin{smallmatrix}0_l & x_{i,j}^1 & x_{i,j}^{2,1} \\  & 0_{c-2l} & x_{i,j}^{2,2} \\  & &  0_l\end{smallmatrix}\right)\in\Mat_c,\qquad
x_{i,j}^2=\left(\begin{smallmatrix}x_{i,j}^{2,1} \\x_{i,j}^{2,2}\end{smallmatrix}\right),
\qquad\text{$x_{i,j}^1$ and $x_{i,j}^2$ are arbitrary}.
\end{align*}
Define $X_{i,j}=I_{kc}+X_{i,j}^+<\GL_{kc}$ and note that $X_{i,i}<N_{\GL_{kc}}\cap M_{(c^k)}$ and for $i>j$, $X_{i,j}<N_{\GL_{kc}}^-$.
For $d\in\{1,2\}$, let $X_{i,j}^{d}<X_{i,j}$ be the subgroup such that $x_{i,j}^{d'}=0$ for $d'\ne d$.

Let $1\leq i<j\leq k$ and $d\in\{1,2\}$. For any $\phi\in\mathcal{S}(X_{j,i+1}^{d})$ and $\phi'\in\mathcal{S}(v_{i,j}^{d})$, define
\begin{align}\label{eq:convolution maps 1}
&\phi(\xi)=\int\limits_{X_{j,i+1}^{d}}x\cdot\xi\,\phi(x)\,dx,\\\label{eq:convolution maps 2}
&\phi'(\xi)=\int\limits_{v_{i,j}^{d}}v\cdot\xi\,\phi'(v)\,dv.
\end{align}
Here $v_{i,j}^{d}$ is regarded as a subgroup of $V$, as explained above.
Let $\widehat{\phi}$ denote the Fourier transform of $\phi$ with respect to $\psi$. Then for
$(i,j,d)\succ(1,2,2)$, $\mathcal{T}_{(i,j,d)}^{\,I_{kc}}(\phi(\xi))=\mathcal{T}_{(i,j,d)_-}^{\,I_{kc}}(\widehat{\phi}(\xi))$, as can be directly checked.
Moreover, for such triplets both $X_{j,i+1}$ and $v_{i,j}^2$ are contained in $\diag(I_{2c},\GL_{(k-2)c})$, whence their elements stabilize $\psi(\tr(<z,v>))$. Hence for each $z$,
\begin{align}\label{eq:T to T}
\mathcal{T}_{(i,j,d)}^{\,\LeviZ}(\phi(\xi))=\mathcal{T}_{(i,j,d)_-}^{\,\LeviZ}(\widehat{\phi}(\xi)).
\end{align}
Therefore by \eqref{eq:TW first} and by repeatedly applying \eqref{eq:T to T},
\begin{align*}
\mathcal{T}(W_{\xi}\otimes\mu)=\int\limits_{U_R}\mathcal{T}_{(1,2,2)}^{\,\LeviZ}(\xi')\mu(z)\,dz=
\int\limits_{U_R}\int\limits_{V_{\succ(1,2,2)}}\xi'(\kappa v)\psi(\tr(<z,v>))\mu(z)\,dv\,dz,
\end{align*}
where $\xi'$ is obtained from $\xi$ by repeatedly convolving it against Schwartz functions.

Consider a pure tensor $\mu$, i.e., $\mu(z)=\prod_{i=1}^3\mu_i(z_i)$ for Schwartz functions $\mu_i$. After a change of variables in $v_{1,2}^1$, when we integrate over $U_R$ we obtain
\begin{align*}
\int\limits_{U_R}\mu(z)\psi(\tr(<z,v>))\,dz=\widehat{\mu}_1(y_1)\widehat{\mu}_2(y_2)\widehat{\mu}_3(v_{1,2}^{2,1}),
\end{align*}
where $y_1\in\Mat_{n-l\times l}$, $y_2\in \Mat_{c-n-l\times l}$ and $v_{1,2}^{1}=\left(\begin{smallmatrix}y_1\\y_2\end{smallmatrix}\right)$.
The character $z_3\mapsto\psi(\tr(z_3v_{1,2}^{2,1}))$ is degenerate. To control the remaining coordinates of $v_{1,2}^{2,1}$ we replace $\xi'$ with $\phi(\xi')$ where $\phi$ is a Schwartz function on a subgroup of $\{x\in X_{2,2}^2:x_{2,2}^{2,2}=0\}$. It follows that
\begin{align}\label{eq:TW third}
\int\limits_{U_R}\int\limits_{V_{\succ(1,2,2)}}\phi(\xi')(\kappa v)\psi(\tr(<z,v>))\mu(z)\,dv\,dz=\int\limits_{
\{v\in V_{\succ(k-2,k,2)}:v_{1,2}^1=0\}}\widehat{\mu}_1(\widehat{\mu}_2(\widehat{\mu}_3(\widehat{\phi}(\xi'))))(\kappa v)\,dv,
\end{align}
where the convolution sections on the r.h.s.~ are defined similarly to \eqref{eq:convolution maps 2}.

Now we proceed as in \cite[Lemma~9]{CFKmodels}: define for any $(i,j,d)\prec(1,2,2)$ such that $(i,j,d)\ne(1,2,1)$,
$\mathcal{T}_{(i,j,d)}(\xi)=\int\limits_{V_{\succ(i,j,d)}}\xi(\kappa v)\,dv$, and $\mathcal{T}_{(1,k,1)_-}(\xi)=\xi(\kappa)$. Note that in \eqref{eq:TW third} we have already zeroed out both $v_{1,2}^2$ and $v_{1,2}^1$, hence for $k>2$ we will pass from $\mathcal{T}_{(2,3,1)}$ directly to $\mathcal{T}_{(k-2,k,1)}$ (skipping $\mathcal{T}_{(1,2,1)}$).
We use convolutions by Schwartz functions on $X_{j,i+1}^d$ as above, and obtain identities similar to \eqref{eq:T to T} (with $\LeviZ=I_{kc}$).

We can now deduce the continuity of $\mathcal{T}$ inductively. The map
$V_{\rho_c(\tau)}\otimes\mathcal{S}(X_{j,i+1}^{d})\to V_{\rho_c(\tau)}$ defined by \eqref{eq:convolution maps 1}, the map
\eqref{eq:convolution maps 2} and the similar convolution maps applied in \eqref{eq:TW third} are all continuous and open, and by \cite{DM} they are also surjective. Since the last map $\mathcal{T}_{(1,k,1)_-}$ is clearly continuous, we conclude that $\mathcal{T}$ is continuous and hence a functional on
$W_{\psi}(\rho_c(\tau))\otimes\SUParts$.

For the last assertion, let $W=W_{\xi}$ and assume $\mathcal{T}(W_{\xi}\otimes\mu)=\xi'(\kappa)$, for some $\xi'$. Since the commutator of $V$ and $\diag(\GL_l,I_{kc-l})$ is trivial, and for $1\leq i<j$ the commutator of $X_{j,i+1}$ and $\diag(\GL_l,I_{kc-l})$ is trivial, we get that
$\mathcal{T}(\LeviA\cdot W_{\xi}\otimes\mu)=\LeviA\cdot\xi'(\kappa)$, for the same $\xi'$, for all $a\in\GL_l$. It remains to observe that
${}^{\kappa}\LeviA=\LeviA'$ and that the function $h\mapsto |\det{h}|^{(-kc+kl+c-l)/2}\xi(\diag(h,I_{k(c-l)}))$, $h\in\GL_{kl}$, belongs to the space of $W_{\psi}(\rho_l(\tau))$.
\end{proof}
Consider the following integral, defined for a matrix coefficient $\omega_l$ of $\sigma^{\vee}$,
$W\in \chi_{\pi}W_{\psi}(\rho_c(\tau))$, $\phi\in\mathcal{S}(\Mat_l)$ and $\mu\in\SUParts$ by
\begin{align}\label{eq:mult inner start 2}
\ZInnerInner(s,\omega_l,W,\phi,\mu)=\int\limits_{\GL_l}\int\limits_{U_R}\omega_l(a)W(\LeviZ\LeviA)\mu(z)\phi(a)|\det{a}|^{s-(kc-c-l+1)/2}\,dz\,da.
\end{align}
Let $V_{\sigma}$ be the space of $\sigma$.
\begin{proposition}\label{proposition:inner 2}
Let $\sigma\in\Irr(\GL_l)$, $\tau\in\IrrGenUni(\GL_k)$, and if $F$ is non-archimedean assume $\sigma$ is unramified.
The integral~\eqref{eq:mult inner start 2} is absolutely convergent in $\Real(s)\gg0$ independently of the data, and admits meromorphic continuation which is continuous as a form on $V_{\sigma}\times V_{\sigma}^{\vee}\times \chi_{\pi}W_{\psi}(\rho_c(\tau))\times\mathcal{S}(\Mat_l)\times\SUParts$. For each $s$ there are data
$(\omega_l,W,\phi,\mu)$ with $\phi\in\SchwartzC(\Mat_l)$ and $\mu\in\SUPartsc$ such that
\begin{align}\label{eq:InnerInner pole}
\ZInnerInner(s,\omega_l,W,\phi,\mu)/L(s,\sigma^{\vee}\times\chi_{\pi}\tau)
\end{align}
is nonzero at $s$.
\end{proposition}
\begin{proof}
For convergence, first recall that over a non-archimedean field we can bound $|\omega_l(a)|$ using the central exponents of $\sigma^{\vee}$, and over archimedean fields
we have $|\omega_l(a)|\leq C||a||^D$ for some constants $C,D>0$ such that $D$ depends only on $\sigma^{\vee}$, where $||\cdot||$ is a norm on $\GL_l$ (see \cite[\S~2.A.2]{Wal88}).
Then note that $\LeviZ\LeviA\in\diag(\GL_c,I_{(k-1)c})$ and thus to bound $|W(\LeviZ\LeviA)|$ we can use Lemma~\ref{lemma:main convergence}.

Let $\mathcal{T}$ be given by Lemma~\ref{lemma:reduce c to l} for the $(k,c)$ model $\chi_{\pi}W_{\psi}(\rho_c(\tau))=W_{\psi}(\rho_c(\chi_{\pi}\tau))$, and let $W_l\in W_{\psi}(\rho_l(\chi_{\pi}\tau))$ be such that
$\mathcal{T}(\LeviA\cdot W\otimes\mu)=|\det{a}|^{(kc-kl-c+l)/2}W_l(\LeviA')$. In $\Real{s}\gg0$,
\begin{align*}
\ZInnerInner(s,\omega_l,W,\phi,\mu)=\int\limits_{\GL_l}\omega_l(a)W_l(\LeviA')\phi(a)|\det{a}|^{s-(kl-2l+1)/2}\,da.
\end{align*}
This integral is described in Appendix~\ref{appendix:GL2 poles} (recall that $k>1$).
It is absolutely convergent in $\Real(s)\gg0$ independently of the data, and admits meromorphic continuation
which is continuous in $\omega_l$, $W_l$ and $\phi$. Moreover for each $s\in\C$, for suitable data $(\omega_l,W_l,\phi)$ the quotient
$Z(s,\omega_l,W_l,\phi)/L(s,\sigma^{\vee}\times\chi_{\pi}\tau)$ is nonzero. This follows from Theorem~\ref{apptheorem:GL2 poles} if $s$ is a pole of $L(s,\sigma^{\vee}\times\chi_{\pi}\tau)$
and from Theorem~\ref{theorem:gcd} otherwise. (Note that in the appendix, the Schwartz function $\phi$ is ``absorbed" into $W_l$.)

Now the remaining assertions for the integral~\eqref{eq:mult inner start 2} follow at once, using Lemma~\ref{lemma:reduce c to l}, and note that by continuity
we can assume that $\phi$ and $\mu$ are compactly supported.
\end{proof}

We return to the integral $Z(s,\omega,f)$. Fix $T_G$ and $B_G$ such that $\embedding(T_G,T_G)<T_H$ and
$\embedding(B_G,B_G)<B_H$ (this is possible), and recall that $\embedding(K_G,K_G)<K_H$.
Let $\pi\in\Irr(G)$ and assume that $\pi$ is the Langlands quotient of $\sigma\rtimes\pi'=\Ind_{R}^G((\sigma\otimes\pi')\circ i_{M_R})$ with the notation of \eqref{LQpi}, where $R<G$ is a standard maximal parabolic subgroup with $M_R\cong\GL_l\times \mathcal{G}_{c-2l}$ and $l>0$.
When $c=2l$ and $G\ne\Sp_c$ we can choose the Siegel parabolic subgroup $R$ such that
$\embeddingR(M_R)<M_{P'}$. Set $\varepsilon=\sigma\otimes\pi'$ and $G'=\mathcal{G}_{c-2l}$.

The representation $\pi$ is isomorphic to the image of the convergent intertwining operator
$I:\Ind_{R}^G(\varepsilon\circ i_{M_R})\to \Ind_{R^-}^G(\varepsilon\circ i_{M_R})$ given by
\begin{align*}
I(\varphi)(g)=\int\limits_{U_R^-}\varphi(yg)\,dy.
\end{align*}
The convergence of this vector-valued integral is explained as follows. Regard $\varphi$ as
a function in the space of the representation $(\times_{i=1}^d\sigma_i)\rtimes\pi'$ induced from a standard parabolic subgroup $R_0$ contained in $R$, let $V_*$ be the space of
$(\otimes_{i=1}^d\sigma_i)\otimes\pi'$ and fix the standard pairing $<,>_*$ on $V_*\otimes V_*^{\vee}$. Then, for any $K_G\cap M_{R_0}$-finite $w\in V_*^{\vee}$, the integrals $\int_{U_{R_0}^-}|<\varphi(y),w>_*|\,dy$ are finite. See Waldspurger \cite[\S~IV.1]{W} and Wallach \cite[\S~5]{Wal88}. As known to experts, in the archimedean case this convergence can be extended to (smooth) $w$, see e.g., Knapp and Stein \cite{KS1} who studied $L^2$ convergence of intertwining operators, which is more difficult. We provide a quick proof of convergence, for any $w$, in Lemma~\ref{lemma:intconvergence} below, because we need this level of generality.

The representation $\pi^{\vee}$ is the Langlands quotient of $\Ind_{R^-}^G(\varepsilon^{\vee}\circ i_{M_R})$ (also $\pi^{\vee}\subset\Ind_{R}^G(\varepsilon^{\vee}\circ i_{M_R})$), and
is isomorphic to the image of the convergent intertwining operator
\begin{align}\label{eq:Iphee}
I^{\vee}:\Ind_{R^-}^G(\varepsilon^{\vee}\circ i_{M_R})\to \Ind_{R}^G(\varepsilon^{\vee}\circ i_{M_R}),\qquad
I^{\vee}(\varphi^{\vee})(g)=\int\limits_{U_R}\varphi^{\vee}(zg)\,dz.
\end{align}

Denote the canonical pairing on $\varepsilon\otimes\varepsilon^{\vee}$ by $\langle,\rangle$.
Let $\varphi^{\vee}$ and $\varphi$ be vectors in the spaces of $\Ind_{R^-}^G(\varepsilon^{\vee}\circ i_{M_R})$ and $\Ind_{R}^G(\varepsilon\circ i_{M_R})$, resp.
Then, the function
\begin{align}\label{eq:matrix coefficient omega formula}
\omega_{\varphi,\varphi^{\vee}}(g)=\int\limits_{M_R\backslash G}\langle\varphi(g_0),\varphi^{\vee}(g_0g)\rangle\, dg_0
=\int\limits_{R^-\backslash G}\int\limits_{U_R^-}\langle\varphi(yo),\varphi^{\vee}(og)\rangle\, dy\, do\qquad (g\in G)
\end{align}
is a matrix coefficient of $\pi^{\vee}$. Note that this integral is absolutely convergent because $R^-\backslash G$ is compact and $I$ is convergent,
where over archimedean fields we use Lemma~\ref{lemma:intconvergence}.

Substituting \eqref{eq:matrix coefficient omega formula} into $Z(s,f,\omega_{\varphi,\varphi^{\vee}})$ and using Lemma~\ref{corollary:inv of U on g giota} we obtain
\begin{align}\label{eq:integral 1}
\int\limits_{C_G^{\circforgspin}G^{\triangle}\backslash G\times G}
\int\limits_{M_R\backslash G}\int\limits_{U_0}\langle\varphi(g_0\dintegrallocalVarG_1),\varphi^{\vee}(g_0\dintegrallocalVarG_2)\rangle f(s,\delta u_0
\embeddingL(\dintegrallocalVarG_1)\Es{\dintegrallocalVarG_2})\,\psi_U(u_0)\,du_0\,dg_0\,\dintegrallocal.
\end{align}
Here $G^{\triangle}$ is the diagonal embedding of $G$ in $G\times G$.
We regard the $dg_0$-integral in \eqref{eq:integral 1} as an integral over
$M_R^{\triangle}\backslash G^{\triangle}$, collapse it into the $\dintegrallocal$-integral then rewrite the integration
domain using
\begin{align}\label{eq:collapsing formula}
\int\limits_{C_G^{\circforgspin}M_R^{\triangle}\backslash G\times G}
=\int\limits_{R\times R^-\backslash G\times G}\,\int\limits_{C_G^{\circforgspin}\backslash R}\,\int\limits_{U_R^-}.
\end{align}
The integral~\eqref{eq:integral 1} becomes
\begin{align}\label{eq:integral 1.5}
&\int\limits_{R\times R^-\backslash G\times G}\,
\int\limits_{U_R}\int\limits_{C_G^{\circforgspin}\backslash M_R}\int\limits_{U_R^-}\int\limits_{U_0}
\delta_{R^-}^{-1/2}(m)
\langle\varphi(\dintegrallocalVarG_1),\varepsilon^{\vee}(m)\varphi^{\vee}(\dintegrallocalVarG_2)\rangle\\&
 \quad f(s,\delta u_0
 \embeddingL(z\dintegrallocalVarG_1)
 \Es{ym\dintegrallocalVarG_2}) \,\psi_U(u_0)\,du_0\,dy\,dm\,dz\,\dintegrallocal.\notag
\end{align}
By Lemma~\ref{corollary:inv of U on g giota} we can replace $\embeddingL(z)$ with $\embeddingR({}^{\iota}z^{-1})=\Es{z^{-1}}$ and obtain (after changing $z\mapsto z^{-1}$)
\begin{align}\label{eq:mult 2 start}
&\int\limits_{R\times R^-\backslash G\times G}\,
\int\limits_{U_R}\int\limits_{C_G^{\circforgspin}\backslash M_R}\int\limits_{U_R^-}\int\limits_{U_0}
\delta_{R^-}^{-1/2}(m)
\langle\varphi(\dintegrallocalVarG_1),\varepsilon^{\vee}(m)\varphi^{\vee}(\dintegrallocalVarG_2)\rangle\\&
 \quad f(s,\delta u_0
 \Es{zym\dintegrallocalVarG_2} \embeddingL(\dintegrallocalVarG_1)) \,\psi_U(u_0)\,du_0\,dy\,dm\,dz\,\dintegrallocal.\notag
\end{align}
The passage from $Z(s,\omega_{\varphi,\varphi^{\vee}},f)$ to \eqref{eq:mult 2 start} in $\Real(s)\gg0$ will be justified by Fubini's theorem:
see Lemma~\ref{Brooks convergence} and the proof of Corollary~\ref{corollary:main poles}, below.

For $m\in M_R$ denote $\omega^{e_G}_{\varphi,\varphi^{\vee}}(m)=\langle\varphi(e_G),\varepsilon^{\vee}(m)\varphi^{\vee}(e_G)\rangle$, which is a matrix coefficient of $\varepsilon^{\vee}$.
We write the inner $du_0\,dy\,dm\,dz$-integral of \eqref{eq:mult 2 start} in the form
\begin{align}\label{eq:integral 1.7}
&Z^1(s,\omega^{e_G}_{\varphi,\varphi^{\vee}},f)=\int\limits_{U_R}\int\limits_{C_G^{\circforgspin}\backslash M_R}\int\limits_{U_R^-}\int\limits_{U_0}
\delta_{R^-}^{-1/2}(m)\omega^{e_G}_{\varphi,\varphi^{\vee}}(m)f(s,\delta u_0\Es{zym})\,\psi_U(u_0)\,du_0\,dy\,dm\,dz.
\end{align}

Fix $T_{G'}$ to be the projection of $i_{M_R}(T_G\cap i_{M_R}^{-1}(I_l,G'))$ into $G'$ and fix $B_{G'}$ similarly. Let $R'$ denote the standard Siegel parabolic subgroup of $G'$ such that $\embeddingR({}^{\iota}\,(i_{M_R}^{-1}(I_l,M_{R'}))\,)<M_{P'}$.
\begin{proposition}\label{proposition:integral over the open cell}
Let $\pi\in\Irr(G)$ and $\tau\in\IrrGen(\GL_k)$ and assume that $\pi$ is written with the notation of \eqref{LQpi}. When the integral $Z^1(s,\omega^{e_G}_{\varphi,\varphi^{\vee}},f)$ is absolutely convergent, it equals
\begin{align}\label{eq:mult 6 start}
&\int\limits_{U_R}\int\limits_{C_G^{\circforgspin}\backslash U_{R'}{R'}^-}\int\limits_{U_R^-}\int\limits_{U_0}\int\limits_{\GL_l}\delta_{R}^{-1/2}(a_R)|\det{a}|^{(-k+1)c}\omega^{e_G}_{\varphi,\varphi^{\vee}}(m)
\\& f(s,i_{M_P}^{-1}(\LeviZ\LeviA ,1)\delta_0\Levi{\Es{y}}\Es{m_{g'}}\BigUni(u_0,z,a,y,g'))\,\psi_U(u_0)\,da\,du_0\,dy\,dg'\,dz.\notag
\end{align}
Here $m_{g'}=i_{M_R}^{-1}(I_l,g_0')$ where $g_0'\in M_{R'}$ is the projection of $g'$ into $M_{R'}$, and
\begin{align*}
\BigUni(U_0,U_R,\GL_l,U_R^-,C_G^{\circforgspin}\backslash U_{R'}{R'}^-)
\end{align*}
is an open subset of $U_{P'}$ with the following properties:
\begin{enumerate}[leftmargin=*]
\item Over a non-archimedean (resp., archimedean) field the map $\BigUni:(u_0,z,a,y,g')\to \BigUni(u_0,z,a,y,g')$ is a homeomorphism (resp., diffeomorphism).
\item If $g'$ is confined to a compact subset of $C_G^{\circforgspin}\backslash U_{R'}{R'}^-$, the map $\BigUni$ restricts to a homeomorphism (resp., diffeomorphism)
of $U_0\times U_R\times \GL_l\times U_R^-$.
\end{enumerate}
\end{proposition}
\begin{proof}
We start with the integral $Z^1(s,\omega^{e_G}_{\varphi,\varphi^{\vee}},f)$. Since $\Es{z}\in \mathrm{St}_{M_Q}(\psi_U)$, we have ${}^{\Es{z}^{-1}}u_0\in U$. After changing variables in $u_0$ and using the left $\psi_k$-equivariance of $f$ we obtain
\begin{align}\label{eq:mult 3 start}
&\int\limits_{U_R}\int\limits_{C_G^{\circforgspin}\backslash M_R}\int\limits_{U_R^-}\int\limits_{U_0}
\delta_{R}^{1/2}(m)\omega^{e_G}_{\varphi,\varphi^{\vee}}(m)f(s,\delta \Es{z}u_0\Es{ym})\,\psi_U(u_0)\,du_0\,dy\,dm\,dz.
\end{align}

Recall that $\delta=\delta_0\delta_1$, ${}^{\delta_0}U_{P'}^-=U_P$ and $\delta_1\in i_{M_Q}^{-1}(I_c,\ldots,I_c,\mathcal{G}_{2c})\cap U_{P'}$. Since $\Es{z}=\Levi{\Es{z}}\Unim{\Es{z}}$,
\begin{align*}
f(s,\delta\Es{z})=f(s,\delta_0\Es{z}{}^{\Es{z}^{-1}}\delta_1)=f(s,\delta_0\Levi{\Es{z}}{}^{\Es{z}^{-1}}\delta_1)=
f(s,i_{M_P}^{-1}(\LeviZ,1)\delta_0\Uni{{}^{\Es{z}^{-1}}\delta_1})
\end{align*}
(see \eqref{eq:LeviZ}). Note that $\Uni{{}^{\Es{z}^{-1}}\delta_1}\in i_{M_Q}^{-1}(I_c,\ldots,I_c,\mathcal{G}_{2c})\cap U_{P'}$.

Write $m=a_Rg_R'$ where $i_{M_R}(a_R)=a\in\GL_l$ and $i_{M_R}(g_R')=g'\in G'$.
We conjugate $\Es{y}$ and $u_0$ by $\Es{a_R}$ on the right, this changes the measures $dy$ and $du_0$ and we obtain
\begin{align}\label{eq:mult 4 start}
&\int\limits_{U_R}\int\limits_{C_G^{\circforgspin}\backslash G'}\int\limits_{U_R^-}\int\limits_{U_0}\int\limits_{\GL_l}
\delta_{R}^{-1/2}(a_R)|\det{a}|^{(-k+1)c}
 \omega^{e_G}_{\varphi,\varphi^{\vee}}(m)\\&\quad f(s,i_{M_P}^{-1}(\LeviZ,1)\delta_0\,\Uni{{}^{\Es{z}^{-1}}\delta_1}\,\Es{a_R}u_0\Es{y}\Es{g_R'})\,\psi_U(u_0)\,da\,du_0\,dy\,dg'\,dz.\notag
\end{align}

We have $\Es{a_R}\in M_{P'}$ and $i_{M_P}({}^{\delta_0}\Es{a_R})=\LeviA$. Note that for $H=\GSpin_{2kc}$, using \cite[Lemma~6.2.5]{HS} we see that the projection of $i_{M_P}({}^{\delta_0}\Es{a_R})$ into $\mathcal{G}_0$ is trivial because $i_{M_{P'}}(\Es{a_R})=(a^*,\det a)$. (Further similar conjugations will
use this lemma implicitly.) Denote
\begin{align*}
\BigUni(u_0,z,a)=({}^{\Es{a_R}^{-1}}\Uni{{}^{\Es{z}^{-1}}\delta_1})\,u_0\in U_{P'}.
\end{align*}
The integral~\eqref{eq:mult 4 start} equals
\begin{align}\label{eq:mult 5 start}
&\int\limits_{U_R}\int\limits_{C_G^{\circforgspin}\backslash G'}\int\limits_{U_R^-}\int\limits_{U_0}\int\limits_{\GL_l}
\delta_{R}^{-1/2}(a_R)|\det{a}|^{(-k+1)c}
 \omega^{e_G}_{\varphi,\varphi^{\vee}}(m)\\&\quad f(s,i_{M_P}^{-1}(\LeviZ\LeviA,1)\delta_0\BigUni(u_0,z,a)\Es{y}\Es{g_R'})\,\psi_U(u_0)\,d(\cdots).\nonumber
\end{align}

Now we proceed with $\Es{y}$ and $\Es{g_R'}$. First, $\Es{y}=\Levi{\Es{y}}\Uni{\Es{y}}$ and we let
\begin{align*}
\BigUni(u_0,z,a,y)=({}^{\Levi{\Es{y}}^{-1}}\BigUni(u_0,z,a))\,\Uni{\Es{y}}\in U_{P'}.
\end{align*}
We turn to $\Es{g_R'}$. It is convenient at this point to assume $g'\in U_{R'}{R'}^-$. After removing a set of measure zero, which does not
change the integral, this assumption is justified. We can then write $\Es{g_R'}=h_-h_+$ where $h_-\in U_{P'}^-$ and $h_+\in P'$. Recall that
$g'_0$ denotes the projection of $g'$ into $M_{R'}$. Then $h_+=\Es{m_{g'}}\Uni{h_+}$. The projection of $\BigUni(u_0,z,a,y)$ into any of the root subgroups of $\embeddingR({}^{\iota}(i_{M_R}^{-1}(I_l,U_{R'}^-)))$ is $0$. Thus
${}^{(h_-)^{-1}}\BigUni(u_0,z,a,y)\in P'$. Moreover, for any $m\in M_{P}$, $f(s,m\delta_0h_-)=f(s,m\delta_0)$ because ${}^{\delta_0}U_{P'}^-=U_P$. We define
\begin{align*}
\BigUni(u_0,z,a,y,g')=({}^{(h_-\Es{m_{g'}})^{-1}}\BigUni(u_0,z,a,y))\Uni{h_+}\in P'.
\end{align*}
Then the integral~\eqref{eq:mult 5 start} becomes \eqref{eq:mult 6 start}, i.e.,
 \begin{align*}
&\int\limits_{U_R}\int\limits_{C_G^{\circforgspin}\backslash U_{R'}{R'}^-}\int\limits_{U_R^-}\int\limits_{U_0}\int\limits_{\GL_l}\delta_{R}^{-1/2}(a_R)|\det{a}|^{(-k+1)c}\omega^{e_G}_{\varphi,\varphi^{\vee}}(m)
\\&\quad f(s,i_{M_P}^{-1}(\LeviZ\LeviA ,1)\delta_0\Levi{\Es{y}}\Es{m_{g'}}\BigUni(u_0,z,a,y,g'))\,\psi_U(u_0)\,da\,du_0\,dy\,dg'\,dz.\notag
\end{align*}

The statement on the properties of $\BigUni(U_0,U_R,\GL_l,U_R^-,C_G^{\circforgspin}\backslash U_{R'}{R'}^-)$ can be verified by direct matrix multiplication, but already
follows from \cite[Lemma~7.7]{Yamana} since the projection of $\embedding(G,G)$ into
$i_{M_Q}^{-1}(I_c,\ldots,I_c,\mathcal{G}_{2c})$
is compatible with the construction for $k=1$.
\end{proof}
We introduce the following subspace of sections, which we use for the integral \eqref{eq:integral 1.7}.
For any $W\in \chi_{\pi}W_{\psi}(\rho_c(\tau))$ and $\Phi\in\SchwartzC(U_{P'})$, let $f_{W,\Phi}\in V(\tau,c)$ be the section whose support is contained in $PU_P^-$, and
such that for any $b\in M_P$ with $i_{M_P}(b)\in\GL_{kc}$, and any ${}^{\delta_0}u\in U_P^-$,
\begin{align*}
f_{W,\Phi}(s,b\,{}^{\delta_0}u)=|\det{i_{M_P}(b)}|^{s-1/2}\delta_P^{1/2}(b)W(i_{M_P}(b))\Phi(u).
\end{align*}
\begin{proposition}\label{proposition:Pole Inner Z1}
Suppose that $\pi\in\IrrUnrForPoles(G)$ and $\tau\in\IrrGenUni(\GL_k)$. We keep the notation \eqref{LQpi}.
For any $s_0\in\C$, there exist data $\omega^{e_G}_{\varphi,\varphi^{\vee}}$ and $f\in V(\tau,c)$ such that
$Z^1(s,\omega^{e_G}_{\varphi,\varphi^{\vee}},f)$ is absolutely convergent in $\Real(s)\gg0$, admits meromorphic continuation, and as a meromorphic function
$Z^1(s,\omega^{e_G}_{\varphi,\varphi^{\vee}},f)/L(s,\sigma^{\vee}\times\chi_{\pi}\tau)$ is nonzero at $s_0$.
\end{proposition}
\begin{proof}
Let $W\in\chi_{\pi}W_{\psi}(\rho_c(\tau))$ be given.
By Proposition~\ref{proposition:integral over the open cell} we can choose, for any $\phi_0\in\SchwartzC(U_0\times U_P^-\times C_G^{\circforgspin}\backslash U_{R'}{R'}^-)$, $\phi\in\SchwartzC(\Mat_l)$ and $\mu\in\SUPartsc$, a function $\Phi\in\SchwartzC(U_{P'})$ such that
\begin{align*}
\Phi(\BigUni(u_0,z,a,y,g'))=\phi_0(u_0,y,g')\phi(a)\mu(z).
\end{align*}
Let $f=\delta_0^{-1}\cdot f_{W,\Phi}\in V(\tau,c)$.
For this choice of $\Phi$, the absolute convergence of $Z^1(s,\omega^{e_G}_{\varphi,\varphi^{\vee}},f)$ in $\Real(s)\gg0$ is clear, from \eqref{eq:mult 6 start} and
the proof of Proposition~\ref{proposition:integral over the open cell}, and from the proof of the absolute convergence of the $dz\,da$-integral which is of the form \eqref{eq:mult inner start 2}. Then in $\Real(s)\gg0$,
\begin{align*}
Z^1(s,\omega^{e_G}_{\varphi,\varphi^{\vee}},f)=&\int\limits_{C_G^{\circforgspin}\backslash U_{R'}{R'}^-}\int\limits_{U_R^-}\int\limits_{U_0}
\ZInnerInner(s,\omega^{e_G}_{\varphi,\varepsilon^{\vee}(i_{M_R}^{-1}(I_l,g'))\varphi^{\vee}},i_{M_P}({}^{\delta_0}(\Levi{\Es{y}}\Es{m_{g'}}))\cdot
W,\phi,\mu)\\\notag&\,|\det{i_{M_P}({}^{\delta_0}\Es{m_{g'}})}|^{s-1/2}\delta_P^{1/2}({}^{\delta_0}\Es{m_{g'}})\phi_0(u_0,y,g')\,\psi_U(u_0)\,du_0\,dy\,dg'.
\end{align*}
Observe that the function $a\mapsto\omega^{e_G}_{\varphi,\varphi^{\vee}}(i_{M_R}^{-1}(a,e_{G'}))$ on $\GL_l$ is a matrix coefficient of $\sigma^{\vee}$, and that if $F$ is non-archimedean, then our assumption on $\pi$ implies that $\sigma$ is unramified. By Proposition~\ref{proposition:inner 2} and since $\phi_0$ is compactly supported, $Z^1(s,\omega^{e_G}_{\varphi,\varphi^{\vee}},f)$ is a meromorphic function.
The result now follows from
Proposition~\ref{proposition:inner 2} since $\ZInnerInner$ is continuous as a form and we can take $\phi_0$ with arbitrarily small compact support near the identity.
\end{proof}

\begin{lemma}\label{Brooks convergence}
Assume $F$ is archimedean, $\pi\in\Irr(G)$ and $\tau\in\IrrGen(\GL_k)$. Keep the notation \eqref{LQpi}.
The integrals \eqref{eq:mult 2 start} and $Z^1(s,\omega^{e_G}_{\varphi,\varphi^{\vee}},f)$
are absolutely convergent in $\Real(s)\gg0$ independently of the data, and are continuous functions of $\varphi$, $\varphi^{\vee}$ and $f$ in this domain.
\end{lemma}
\begin{proof}
We start with the integral~\eqref{eq:integral 1.7} and bound
\begin{align*}
&\int\limits_{U_R}\int\limits_{C_G^{\circforgspin}\backslash M_R}\int\limits_{U_R^-}\int\limits_{U_0}
\left|\delta_{R^-}^{-1/2}(m)\omega^{e_G}_{\varphi,\varphi^{\vee}}(m)f(s,\delta u_0\Es{zym})\right|\,du_0\,dy\,dm\,dz.
\end{align*}
First we write $z$ using the Iwasawa decomposition of $G$ with respect to $R^-$, namely $z=k_zm_zy_z$ where $k_z\in K_G$, $m_z\in M_R$ and $y_z\in U_R^-$.
Changing variables in $y$ and $m$, we obtain
\begin{align*}
&\int\limits_{U_R}\int\limits_{C_G^{\circforgspin}\backslash M_R}\int\limits_{U_R^-}\int\limits_{U_0}
\left|\delta_{R^-}^{-1/2}(m)\delta_{R^-}^{1/2}(m_z^{-1})\omega^{e_G}_{\varphi,\varphi^{\vee}}(m_z^{-1}m)f(s,\delta u_0\Es{k_zym})\right|\,du_0\,dy\,dm\,dz.
\end{align*}
Denote $m=a_Rg_R'$ with $i_{M_R}(a_R)=a\in\GL_l$ and $i_{M_R}(g_R')\in G'$. Write
$g_R'$ using the Cartan decomposition $g_R'=k't_R'k''$ where $i_{M_R}(k'),i_{M_R}(k'')\in K_{G'}$ and
$i_{M_{R'}}\circ i_{M_R}(t_R')=(t,t_0)$ for $t\in T_{\GL_{n-l}}$ and $t_0\in\mathcal{G}_0$.
Since we integrate over $C_G^{\circforgspin}\backslash M_R$, we can already
take $t_0=1$.

Note that if $G'\notin\{\SO_2,\GSpin_2\}$, then we can suppose that $|\det t|\leq1$.

Since $k'$ normalizes $U_R^-$, and by Lemma~\ref{corollary:inv of U on g giota} applied to $\Es{k_zk'}$, we obtain
\begin{align}\label{eq:int before split}
&\int\limits_{U_R}\int\limits_{\GL_l}\int\limits_{K_{G'}}\int\limits_{T_{\GL_{n-l}}}\int\limits_{K_{G'}}
\int\limits_{U_R^-}\int\limits_{U_0}
\left|\delta_{R^-}^{-1/2}(m)\delta_{R^-}^{1/2}(m_z^{-1})\omega^{e_G}_{\varphi,\varphi^{\vee}}(m_z^{-1}m)f(s,\delta u_0\Es{ya_Rt_R'k''}
k'_{z}
)\right|\\&\,du_0\,dy\,dk'\,dt\,dk''\,da\,dz.\nonumber
\end{align}
Here $k'_{z}=\embeddingL((k_zk')^{-1})$.

Our next task is to split the integral \eqref{eq:int before split} into summands, such that in each summand both $|\det{a}|\leq1$ and $|\det{t}|\leq 1$ (the latter already holds when $G'\notin\{\SO_2,\GSpin_2\}$). To obtain this we will use conjugations by Weyl elements which invert the projections of $\Es{a_R}$ and $\Es{t_R'}$ into $M_{P'}$. These conjugations will change $\Es{y}$ as well. First observe that
\begin{align*}
i_{M_{P'}}(\Es{a_R})=(\diag(I_{kc-l},a^*),\det{a}),\qquad
i_{M_{P'}}(\Es{t_R'})=(\diag(I_{kc-n},t^*,I_l),\det{t}).
\end{align*}
Here $\det{a}$ and $\det{t}$ are omitted if $G\ne\GSpin_c$.

Consider the cases when either $G=\Sp_{c}$ or $2\nmid c$ or both $c$ and $l$ are even. Then there is $\WeylElement_0\in W_G$ of minimal length such that
$i_{M_R}({}^{\WeylElement_0}a_R)=(a^*,\det{a})$ and ${}^{\WeylElement_0}t_R'=t_R'$. Put $w=\embedding(w_0,{}^{\iota}w_0)\in H$ for a representative $w_0\in G$ for $\WeylElement_0$.
If $G=\SO_c$ (resp., $\GSpin_c$) and $2|c$ and $2\nmid l$, then $\WeylElement_0$ with the above properties will be in $W_{\Orth_c}$ (resp., $W_{\GPin_c}$) and the representative $w_0$ will be in $\Orth_c$ (resp., $\GPin_c$); nonetheless, $\embedding$ extends to a map $\Orth_c\times\Orth_c\to \Orth_{2kc}$ (resp., $\GPin_c\times\GPin_c\to\GPin_{2kc}$) and
$w=\embedding(w_0,{}^{\iota}w_0)\in H$ because it represents an even permutation. Then in all cases
\begin{align*}
i_{M_{P'}}({}^{w}\Es{a_R})=(\diag(I_{kc-l},a),1),\qquad {}^{w}\Es{t_R'}=\Es{t_R'}.
\end{align*}

Similarly if $G'\in\{\SO_2,\GSpin_2\}$, take $w'=\embedding(w'_0,{}^{\iota}(w'_0))\in H$ with a representative $w'_0$ for an element $\WeylElement'_0$ of $W_{\Orth_c}$ or
$W_{\GPin_c}$, such that
\begin{align*}
{}^{w'}\Es{a_R}=\Es{a_R},\qquad i_{M_{P'}}({}^{w'}\Es{t_R'})=(\diag(I_{kc-n},t,I_l),1).
\end{align*}

Define for $X<\GL_e$ with $e\geq1$, $X^0=\{x\in X:|\det{x}|\leq1\}$ and $X^1=\{x\in X:|\det x|>1\}$.
We then write the integral~\eqref{eq:int before split} as a sum of integrals $\mathcal{I}(i,j)$, $i,j\in\{0,1\}$, with
$\GL_l^i$ and $T_{\GL_{n-l}}^j$ replacing $\GL_l$ and $T_{\GL_{n-l}}$, but if $G'\notin\{\SO_2,\GSpin_2\}$ we only have $\mathcal{I}(0,0)$ and $\mathcal{I}(1,0)$.

To each $\mathcal{I}(i,j)$ we apply Lemma~\ref{corollary:inv of U on g giota} with $\widetilde{w}=w^i(w')^j\in H$. Even if $\WeylElement_0\notin W_G$ or $\WeylElement'_0\notin W_G$, Lemma~\ref{corollary:inv of U on g giota} is still applicable to $\widetilde{w}$ because $\widetilde{w}\in\mathrm{St}_{M_Q}(\psi_U)$ and $i_{M_P}({}^{\delta}\widetilde{w})$ belongs to the stabilizer of $\psi_k$ in $\GL_{kc}$ (this is not a coincidence, the classical doubling method was defined for $\Orth_c$, see \cite{PSR}). We obtain integrals of the form
\begin{align*}
\mathcal{I}(i,j)=
&\int\limits_{U_R}
\int\limits_{\GL_l^i}\int\limits_{K_{G'}}\int\limits_{T_{\GL_{n-l}}^j}\int\limits_{K_{G'}}
\int\limits_{U_R^-}\int\limits_{U_0}\\&\notag
\left|\delta_{R^-}^{-1/2}(m)\delta_{R^-}^{1/2}(m_z^{-1})\omega^{e_G}_{\varphi,\varphi^{\vee}}(m_z^{-1}m)f(s,\delta u_0{}^{\widetilde{w}}(\Es{ya_Rt_R'})\widetilde{w}\Es{k''}k'_z
)\right|\,du_0\,dy\,dk'\,dt\,dk''\,da\,dz.
\end{align*}
For $i=0$, $\widetilde{w}$ normalizes $\Es{U_R^-}$ and we proceed as in the proof of Proposition~\ref{proposition:integral over the open cell}.
That is, first conjugate $y\mapsto {}^{m}y$ (multiplying $dy$ by $\delta_{R^-}(m)$), conjugate $\delta_1u_0$ by ${}^{\widetilde{w}}(\Es{a_Rt_R'})$ on the right, then
conjugate by $\Levi{\Es{y}}$ on the right. If $i=1$, then ${}^{\widetilde{w}}\Es{U_R^-}=\Es{U_R}$ and again we continue as in that proposition (for $U_R$), conjugate ${}^{\widetilde{w}}\Es{y}$ first, then ${}^{\widetilde{w}}(\Es{a_Rt_R'})$. Define
\begin{align*}
\begin{cases}
\BigLevi(y,a,t)={}^{\delta_0\widetilde{w}}(\Es{a_Rt_R'}) {}^{\delta_0}\Levi{\Es{y}},\quad \BigUni(u_0,y,a,t)={}^{\Levi{\Es{y}}^{-1}}({}^{{}^{\widetilde{w}}(\Es{a_Rt_R'})^{-1}}(\delta_1u_0))\Uni{\Es{y}}&i=0,\\
\BigLevi(y,a,t)=i_{M_P}^{-1}(\LeviY,1)({}^{\delta_0\widetilde{w}}(\Es{a_Rt_R'})),\quad
\BigUni(u_0,y,a,t)={}^{{}^{\widetilde{w}}(\Es{a_Rt_R'})^{-1}}({}^{\Es{{}^{\widetilde{w}}y}^{-1}}(\delta_1u_0))
&i=1.
\end{cases}
\end{align*}
When $i=1$, the coordinates of $y$ in $\BigUni(u_0,y,a,t)$ do depend on $t_R'$ and we remove this dependency by a change of variables
which multiplies $dy$ by $|\det t|^{-1+2j}$ (for $\Sp_c$, with the notation \eqref{eq:z for Sp example} the change is to the coordinates of $\LeviY$ corresponding to $z_1$).
After this change, $\LeviY$ depends on $t$ but remains a unipotent element.

Observe that $i_{M_P}({}^{\delta_0}\Levi{\Es{y}})\in \diag(\GL_c,I_{k(c-1)})$ and
\begin{align*}
i_{M_P}({}^{\delta_0\widetilde{w}}(\Es{a_Rt_R'}))=(\diag([a]_i,[t]_j,I_{kc-n}),(\det{a})^{i}(\det{t})^{j}),
\end{align*}
where we set $[a]_0=a$, $[a]_1=a^*$ and similarly $[t]_0=t$, $[t]_1=t^*$. The upshot is that $|\det[a]_i|\leq1$ for both $i$, $|\det[t]_0|\leq1$
and if $G'\in\{\SO_2,\GSpin_2\}$ then we also have $j=1$ and $|\det[t]_1|\leq1$. With these modifications
\begin{align*}
\mathcal{I}(i,j)=
&\int\limits_{U_R}\int\limits_{\GL_l^i}\int\limits_{K_{G'}}\int\limits_{T_{\GL_{n-l}}^j}\int\limits_{K_{G'}}\int\limits_{U_R^-}\int\limits_{U_0}
|\det[a]_i|^{(-k+1)c}|\det[t]_j|^{(-k+1)c-i}\\&
\left|\delta_{R^-}^{\tfrac12-i}(m)\delta_{R^-}^{1/2}(m_z^{-1})
\omega^{e_G}_{\varphi,\varphi^{\vee}}(m_z^{-1}m)
f(s,\BigLevi(y,a,t)\delta_0\BigUni(u_0,y,a,t)\widetilde{w}\Es{k''}k'_z)\right|\,d(\cdots).
\end{align*}
(The factors $|\cdots|^{(-k+1)c}$ appear because of the conjugations of $u_0$, and $i$ in the exponents is either $0$ or $1$.)

Write $\delta_0\BigUni(u_0,y,a,t)$ according to the Iwasawa decomposition $H=PK_H$:
\begin{align*}
\delta_0\BigUni(u_0,y,a,t)=v_{u_0,y,a,t}i_{M_P}^{-1}(b_{u_0,y,a,t},1)k_{u_0,y,a,t},\quad v_{u_0,y,a,t}\in N_H,\, b_{u_0,y,a,t}\in T_{\GL_{kc}},\,k_{u_0,y,a,t}\in K_H.
\end{align*}
According to the bound \eqref{eq:conv main N 1} of Lemma~\ref{lemma:main N invariance},
\begin{align*}
&|f(s,\BigLevi(y,a,t)\delta_0\BigUni(u_0,y,a,t)\widetilde{w}\Es{k''}k'_z)|\leq \delta_P^{1/2}({}^{\delta_0\widetilde{w}}(\Es{a_Rt_R'})i_{M_P}^{-1}(b_{u_0,y,a,t},1))\\&\qquad\cdot|\det [a]_i\cdot\det [t]_j|^{s-1/2-d_0}|\chi_{\pi}(\det a)^{i}\chi_{\pi}(\det t)^{j}|\cdot|\det b_{u_0,y,a,t}|^{s-1/2}||b_{u_0,y,a,t}||^{D_1}\gamma(f).
\end{align*}
Here $||\cdot||$ is a norm on $\GL_{kc}$; $d_0$ is the constant of the lemma, $D_1$ is a constant which depends on $d_0$, and $\gamma$ is a continuous seminorm on $V(\tau,c)$. Thus there are constants $M_1,M_2,M_3\in\Z$ which depend on the groups (through modulus characters), on $\chi_{\pi}$ and on $i$ and $j$, such that $\mathcal{I}(i,j)$ is bounded by
\begin{align*}
\int\limits_{\cdots}
\left|\delta_{R^-}^{1/2}(m_z^{-1})
\omega^{e_G}_{\varphi,\varphi^{\vee}}(m_z^{-1}m)\right|
|\det [a]_i|^{s+M_1}|\det [t]_j|^{s+M_2}|\det b_{u_0,y,a,t}|^{s+M_3}
||b_{u_0,y,a,t}||^{D_1}\gamma(f)\,d(\cdots).
\end{align*}

To proceed consider the $dz$-integral:
\begin{align}\label{eq:dz integral}
&\int\limits_{U_R}|\delta_{R^-}^{1/2}(m_z^{-1})\omega^{e_G}_{\varphi,\varphi^{\vee}}(m_z^{-1}m)|\,dz=
\int\limits_{U_R}\left|\langle\varphi(e_G),\varphi^{\vee}(m_z^{-1}m)\rangle\right|\,dz.
\end{align}
Observing that $z^{-1}=y_z^{-1}m_z^{-1}k_z^{-1}$, the proof of Lemma~\ref{lemma:intconvergence} applied to $I^{\vee}$ (see \eqref{eq:Iphee}) implies that the r.h.s.~ of \eqref{eq:dz integral} is bounded by $\beta(\varphi(e_G))\gamma_0(\varphi^{\vee})||m||^{D_2}$, where $\beta$ and $\gamma_0$ are continuous seminorms on the spaces of $\varepsilon$ and $\Ind_{R^-}^G(\varepsilon^{\vee}\circ i_{M_R})$ (resp.) and $D_2$ is a constant depending only on $\pi^{\vee}$.

Identifying $\BigUni(u_0,y,a,t)$ with a matrix in $\Mat_{kc}$, we let $||\BigUni(u_0,y,a,t)||_E$ denote its Euclidian norm. Observe that, because of the aforementioned changes of variables in $y$ and $u_0$,
there are positive constants $C_1$ and $C_2$ such that
\begin{align}\label{volume bounds}
||u_0||_E\leq C_1||\BigUni(u_0,y,a,t)||_E,\quad ||y||\leq C_2||\BigUni(u_0,y,a,t)||_E,\qquad\forall a\in\GL_l,\, t\in T_{\GL_{n-l}}.
\end{align}
In fact the growth of $||\BigUni(u_0,y,a,t)||_E$ is, essentially, dominated by $||\BigUni(u_0,y,I_l,I_{\GL_{n-l}})||_E$ because of the conjugation by $\widetilde{w}$.
We use the standard bounds for $|\det{b_{u_0,y,a,t}}|$ and $||b_{u_0,y,a,t}||$, see Soudry \cite[\S~7.3, Lemma~3]{Soudry} (for $\SO_{2kc}$, similar bounds hold for the other groups):
\begin{align*}
&(1+||\BigUni(u_0,y,a,t)||_E^2)^{-kc/2}\leq|\det b_{u_0,y,a,t}|\leq(1+||\BigUni(u_0,y,a,t)||_E^2)^{-1/2},\\
&||b_{u_0,y,a,t}||\leq (1+||\BigUni(u_0,y,a,t)||_E^2)^{D_3},
\end{align*}
for a constant $D_3$. It remains to consider, for some constant $M$,
\begin{align*}
&\int\limits_{\GL_l^i}\int\limits_{T_{\GL_{n-l}}^j}\int\limits_{U_R^-}\int\limits_{U_0}
||a||^{D_2}||t||^{D_2}|\det[a]_i|^{s+M_1}|\det[t]_j|^{s+M_2}(1+||\BigUni(u_0,y,a,t)||^2_E)^{-\tfrac12\Real(s)+M}\,d(\cdots).
\end{align*}
The multiple integral converges in $\Real(s)\gg0$ (use \eqref{volume bounds}), and the continuity of \eqref{eq:integral 1.7} in $\varphi$, $\varphi^{\vee}$ and $f$ is immediate.
Turning to \eqref{eq:mult 2 start}, since we can write the $\dintegrallocal$-integral of \eqref{eq:mult 2 start} over $K_G\times K_G$, the convergence and continuity results for
\eqref{eq:mult 2 start} follow as well, with the seminorm $\beta(\varphi(e_G))$ replaced by a continuous seminorm $\gamma_1$ on the space of
$\Ind_{R}^G(\varepsilon\circ i_{M_R})$ such that
$\beta(\varphi(\dintegrallocalVarG_1))\leq\gamma_1(\varphi)$ for all $\varphi$ and $\dintegrallocalVarG_1\in K_G$.
\end{proof}

\begin{corollary}\label{corollary:main poles}
Assume that $\pi\in\IrrUnrForPoles(G)$ and $\tau\in\IrrGenUni(\GL_k)$ and keep the notation \eqref{LQpi}.
Then for any $s\in\C$, there exist data $\omega$ and $f\in V(\tau,c)$ such that $Z(s,\omega,f)/L(s,\sigma^{\vee}\times\chi_{\pi}\tau)$ is nonzero at $s$.
\end{corollary}
\begin{proof}
By Proposition~\ref{proposition:Pole Inner Z1} there are data $f\in V(\tau,c)$, $\varphi(e_G)$ and $\varphi^{\vee}(e_G)$ such that the integral $Z^1(s,\omega^{e_G}_{\varphi,\varphi^{\vee}},f)$ is absolutely convergent in $\Real(s)\gg0$ and (its meromorphic continuation) is nonzero at $s_0$ when divided by
$L(s,\sigma^{\vee}\times\chi_{\pi}\tau)$.

For any $\phi^-\in \SchwartzC(U_R^-)$ we can define $\varphi_{\phi^-}$ in the space of $\Ind_{R}^G(\varepsilon\circ i_{M_R})$, whose support is contained in $RU_R^-$ and restricts to $\phi^-$ on $U_R^-$, and such that $\varphi_{\phi^-}(e_G)=\varphi(e_G)$. Define $\varphi_{\phi}^{\vee}$ in the space of $\Ind_{R^-}^G(\varepsilon^{\vee}\circ i_{M_R})$
similarly, for $\phi\in \SchwartzC(U_R)$ and $\varphi^{\vee}(e_G)$.

Write the $\dintegrallocal$-integral of \eqref{eq:mult 2 start} over $U_R^-\times U_R$. For the data $(f,\varphi_{\phi^-},\varphi_{\phi}^{\vee})$ we obtain
\begin{align}\label{eq:integral 1.5222}
&\int\limits_{U_R^-}\int\limits_{U_R}\,
Z^1(s,\omega^{e_G}_{\varphi,\varphi^{\vee}},\Es{\dintegrallocalVarG_2}\embeddingL(\dintegrallocalVarG_1)\cdot f)\phi^-(\dintegrallocalVarG_1)\phi(\dintegrallocalVarG_2)\,d\dintegrallocalVarG_2\,d\dintegrallocalVarG_1.
\end{align}

In the non-archimedean case we can choose $\phi^-$ and $\phi$ for which $\Es{\dintegrallocalVarG_2}\embeddingL(\dintegrallocalVarG_1)\cdot f=f$
when $\phi^-(\dintegrallocalVarG_1)\phi(\dintegrallocalVarG_2)\ne0$. Thus \eqref{eq:integral 1.5222} is equal to $Z^1(s,\omega^{e_G}_{\varphi,\varphi^{\vee}},f)$,
and is also absolutely convergent as a $du_0\,dy\,dm\,dz\,\dintegrallocal$-integral in $\Real(s)\gg0$, justifying the passage from $Z(s,\omega_{\varphi_{\phi^-},\varphi_{\phi}^{\vee}},f)$ to \eqref{eq:mult 2 start}. The result follows.

In the archimedean case, by Lemma~\ref{Brooks convergence}, in $\Real(s)\gg0$ the integrals
$Z(s,\omega_{\varphi_{\phi^-},\varphi_{\phi}^{\vee}},f)$, \eqref{eq:mult 2 start} and \eqref{eq:integral 1.5222} are all absolutely convergent and equal. Now by \cite{DM} applied to the action of $\Es{U_R}\embeddingL(U_R^-)$ on $V(s,\tau,c)$, there is a finite set of Schwartz functions $\phi^-_d$ and $\phi_d$, and
meromorphic sections $f_d$ (see \cite[Appendix~A]{CFK2022} for the definition of meromorphic sections) such that
\begin{align*}
\sum_dZ(s,\omega_{\varphi_{\phi^-_d},\varphi_{\phi_d}^{\vee}},f_d)=Z^1(s,\omega^{e_G}_{\varphi,\varphi^{\vee}},f).
\end{align*}
As described by Gourevitch \cite[Appendix~A]{CFK2022}, we can apply \cite{DM} in a way such that when $f\in V(\tau,c)$, each $f_d$ is also in $V(\tau,c)$, i.e., entire.
The proof is complete.
\end{proof}
\begin{proof}[Proof of Theorem~\ref{theorem:archimedean producing poles}]
If $L(s,\pi\times\tau)$ is holomorphic at $s$, then the theorem follows immediately from the (stronger) statement of Theorem~\ref{theorem:all local props} \eqref{it:nonzero}.
Hence it is enough to consider the poles and then by Theorem~\ref{theorem:L for unr temp and unitary is holomorphic in half plane} we can assume that $\pi$ is not tempered.

By Theorem~\ref{theorem:all local props} \eqref{it:cont} and by Cauchy's formula, the leading term in the Laurent expansion of $Z(s,\omega,f)$ is continuous.
Hence in the statement of Theorem~\ref{theorem:archimedean producing poles} we may drop the condition that $f$ is $K_H$-finite.
Now write $\pi$ as the Langlands quotient of $\sigma\rtimes\pi'$ as in \eqref{LQpi}. The result follows immediately from Corollaries~\ref{corollary:L of pi vs L of sigma in half plane} and \ref{corollary:main poles}.
\end{proof}

It remains to address the convergence of the intertwining operator in the archimedean case. We state a general result.
Let $\mathcal{G}$ be a reductive group as in \S~\ref{the groups}, defined over $\R$ or $\C$. Fix a maximal compact subgroup $K_{\mathcal{G}}$ relative to $T_{\mathcal{G}}$, i.e.,
$K_{\mathcal{G}}$ and $T_{\mathcal{G}}$ are orthogonal with respect to the Killing form. Let $R<\mathcal{G}$ be a standard parabolic subgroup, $A_R$ be the identity component of $C_{M_R}$, $X(M_R)$ (resp., $X(A_R)$) denote the group of $F$-rational characters of $M_R$ (resp., $A_R$), $\mathfrak{a}=\Hom(X(A_R),\R)$, $\mathfrak{a}^*=X(A_R)\otimes_{\Z}\R$ and let $\mathfrak{a}^*_{\C}$ be the complexification of $\mathfrak{a}^*$. Let $H_{M_R}:M_R\to\mathfrak{a}$ be the standard map defined by $e^{\langle\chi,H_{M_R}(m)\rangle}=|\chi(m)|$ for all $\chi\in X(M_R)$. For any $\sigma\in\Irr(M_R)$ and $\nu\in\mathfrak{a}^*_{\C}$,
denote the space of the representation $\mathrm{Ind}_R^{\mathcal{G}}(\sigma\otimes e^{<\nu,H_{M_R}(\cdot)>})$ by $V_R^{\mathcal{G}}(\nu,\sigma)$.
\begin{lemma}\label{lemma:intconvergence}
Assume that $\sigma\in\Irr(M_R)$ is tempered, and fix the inner product $<,>_{\sigma}$ on the space $V_{\sigma}$ of $\sigma$.
Let $\nu\in\mathfrak{a}^*_{\C}$ be such that $\Real(\nu)$ belongs to the positive Weyl chamber of $\mathfrak{a}^*$ with respect to $R$. Then, for any $\varphi\in V_R^{\mathcal{G}}(\nu,\sigma)$ and $w\in V_{\sigma}$,
\begin{align}\label{conv}
\int\limits_{U_R^-}|<\varphi(u),w>_{\sigma}|\,du<\infty.
\end{align}
Moreover, the map $(\varphi,w)\mapsto\int_{U_R^-}<\varphi(u),w>_{\sigma}\,du$ is continuous.
\end{lemma}
\begin{proof}
We combine the bound of Sun \cite{Sun2009} on smooth matrix coefficients of tempered representations with the standard proof
of convergence from \cite{Wal88}. By \cite[Theorem~1.2]{Sun2009}, there is a continuous seminorm $\beta$ on $V_{\sigma}$ such that for all $v,w\in V_{\sigma}$ and $m\in M_R$,
\begin{align*}
|<\sigma(m)v,w>|\leq \beta(w)\beta(v)\Xi(m).
\end{align*}
Here $\Xi$ is the basic spherical function of Harish--Chandra on $M_R$ (see e.g., \cite[\S~3.6.1]{Wal88}).

To check \eqref{conv} we write $u=r_um_uk_u$ where $r_u\in U_R$, $m_u\in M_R$ and $k_u\in K_{\mathcal{G}}$ according to the Iwasawa decomposition $\mathcal{G}=RK_{\mathcal{G}}$.
Let $\rho_R$ be half the sum of roots in $U_R$. Since
\begin{align*}
\varphi(u)=e^{<\nu+\rho_R,H_{M_R}(m_u)>}\sigma(m_u)\varphi(k_u),
\end{align*}
it follows that
\begin{align*}
|<\varphi(u),w>|&\leq e^{<\Real(\nu)+\rho_R,H_{M_R}(m_u)>}\beta(w)\beta(\varphi(k_u))\,\Xi(m_u).
\end{align*}
Let $\gamma$ be a continuous seminorm on $V_R^{\mathcal{G}}(\nu,\sigma)$ such that for all $\varphi$ and $k\in K_{\mathcal{G}}$, $\beta(\varphi(k))\leq \gamma(\varphi)$. Then
\begin{align*}
\int\limits_{U_R^-}|<\varphi(u),w>|\,du\leq\beta(w)\gamma(\varphi)\int\limits_{U_R^-}e^{<\Real(\nu)+\rho_R,H_{M_R}(m_u)>}\Xi(m_u)\,du.
\end{align*}
The convergence of the r.h.s.~follows from the proof of \cite[Lemma~5.3.1]{Wal88}, and the continuity is now clear.
\end{proof}

\section{Global theory}\label{Local theory}
We begin by describing the Eisenstein series which appears in the global integral, and prove that it is holomorphic in $\Real(s)\geq0$
when the associated cuspidal representation is not self-dual up to a twist. See Theorem~\ref{theorem:series}. Then we recall the definition of the global integral (from \cite{CFK2022}), and use Theorem~\ref{theorem:series} and the local result Theorem~\ref{theorem:archimedean producing poles} to show that the global $L$-functions are entire,
under certain additional local conditions. See Theorem~\ref{theorem:twisting to obtain entire L function}.
Finally in \S~\ref{section:Constructing the lift} we construct a weak functorial transfer.

\subsection{Notation}\label{global groups and notation}
Let $F$ be a number field with a ring of adeles $\A$, and denote the set of infinite places of $F$ by $S_{\infty}$. For a place $\nu$ of $F$, $F_{\nu}$ denotes the completion of $F$ at $\nu$. By a cuspidal representation we mean an irreducible cuspidal automorphic representation. Let $\psi$ be a nontrivial additive character of $F\backslash\A$. In analogy with the local setting (see \S~\ref{local groups and notation}), for an irreducible automorphic representation $\pi$ of $G(\A)$, when $G=\GSpin_c$
we let $\chi_{\pi}$ be the pullback of $\pi$ by $e_0^{\vee}$ viewed as a Hecke character of $\A^*$, and for $G\ne\GSpin_c$ we fix $\chi_{\pi}=1$.
For a Hecke character $\eta$ of $\A^*$ and an irreducible automorphic representation $\sigma$ of $\GL_k(\A)$, let $\eta\sigma(g)=\eta(\det g)\sigma(g)$ be the twist of $\sigma$ by $\eta$. We denote the local component of $\pi$ at $\nu$ by $\pi_{\nu}$, and also $\pi_{\infty}=\otimes_{\nu\in S_{\infty}}\pi_{\nu}$.
We use similar notation for $\psi$, $\eta$ and $\sigma$.

\subsection{The Eisenstein series}\label{section:complete L functions for twists}
Let $\tau$ be a cuspidal representation of $\GL_k(\A)$. For $c\geq1$, let $\rho_c(\tau)$ be the generalized Speh representation of $\GL_{kc}(\A)$ associated with $\tau$ (\cite[\S~2.4]{Jac4}). Denote $H=\mathcal{G}_{2kc}$. Fix a maximal compact group $K_H=\prod_{\nu}K_{H_{\nu}}$ in ``good position" with respect to $T_H$ (\cite[\S~I.1.1, \S~I.1.4]{MW2}). Let $P$ be a standard Siegel parabolic subgroup of $H$. Thus, $M_P\cong\GL_{kc}\times\mathcal{G}_0$. When $H=\GSpin_{2kc}$ (then
$\mathcal{G}_0=\GL_1$), let $\theta$ be a unitary Hecke character of $\A^*$, which is understood to be trivial in the remaining part of the paper when $H\ne\GSpin_{2kc}$. Let $s$ be a complex parameter.

For an entire $K_H$-finite section $f$ of the space of $\Ind_{P(\A)}^{H(\A)}((\absdet^{s}\rho_c(\tau)\otimes\theta)\circ i_{M_P})$, define the Eisenstein series
\begin{align}\label{eq:Eisenstein series main}
E_{\tau,\theta,c}(h;s,f)=\sum\limits_{\gamma\in P(F)\backslash H(F)}f(s,\gamma h),\qquad h\in H(\A).
\end{align}
It is absolutely convergent in $\Real(s)\gg0$ and admits meromorphic continuation to $\C$ (\cite{La5,MW2}).

In this subsection we prove the following result.
\begin{theorem}\label{theorem:series}
Let $\tau$ be a unitary cuspidal representation of $\GL_k(\A)$ and let $\theta$ be a unitary Hecke character of $\A^*$ (when $H\ne\GSpin_{2kc}$, $\theta=1$). Assume
$\absdet^{it}\tau\not\cong\theta{\tau}^{\vee}$ for all $t\in\R$. Then
$E_{\tau,\theta,c}(h;s,f)$ is holomorphic in $\Real(s)\geq0$.
\end{theorem}

\begin{remark}
The Eisenstein series $E_{\tau,1,c}$ was studied in \cite{JiangLiuZhang2013} for quasi-split classical groups, in the self-dual case. According to \cite[Theorem~5.2]{JiangLiuZhang2013} the operator $E_{\tau,1,c}(h;s,f)$ is holomorphic in $\Real(s)\geq0$ except perhaps at a finite set $X_{\tau}$ of poles at half integers, explicitly given in \textit{loc. cit.}~ depending on $H$ and on the self-duality type of $\tau$. (The statement in \textit{loc. cit.}~ was formulated for the normalized Eisenstein series but the result is the same.) For example when $H=\SO_{2kc}$, $2|c$ and $L(s,\tau,\wedge^2)$ has a pole at $s=1$, $X_{\tau}=\{c/2,(c-2)/2,\ldots,1\}$. We expect an analogous statement to hold for
$H=\GSpin_{2kc}$, except that the self-duality type condition which determines the poles depends also on $\theta$, e.g., in the aforementioned example $L(s,\tau,\wedge^2)$ is replaced by $L(s,\tau,\wedge^2\otimes\theta^{-1})$. A key important ingredient for the proof of \cite[Theorem~5.2]{JiangLiuZhang2013} was Arthur's classification \cite{Arthur2013}, which we avoid using in this work (and does not yet include the case $\GSpin_{2kc}$). Thus we will not use \cite[Theorem~5.2]{JiangLiuZhang2013}, however, we will use a constant term computation from \cite{JiangLiuZhang2013} (see below) which is independent of \cite{Arthur2013}.
\end{remark}

We use the notation $\rtimes$ and $\times$ for parabolic induction as in the local case, see \S~\ref{local groups and notation}.
When $H\ne\Sp_{2kc}$ and $k$ is odd, we will need to distinguish between induction from $P$ and from ${}^{\Specialjmath}P$. For clarity, we ignore this case for the moment,
and at the end of the proof explain the modifications needed to handle it.

For $c=1$, the constant term of $E_{\tau,\theta,1}(h;s,f)$ along $P$ is given by
\begin{align}\label{eq:constant term c=1}
(E_{\tau,\theta,1})_P(h;s,f)=f(s,h)+M_0(s)f(s,h),
\end{align}
where $M_0(s)$ is the standard intertwining operator
\begin{align*}
&M_0(s):\absdet^{s}\tau\rtimes\theta\to\absdet^{-s}\theta\tau^{\vee}\rtimes\theta.
\end{align*}
By M{\oe}glin and Waldspurger \cite[Remark~IV.3.12]{MW2}, under the assumption
$\absdet^{it}\tau\not\cong\theta{\tau}^{\vee}$ for all $t\in\R$, the series $E_{\tau,\theta,1}(h;s,f)$ or equivalently $M_0(s)$ is holomorphic in $\Real(s)\geq0$ as required.

Henceforth assume $c\geq2$.

Fix $T_{\GL_{kc}}$ and $B_{\GL_{kc}}$ using $i_{M_P}$ as in \S~\ref{The local integral}. For a composition $\beta$ of $kc$,
denote by $P_{\beta}$ the standard parabolic subgroup of $H$ such that $M_{P_{\beta}}\cong M_{\beta}\times\mathcal{G}_{0}$ and $M_{P_{\beta}}<M_P$.

Observe that $M_{P_{(k^c)}}$ is the only standard Levi subgroup of $H$ contributing to the cuspidal support of $E_{\tau,\theta,c}(h;s,f)$.
Denote $c'=c-1$ and $H'=\mathcal{G}_{2kc'}$. Let $\ParabolicCT$ and $\ParabolicCT'$ be the standard maximal parabolic subgroups of $H$ such that
$M_{\ParabolicCT}\cong\GL_k\times H'$ and $M_{\ParabolicCT'}\cong\GL_{kc'}\times \mathcal{G}_{2k}$.
We use the computation of the constant term of the series along $\ParabolicCT$ from \cite{JiangLiuZhang2013}.
To describe the constant term we introduce the following notation. First, let $f_{P_{(k,kc')}}$ and $f_{P_{(kc',k)}}$ denote the constant terms of $f$ along the specified parabolic subgroups of $H$.
Consider the representations $\absdet^{s+1/2}\rho_{c'}(\tau)\rtimes\theta$ and $\absdet^{s+c'/2}\tau\rtimes\theta$ of $H'$ and $\mathcal{G}_{2k}$ (resp.), induced from
the projection of $i_{\ParabolicCT}(P_{(k,kc')}\cap i_{\ParabolicCT}^{-1}(I_k,H'))$ into $H'$ and of $i_{\ParabolicCT'}(P_{(kc',k)}\cap i_{\ParabolicCT'}^{-1}(I_{kc'},\mathcal{G}_{2k}))$ into $\mathcal{G}_{2k}$ (resp.). The function $f_{P_{(k,kc')}}$ belongs to the space of
\begin{align}\label{rep:image of the f1}
\absdet^{s-c'/2}\tau\rtimes(\absdet^{s+1/2}\rho_{c'}(\tau)\rtimes\theta),
\end{align}
and $f_{P_{(kc',k)}}$ belongs to the space of
\begin{align*}
\absdet^{s-1/2}\rho_{c'}(\tau)\rtimes(\absdet^{s+c'/2}\tau\rtimes\theta).
\end{align*}
Also consider the standard intertwining operators
\begin{align*}
&M_1(s):\absdet^{s+c'/2}\tau\rtimes\theta\to\absdet^{-s-c'/2}\theta\tau^{\vee}\rtimes\theta,\nonumber\\
&M_2(s):\absdet^{s-1/2}\rho_{c'}(\tau)\times\absdet^{-s-c'/2}\theta\tau^{\vee}\to
\absdet^{-s-c'/2}\theta\tau^{\vee}\times\absdet^{s-1/2}\rho_{c'}(\tau).
\end{align*}
Applying the parabolic induction functors we obtain the operators
\begin{align*}
&\intertInduced{1}{s}:\absdet^{s-1/2}\rho_{c'}(\tau)\rtimes(\absdet^{s+c'/2}\tau\rtimes\theta)\to \absdet^{s-1/2}\rho_{c'}(\tau)\rtimes(\absdet^{-s-c'/2}\theta\tau^{\vee}\rtimes\theta),\\
&\intertInduced{2}{s}:(\absdet^{s-1/2}\rho_{c'}(\tau)\times\absdet^{-s-c'/2}\theta\tau^{\vee})\rtimes\theta
\to (\absdet^{-s-c'/2}\theta\tau^{\vee}\times\absdet^{s-1/2}\rho_{c'}(\tau))\rtimes\theta.
\end{align*}
Then
\begin{align}\label{rep:image of the last M s}
\intertInduced{2}{s}\intertInduced{1}{s}:\absdet^{s-1/2}\rho_{c'}(\tau)\rtimes(\absdet^{s+c'/2}\tau\rtimes\theta)\to
\absdet^{-s-c'/2}\theta\tau^{\vee}\rtimes(\absdet^{s-1/2}\rho_{c'}(\tau)\rtimes\theta).
\end{align}

Now as in \cite[Proposition~2.3]{JiangLiuZhang2013},
\begin{align}\label{eq:inductive formula series JLZ}
(E_{\tau,\theta,c})_{\ParabolicCT}(s,f)=E_{\tau,\theta,c'}(s+1/2,f_{P_{(k,kc')}})+E_{\tau,\theta,c'}(s-1/2,\intertInduced{2}{s}\intertInduced{1}{s}f_{P_{(kc',k)}}).
\end{align}
Here in the first (resp., second) summand $E_{\tau,\theta,c'}$ is applied to a section of the space of \eqref{rep:image of the f1} (resp.,
the r.h.s.~ of \eqref{rep:image of the last M s}).

\begin{lemma}\label{lemma:hol of M1 and M2}
The operator $M_1(s)$ is holomorphic in $\Real(s)>-c'/2$ and $M_2(s)$ is
holomorphic in $\Real(s)\geq1/2$.
\end{lemma}
\begin{proof}
The statement regarding $M_1(s)$ follows from the case $c=1$ noting that $M_1(s)=M_0(s+c'/2)$.
Regarding $M_2(s)$, according to the local results \cite[\S~I.10]{MW4} the operator
\begin{align*}
\frac{L(2s+c',\tau\times\theta^{-1}\tau)}{
L(2s,\tau\times\theta^{-1}\tau)}M_2(s)
\end{align*}
is holomorphic when $\Real(s-1/2)>\Real(-s-c'/2)$, i.e., $\Real(s)>\tfrac14(1-c')$. By \cite{JS2,JS1,Sh4},
$L(2s+c',\tau\times\theta^{-1}\tau)$ is holomorphic when $\Real(2s+c')>1$, hence in $\Real(s)>0$, and $L(2s,\tau\times\theta^{-1}\tau)$ is nonvanishing when
$\Real(2s)\geq1$, i.e., in $\Real(s)\geq1/2$. Therefore $M_2(s)$ is holomorphic in $\Real(s)\geq1/2$.
\end{proof}

For any $s_0\in\C$, let $\LeadingCoeff_{s_0,c}$ be the leading coefficient of the Laurent expansion of $E_{\tau,\theta,c}(s,f)$ about $s=s_0$.
It is a nonzero automorphic form on $H(\A)$. Let $e(s,c)$ be the set of $\Real(\Xi)$ where $\Xi$ varies over the cuspidal exponents of $\LeadingCoeff_{s,c}$ along $P_{(k^c)}$ (see \cite[\S~I.3.3, \S~I.3.5]{MW2}). Recall that $M_{P_{(k^c)}}\cong M_{(k^c)}(=\GL_k\times\ldots\times\GL_k)$ when $H\ne\GSpin_{2kc}$, and
$M_{P_{(k^c)}}\cong M_{(k^c)}\times\GL_1$ otherwise. The elements of $e(s,c)$ can be written using the standard basis for the character lattice of $M_{P_{(k^c)}}$, identified
with $\Z^c$ (in the obvious way) if $H\ne\GSpin_{2kc}$ and with $\Z^{c+1}$ when $H=\GSpin_{2kc}$. In the latter case, because $\theta$ is unitary,
the $(c+1)$-th coordinate is $0$, therefore we view $e(s,c)$ as a subset of $\R^c$ in all cases.
\begin{lemma}\label{lemma:exponents}
Any $e\in e(s,c)$ can be written in the form
\begin{align}\label{eq:form of e}
e=\Real(s)(a_1,\ldots,a_c)+(b_1,\ldots,b_c),
\end{align}
where for each $1\leq j\leq c$, $a_j=\pm1$ and $\sum_{i=1}^jb_i\leq0$.
\end{lemma}
\begin{proof}
It is enough to prove the lemma for a generic $s$.
We argue using induction on $c$. For $c=1$, by \eqref{eq:constant term c=1},
$e(s,1)\subset\{-\Real(s),\Real(s)\}$ and the lemma holds ($a_1=\pm1$, $b_1=0$).

Assume the result for $c'$ and consider $c$.
Then by \eqref{eq:inductive formula series JLZ}, \eqref{rep:image of the f1} and \eqref{rep:image of the last M s},
\begin{align*}
e(s,c)\subset\{(\Real(s)-c'/2,e'):e'\in e(s+1/2,c')\}\bigcup \{(-\Real(s)-c'/2,e'):e'\in e(s-1/2,c')\}.
\end{align*}
The proof follows from this by an elementary induction.
\end{proof}

\begin{proof}[Proof of Theorem~\ref{theorem:series}]
By \cite[Theorem~VI.2.1(i)]{MW2}, $E_{\tau,\theta,c}(s,f)$ is holomorphic on $\Real(s)=0$.
Assume $s_0$ is a pole of $E_{\tau,\theta,c}(s,f)$ with $\Real(s_0)>0$. By \cite[Lemma~I.4.10]{MW2},
since $M_{P_{(k^c)}}$ is the only standard Levi subgroup of $H$ contributing to the cuspidal support of $E_{\tau,\theta,c}(s,f)$,
$s_0$ is a pole of $(E_{\tau,\theta,c})_{P_{(k^c)}}(s,f)$ and thereby of $(E_{\tau,\theta,c})_{\ParabolicCT}(s,f)$.

By the induction hypothesis $E_{\tau,\theta,c'}(s+1/2,f_{P_{(k,kc')}})$ is holomorphic at $s_0$. Hence by \eqref{eq:inductive formula series JLZ}, $E_{\tau,\theta,c'}(s-1/2,\intertInduced{2}{s}\intertInduced{1}{s}f_{P_{(kc',k)}})$ has a pole at $s_0$, and by Lemma~\ref{lemma:hol of M1 and M2} and the induction hypothesis again, $\Real(s_0)<1/2$.

Denote $\LeadingCoeff=\LeadingCoeff_{s_0,c}$. As in the proof of \cite[Theorem~6.1]{JiangLiuZhang2013}, we show that $\LeadingCoeff$ is square-integrable using the criterion on the negativity of (the real parts of) the cuspidal exponents from \cite[\S~I.4.11]{MW2}.

We need to show that for any $e=(e_1,\ldots,e_c)\in e(s_0,c)$, i.e., the real part of a cuspidal exponent of $\LeadingCoeff$ along $P_{(k^c)}$, $\sum_{i=1}^j\Real(e_i)<0$ for all $1\leq j\leq c$. By \eqref{eq:inductive formula series JLZ} and since $E_{\tau,\theta,c'}(s+1/2,f_{P_{(k,kc')}})$ is holomorphic at $s_0$, when we consider
the image of \eqref{rep:image of the last M s} we deduce $e=(-\Real(s_0)-c'/2,e')$ where
$e'\in e(s_0-1/2,c')$. The elements of $e(s_0-1/2,c')$, that is, the real parts of cuspidal exponents of $\LeadingCoeff_{s_0-1/2,c'}$ along $P_{(k^{c'})}$,
 are described by Lemma~\ref{lemma:exponents} and we conclude that
\begin{align*}
e&=(-\Real(s_0)-c'/2,0,\ldots,0)+(\Real(s_0)-1/2)(0,a_2,\ldots,a_c)+(0,b_2,\ldots,b_c) \\
&=(\Real(s_0)-1/2)(-1,a_2,\ldots,a_c)+(-c/2,b_2,\ldots,b_c),
\end{align*}
where $a_i=\pm1$ for each $i$ and $\sum_{i=2}^jb_i\leq 0$ for all $2\leq j\leq c$.
It remains to show that
\begin{align}\label{eq:identity for exponents}
\Real(1/2-s_0-c/2+\sum_{i=2}^{j}(a_i(s_0-1/2)+b_i))<0,\qquad\forall 1\leq j\leq c.
\end{align}
(By definition $\sum_{i=2}^{1}=0$.) Since $\Real(s_0)<1/2$, the l.h.s.~ is maximized when $a_2=\ldots=a_{j}=-1$, and using $\sum_{i=2}^jb_i\leq 0$,
the l.h.s.~ is bounded by $-j\Real(s_0)$ (for each $j\geq1$). Then because $\Real(s_0)>0$, we deduce \eqref{eq:identity for exponents}. We conclude that $\LeadingCoeff$ is square-integrable.

Let $W:=W(M_{P_{(k^c)}})$ denote the set of $\WeylElement\in W_H$ of minimal length in $\WeylElement W_{M_{P_{(k^c)}}}$ such that ${}^{\WeylElement}M_{P_{(k^c)}}$ is a standard Levi subgroup of $H$ (\cite[\S~I.1.7]{MW2}). In fact in our case $W$ normalizes $M_{P_{(k^c)}}$. Note that
$W$ can be identified with the set of signed permutations on $\{1,\ldots,c\}$. The cuspidal support of $\LeadingCoeff$ is contained in the $W$-orbit $\orb$ of
\begin{align*}
\Dpt=\absdet^{s_0+c'/2}\tau\otimes\ldots\otimes\absdet^{s_0-c'/2}\tau\otimes\theta.
\end{align*}
By the description of the residual spectrum \cite[Theorem~V.3.13(iii), Corollary~V.3.16]{MW2},
$\LeadingCoeff$ belongs to the space $\mathcal{L}:=L^2(H(F)\backslash H(\A))_{\orb}$ defined in
\cite[V.3.14]{MW2}. In particular $\mathcal{L}\ne0$. It follows from the definition of
$\mathcal{L}$ and \cite[Remark~V.3.6(b)]{MW2} that the Hermitian dual $\Dpt^*$ of $\Dpt$ (the conjugate of the contragredient) belongs to $\orb$,
i.e., $\Dpt^*=\WeylElement\Dpt$ for some $\WeylElement\in W$. In particular $\WeylElement\lambda_0=-\lambda_0$ where
\begin{align*}
\lambda_0=(\Real(s_0)+c'/2,\dots,\Real(s_0)-c'/2)\in\R^c
\end{align*}
is the ``real part" of $\Dpt$. Since $\Real(s_0)\not\in\frac12\Z$, $\lambda_0$ is regular and we infer that $\WeylElement=\WeylElement_H\WeylElement_{M_{P_{(k^c)}}}$. It follows that
$\absdet^{s_0-\overline{s_0}}\tau\cong\theta\tau^\vee$, which contradicts the hypothesis of Theorem~\ref{theorem:series}.
\end{proof}

As indicated above, the cases of $H=\SO_{2kc}$ and $\GSpin_{2kc}$, when $k$ is odd, are a bit different.
We explain this now. Suppose in addition that $k>1$. The standard Levi subgroups of $H$ contributing to the cuspidal support of $E_{\tau,\theta,c}(h;s,f)$ are $M_{P_{(k^c)}}$ and ${}^{\jmath}M_{P_{(k^c)}}$.
The computation of the constant term is slightly more complicated because we need to distinguish between $P$ and ${}^{\jmath}P$. To this end we use the notation $\rtimes_{\Specialjmath^r}$ to define the induction from ${}^{\Specialjmath^r}P$, where $r\in\{0,1\}$. We can then consider the Eisenstein series $E_{\tau,\theta,c}^{\Specialjmath}(h;s,f')$
similar to \eqref{eq:Eisenstein series main} but with respect to ${}^{\Specialjmath}P$, for a section $f'$ of the space of $\absdet^{s}\rho_c(\tau)\rtimes_{\Specialjmath}\theta$.
For $c=1$,
\begin{align*}
(E_{\tau,\theta,1})_{{}^{\Specialjmath^{r}}P}(h;s,f)=
\begin{dcases}
f(s,h) & r=0,\\
M_0(s)f(s,h) & r=1.
\end{dcases}
\end{align*}
Here
\begin{align}\label{eq:twisted M0}
&M_0(s):\absdet^{s}\tau\rtimes\theta\to\absdet^{-s}\theta\tau^{\vee}\rtimes_{\Specialjmath}\theta.
\end{align}
Similar results hold for $(E_{\tau,\theta,1}^{\Specialjmath})_{{}^{\Specialjmath^{r}}P}(h;s,f')$.

We describe the analog of \eqref{eq:inductive formula series JLZ}. For the formula for $(E_{\tau,\theta,c})_{\ParabolicCT}$, let
$M_1(s)=M_0(s+c'/2)$ with $M_0(s)$ now given by \eqref{eq:twisted M0}, define $\intertInduced{1}{s}$ by applying the induction functor to
$M_1(s)$, and $\intertInduced{2}{s}$ is defined as above. Then we have
\begin{align}\label{eq:inductive formula series JLZ odd k case 1}
(E_{\tau,\theta,c})_{\ParabolicCT}(s,f)=E_{\tau,\theta,c'}(s+1/2,f_{P_{(k,kc')}})+E_{\tau,\theta,c'}^{\Specialjmath}(s-1/2,{}^{\Specialjmath}\intertInduced{2}{s}\intertInduced{1}{s}f_{P_{(kc',k)}}).
\end{align}
A similar formula holds for $(E_{\tau,\theta,c}^{\Specialjmath})_{\ParabolicCT}(s,f')$, with the intertwining operators properly adjusted.

Now we can compute cuspidal exponents along either $P_{(k^c)}$ or ${}^{\jmath}P_{(k^c)}$ (both contribute), and the result of Lemma~\ref{lemma:exponents} applies
equally well to both. The proof of Theorem~\ref{theorem:series} follows along the same lines, but note that
each $\WeylElement\in W(M_{P_{(k^c)}})$ either normalizes $M_{P_{(k^c)}}$ or ${}^{\WeylElement}M_{P_{(k^c)}}={}^{\Specialjmath}M_{P_{(k^c)}}$.

The proof for $k=1$ ($H=\SO_{2c}$ or $\GSpin_{2c}$) is similar and simpler: in identity~\eqref{eq:inductive formula series JLZ odd k case 1} $\intertInduced{1}{s}$ is omitted, and $P_{(k^c)}=B_H$ which is normalized by $\jmath$. See also Kudla and Rallis \cite{KudlaRallis1989} and \cite[Theorem~1.1]{KudlaRallis1994} (for $\Orth_{2c}$).

\subsection{The integral}\label{The gbl integral}
Let $c>1$ and $k\geq1$ be integers, $G=\mathcal{G}_c$ and $H=\mathcal{G}_{2kc}$. As in
\S~\ref{The local integral} we have the subgroups $P$ and $Q$, $U=U_Q$, and we define the character $\psi_U$,
now of $U(F)\backslash U(\A)$, which is generic with respect to the unipotent orbit $((2k-1)^c1^c)$. Let
$\mathrm{St}_{M_Q(\A)}(\psi_U)$ be the stabilizer of $\psi_U$ in $M_Q(\A)$. The character $\psi_U$ can be chosen such that
there is a map $\embedding:G(\A)\times G(\A)\rightarrow\mathrm{St}_{M_Q(\A)}(\psi_U)$ whose kernel is the diagonal embedding of $C_G^{\circforgspin}(\A)$. We let $K_H$ be the maximal compact subgroup as in \S~\ref{section:complete L functions for twists} with the additional requirement that for all places $\nu$ of $F$, $\embedding_{\nu}(K_{G_{\nu}},K_{G_{\nu}})<K_{H_{\nu}}$.

Let $\pi$ and $\tau$ be cuspidal representations of $G(\A)$ and $\GL_k(\A)$.
Consider the representation $\Ind_{P(\A)}^{H(\A)}((\absdet^{s}\chi_{\pi}\rho_c(\tau)\otimes\chi_{\pi})\circ i_{M_P})$ and the Eisenstein series
$E(h;s,f)=E_{\chi_{\pi}\tau,\chi_{\pi},c}(h;s,f)$ given by \eqref{eq:Eisenstein series main} (with $\theta=\chi_{\pi}$), defined for an entire $K_H$-finite section $f$.

We form the Fourier coefficient of $E(h;s,f)$ along $(U,\psi_U)$:
\begin{align*}
E^{U,\psi_U}(h;s,f)=\int\limits_{U(F)\backslash U(\A)}E(uh;s,f)\psi_U(u)\,du,\qquad h\in H(\A).
\end{align*}
This is an automorphic function on $G(\A)\times G(\A)$.
For cusp forms $\varphi_1$ and $\varphi_2$ in the spaces of $\pi$ and $\pi^{\vee}$, resp., define
\begin{align*}
&Z(s,\varphi_1,\varphi_2,f)\\&=\int\limits_{(C_G^{\circforgspin}(\A)G(F)\times C_G^{\circforgspin}(\A)G(F))\backslash G({\A})\times G({\A})}
\varphi_1(g_1)\,{}^{\iota}\varphi_2(g_2)\,
E^{U,\psi_U}(\embedding(g_1,g_2);s-1/2,f)\,dg_1\,dg_2.
\end{align*}
Here ${}^{\iota}\varphi_2(g)=\varphi_2(\iota g\iota^{-1})$ ($\iota=\prod_{\nu}\iota_{\nu}$, $\iota_{\nu}$ was defined in \S~\ref{The local integral}).
\begin{remark}
We consider the Eisenstein series at $s-1/2$, so that $Z(s,\varphi_1,\varphi_2,f)$ will represent the $L$-function at $s$ (instead of $s+1/2$).
\end{remark}
The integral $Z(s,\varphi_1,\varphi_2,f)$ is absolutely convergent for all $s$ where the Eisenstein series is holomorphic. This follows from the rapid decay of the cusp forms and moderate growth of the
series. Thus the integral admits meromorphic continuation to $\C$, which is analytic except perhaps at the poles of the Eisenstein series. According to the unfolding argument of \cite[\S~3.2]{DimaKaplan} (a preliminary version was sketched in \cite{CFGK2}), for decomposable data $(\varphi_1,\varphi_2,f)$, in $\Real(s)\gg0$,
\begin{align*}
Z(s,\varphi_1,\varphi_2,f)=\prod_{\nu}Z(s,\omega_{\nu},f_{\nu})
\end{align*}
where $\omega_{\nu}$ is a matrix coefficient of $\pi_{\nu}^{\vee}$ and $f_{\nu}\in V(\tau_{\nu},c)$. Specifically, $f_{\nu}$ is
an entire $K_{H_{\nu}}$-finite section corresponding to $\Ind_{P(F_{\nu})}^{H(F_{\nu})}((\absdet_{\nu}^{s-1/2}\chi_{\pi_{\nu}}W_{\psi_{\nu}}(\rho_c(\tau_{\nu}))\otimes
\chi_{\pi_{\nu}})\circ i_{M_P})$.
Thus, for a sufficiently large finite set $S$ of places of $F$, by \eqref{int:unr} we have the following identity of meromorphic functions:
\begin{align}\label{eq:consequence of crude functional equation}
b^{S}(s,c,\tau\otimes\chi_{\pi})Z(s,\varphi_1,\varphi_2,f)=L^{S}(s,\pi\times\tau)\prod_{\nu\in S}Z(s,\omega_{\nu},f_{\nu}).
\end{align}
Here $b^{S}(s,c,\tau\otimes\chi_{\pi})=\prod_{\nu\not\in S}b(s,c,\tau_{\nu}\otimes\chi_{\pi_{\nu}})$.

Define the completed $L$-function $L(s,\pi\times\tau)=\prod_{\nu}L(s,\pi_{\nu}\times\tau_{\nu})$, which is a meromorphic function, and also define
$\epsilon(s,\pi\times\tau)=\prod_{\nu}\epsilon(s,\pi_{\nu}\times\tau_{\nu},\psi_{\nu})$. The following global results were proved in \cite{CFK2022}.
\begin{theorem}\cite[Theorem~8.3]{CFK2022}\label{theorem:global properties of L function}
$L(s,\pi\times\tau)=\epsilon(s,\pi\times\tau)L(1-s,\pi^{\vee}\times\tau^{\vee})$.
\end{theorem}
\begin{theorem}\cite[Corollary~8.5]{CFK2022}\label{corollary:BVS for entire}
If $L(s,\pi\times\tau)$ and $L(1-s,\pi^{\vee}\times\tau^{\vee})$ are entire, then they are bounded in vertical strips of finite width.
\end{theorem}

The next result will be used in \S~\ref{section:Constructing the lift}, together with Theorem~\ref{theorem:global properties of L function}, to
verify the conditions of the Converse Theorem. The proof makes critical use of Theorem~\ref{theorem:series} and of the local result Theorem~\ref{theorem:archimedean producing poles}.

\begin{theorem}\label{theorem:twisting to obtain entire L function}
Assume that $\pi$ and $\tau$ are unitary cuspidal representations of $G(\A)$ and $\GL_k(\A)$, resp., and suppose the following conditions hold.
\begin{enumerate}[leftmargin=*]
\item\label{it:not self-dual} For any $t\in\R$, $\absdet^{it}\chi_{\pi}\tau\not\cong{\tau}^{\vee}$.
\item\label{it:local good} For each place $\nu<\infty$ such that $\pi_{\nu}$ is not unramified,
$L(s,\pi_{\nu}\times\tau_{\nu})=L(s,\pi_{\nu}^{\vee}\times\tau_{\nu}^{\vee})=1$.
\end{enumerate}
Then $L(s,\pi\times\tau)$ and $L(s,\pi^{\vee}\times\tau^{\vee})$ are entire and bounded in vertical strips of finite width.
\end{theorem}

\begin{proof}
Let $S$ be a finite set of places, which contains $S_{\infty}$, and such that for every $\nu\notin S$ all data are unramified (e.g., $\pi_{\nu}$, $\tau_{\nu}$ and $\psi_{\nu}$).
Under assumption~\eqref{it:not self-dual}, the partial $L$-functions $L^S(s,\tau,\wedge^2\otimes\chi_{\pi})$ and $L^S(s,\tau,\vee^2\otimes\chi_{\pi})$
are holomorphic in $\Real(s)\geq0$ by the result \cite[Remark~IV.3.12]{MW2} on the intertwining operator and a standard argument (see e.g., \cite[(1.3)]{Kim2000}). In addition since $c\geq2$, $L(s+c/2,\tau)$ is holomorphic in $\Real(s)\geq1/2$ even without assumption~\eqref{it:not self-dual}, by \cite{JS1}. Thus by \eqref{eq:b}, $b^S(s,c,\tau\otimes\chi_{\pi})$ is holomorphic in $\Real(s)\geq1/2$.~\footnote{In fact $b^S(s,c,\tau\otimes\chi_{\pi})$ is entire by Kim and Shahidi \cite[Proposition~2.1]{KimShahidi2002}, for which assumption~\eqref{it:not self-dual} is seen to be sufficient because in the notation of \textit{ibid.}~ $m=1$ and $r=r_1$.}

Multiplying and dividing the r.h.s.~ of \eqref{eq:consequence of crude functional equation} by
$\prod_{\nu\in S}L(s,\pi_{\nu}\times\tau_{\nu})$, we obtain
\begin{align}\label{eq:global integral computation with b on lhs 2}
b^S(s,c,\tau\otimes\chi_{\pi})Z(s,\varphi_1,\varphi_2,f)=L(s,\pi\times\tau)Q(s),
\end{align}
where $Q(s)=\prod_{\nu\in S}\frac{Z(s,\omega_{\nu},f_{\nu})}{L(s,\pi_{\nu}\times\tau_{\nu})}$.

Now consider $s_0\in\C$ with $\Real(s_0)\geq1/2$.
On the one hand, by Theorem~\ref{theorem:archimedean producing poles} and Theorem~\ref{theorem:all local props} \eqref{it:nonzero}, we can choose for each $\nu\in S$
a matrix coefficient $\omega_{\nu}$ of $\pi^{\vee}_{\nu}$ and a $K_{H_{\nu}}$-finite $f_{\nu}\in V(\tau_{\nu},c)$ (an entire section) such that $Q(s_0)\ne0$.
Observe that for the places $\nu\in S-S_{\infty}$ where $\pi_{\nu}$ is not unramified, we only need to apply Theorem~\ref{theorem:all local props} \eqref{it:nonzero},
by assumption~\eqref{it:local good}.

On the other hand, take $\omega$ and $f$ that equal $\omega_{\nu}$ and $f_{\nu}$ at each $\nu\in S$, and equal the normalized unramified vectors outside $S$, resp.
Then $\omega$ and $f$ are factorizable and $f$ is entire and $K_H$-finite. Now by Theorem~\ref{theorem:series}, which is applicable by assumption~\eqref{it:not self-dual},
the l.h.s.~ of \eqref{eq:global integral computation with b on lhs 2} is holomorphic at $s_0$,
hence $L(s,\pi\times\tau)$ is holomorphic at $s_0$. The same argument shows that $L(s,\pi^{\vee}\times\tau^{\vee})$ is holomorphic in $\Real(s)\geq1/2$, and by
Theorem~\ref{theorem:global properties of L function} they are both entire. The remaining part of the statement now follows from
Theorem~\ref{corollary:BVS for entire}.
\end{proof}

\subsection{Constructing a weak functorial transfer}\label{section:Constructing the lift}
We can now prove the main result of the paper.
\begin{proof}[Proof of Theorem~\ref{theo:globl functorial lift}]
Let $\pi$ be a cuspidal representation of $G(\A)$. For the proof we can assume
$\pi$ is unitary. We construct a transfer following the prescription in \cite{CKPS}.

Write $\pi=\otimes'_{\nu}\pi_{\nu}$ as a restricted tensor product. Fix a nonempty finite set $S_{\pi}$ of finite places such that for each $\nu\notin S_{\pi}\coprod S_{\infty}$, $\pi_{\nu}$ is unramified. Define representations $\Pi_{\nu}$ of $\GL_N(F_{\nu})$ as follows:
\begin{itemize}[leftmargin=*]
\item For each $\nu\notin S_{\pi}\coprod S_{\infty}$, $\Pi_{\nu}=\transfer_{\nu}(\pi_{\nu})$ is the unramified transfer (see
\S~\ref{local transfer results}).
\item For the places $\nu\in S_{\infty}$, $\Pi_{\nu}=\transfer_{\nu}(\pi_{\nu})$ is the archimedean functorial transfer (see \S~\ref{local transfer results}).
\item For each $\nu\in S_{\pi}$, $\Pi_{\nu}$ is an arbitrary irreducible representation whose central character is $\chi_{\pi_{\nu}}^{N/2}$.
\end{itemize}
The central character of $\Pi_{\nu}$ is $\chi_{\pi_{\nu}}^{N/2}$ at all places; for $\nu\not\in S_{\pi}$ this follows from \S~\ref{local transfer results}.

Let $\Pi=\otimes'_{\nu}\Pi_{\nu}$, which is an irreducible admissible representation of $\GL_N(\A)$.

We verify the conditions of the Converse Theorem -- Theorem~\ref{theorem:cnv}.

By construction, the central character of $\Pi$ is $\chi_{\pi}^{N/2}$ which is trivial on $F^*$ (even on $\A^*$ when $G\ne\GSpin_c$).
The Euler product
$L(s,\Pi)=\prod_{\nu}L(s,\Pi_{\nu})$ is absolutely convergent in some right half plane, because if $S=S_{\pi}\coprod S_{\infty}$, then $L^S(s,\Pi)=L^S(s,\pi)$ is absolutely convergent in $\Real(s)\gg0$.

Let $\eta$ be a unitary Hecke character of $\A^*$ which is sufficiently highly ramified at all $\nu\in S_{\pi}$, depending
on the representations $\pi_{\nu}$ and $\Pi_{\nu}$;
and such that in particular for some $\nu_0\in S_{\pi}$, $\eta_{\nu_0}^2\chi_{\pi_{\nu_0}}$ is nontrivial on $\mathcal{O}_{\nu_0}^*$.

Let $\tau$ belong to the set $\mathscr{A}(S_{\pi},\eta)$ of cuspidal representations defined in \S~\ref{The converse theorem}. In order to check the remaining conditions of Theorem~\ref{theorem:cnv} for $L(s,\Pi\times\tau)$ and $L(s,\Pi^{\vee}\times\tau^{\vee})$,
observe that
\begin{align*}
&L(s,\Pi\times\tau)=L(s,\pi\times\tau),
\quad\epsilon(s,\Pi\times\tau)=\epsilon(s,\pi\times\tau),
\quad L(s,\Pi^{\vee}\times\tau^{\vee})=L(s,\pi^{\vee}\times\tau^{\vee}).
\end{align*}
Indeed this follows because the local $L$- and $\epsilon$-factors for $\Pi_{\nu}\times\tau_{\nu}$ and $\pi_{\nu}\times\tau_{\nu}$ coincide for all $\nu$, and similarly for $\Pi_{\nu}^{\vee}\times\tau_{\nu}^{\vee}$ and $\pi_{\nu}^{\vee}\times\tau_{\nu}^{\vee}$:
\begin{itemize}[leftmargin=*]
\item For all $\nu\notin S_{\pi}$ by Theorem~\ref{theorem:pi and Pi for unramified pi or archimedean},
\item For each $\nu\in S_{\pi}$ by Theorem~\ref{theorem:pi and Pi for ramified twisted}.
\end{itemize}
Hence we may verify Theorem~\ref{theorem:cnv} \eqref{It:analytic cont}--\eqref{It:functional eq} for $L(s,\pi\times\tau)$ and $L(s,\pi^{\vee}\times\tau^{\vee})$.

We apply Theorem~\ref{theorem:twisting to obtain entire L function}. The assumptions of this theorem
hold: indeed the condition $\eta_{\nu_0}^2\chi_{\pi_{\nu_0}}|_{\mathcal{O}_{\nu_0}^*}\ne1$ on $\eta$ implies that $\absdet^{it}\chi_{\pi}\tau\not\cong{\tau}^{\vee}$ for every $t\in\R$ (see
the proof of \cite[Lemma~3.2]{CKPS2}), and the assumption \eqref{it:local good} on the local $L$-factors at the places in $S_{\pi}$ is satisfied by
Theorem~\ref{theorem:pi and Pi for ramified twisted}. Hence Theorem~\ref{theorem:twisting to obtain entire L function} is applicable and we deduce that $L(s,\pi\times\tau)$ and $L(s,\pi^{\vee}\times\tau^{\vee})$ are entire and bounded in vertical strips of finite width. Finally by Theorem~\ref{theorem:global properties of L function}, $L(s,\pi\times\tau)=\epsilon(s,\pi\times\tau)L(1-s,\pi^{\vee}\times\tau^{\vee})$.

Therefore the conditions of Theorem~\ref{theorem:cnv} hold. Thus there exists an irreducible automorphic representation
$\Pi'$ of $\GL_N(\A)$ such that $\Pi'_{\nu}\cong\Pi_{\nu}$ for all $\nu\notin S_{\pi}$. The representation $\Pi'$ is hence a weak functorial transfer of $\pi$.
The proof of Theorem~\ref{theo:globl functorial lift} is complete.
\end{proof}
\begin{remark}\label{remark:archimedean places benefit}
The proof of Theorem~\ref{theo:globl functorial lift} yields a little bit more. Namely, it provides a weak functorial transfer $\Pi$ of $\pi$ such that $\Pi_{\infty}$ is the
local archimedean transfer of $\pi_{\infty}$.
\end{remark}
Recall the definition of coarse transfer from \S~\ref{coarse transfer} (see \eqref{local-FE-gamma}). With this notion we can formulate the following complementary results to
Theorem~\ref{theo:globl functorial lift}.
\begin{theorem}\label{theorem:weak transfer fixes gamma}
Let $\pi$ be a cuspidal representation of $G(\A)$ and let $\Pi$ be a weak functorial transfer of $\pi$.
Then for all places $\nu$ of $F$, $\Pi_{\nu}$ is a coarse transfer of $\pi_{\nu}$.
\end{theorem}
\begin{proof}
The representation $\Pi$ is an irreducible automorphic representation of $\GL_N(\A)$, and there is a (possibly empty) finite set $S_{\pi}$ of finite places such that for all $\nu\not\in S_{\pi}\coprod S_{\infty}$, $\pi_{\nu}$ is unramified and $\Pi_{\nu}$ is its local functorial transfer (in particular, $\Pi_{\nu}$ is unramified).
By the classification of \cite[\S~4]{JS2} and the multiplicativity of the Rankin--Selberg $\gamma$-factors (\cite{JPSS,JS3}), if the statement of the theorem holds for
one weak functorial transfer of $\pi$, it is true for any such transfer. Thus we may assume by Remark~\ref{remark:archimedean places benefit}, that for each $\nu\in S_{\infty}$, $\Pi_{\nu}$ is the archimedean functorial transfer of $\pi_{\nu}$.

Suppose $S_{\pi}$ is nonempty (otherwise, we are done).
For $\nu\notin S_{\pi}$, $\Pi_{\nu}$ is a coarse transfer of $\pi_{\nu}$ by Theorem~\ref{theorem:pi and Pi for unramified pi or archimedean}. This theorem also implies
\begin{align}\label{global:identity}
L^{S_{\pi}}(s,\pi\times\tau)=L^{S_{\pi}}(s,\Pi\times\tau),\qquad L^{S_{\pi}}(1-s,\pi^{\vee}\times\tau^{\vee})=L^{S_{\pi}}(1-s,\Pi^{\vee}\times\tau^{\vee}).
\end{align}

Fix $\nu_0\in S_{\pi}$. We need to show that for any $l\geq1$ and any supercuspidal $\sigma'\in\Irr(\GL_l(F_{\nu}))$,
\begin{align}\label{eq:gamma factors identical for any weak lift 2}
\gamma(s,\pi_{\nu_0}\times\sigma',\psi_{\nu_0})=\gamma(s,\Pi_{\nu_0}\times\sigma',\psi_{\nu_0}).
\end{align}
Now we argue as in the proof of \cite[Proposition 7.2]{CKPS}. Since $\sigma'$ is supercuspidal, we can globalize it to a cuspidal representation
$\sigma$ of $\GL_l(\A)$ such that $\sigma_{\nu}$ is unramified for all finite places $\nu\ne\nu_0$ and $\sigma_{\nu_0}=\sigma'$
(as in \cite[Appendice 1]{GH2}). Since $\sigma$ is globally generic (as a cuspidal representation of $\GL_l(\A)$),
the representations $\sigma_{\nu}$ are all generic. Take $S=S_{\pi}-\{\nu_0\}$.
Let $\eta$ be a unitary Hecke character of $\A^*$ which is sufficiently highly ramified for all $\nu\in S$, depending on
$\pi_{\nu}$ and $\Pi_{\nu}$, and such that $\eta_{\nu_0}=1$. Then for all $\nu\ne\nu_0$,
\begin{align}\label{eq:gamma factors identical for any weak lift 3}
\gamma(s,\pi_{\nu}\times\eta_{\nu}\sigma_{\nu},\psi_{\nu})=\gamma(s,\Pi_{\nu}\times\eta_{\nu}\sigma_{\nu},\psi_{\nu}).
\end{align}
This is clear for $\nu\notin S_{\pi}$ by Theorem~\ref{theorem:pi and Pi for unramified pi or archimedean}, and for $\nu\in S$ follows from Theorem~\ref{theorem:pi and Pi for ramified twisted}. Thus \eqref{eq:gamma factors identical for any weak lift 2} follows from \eqref{eq:gamma factors identical for any weak lift 3} and \eqref{global:identity}, combined with \eqref{eq:crude} and the analogous identity
for the Rankin--Selberg partial $L$-function (\cite{Cogdell2004,CPS}).
\end{proof}

\begin{corollary}\label{corollary:local converse thm}
Let $\localfield$ be a local field.
Every $\pi'\in\Irr(G(\localfield))$ has a coarse transfer (to $\GL_N(\localfield)$).
\end{corollary}
\begin{proof}
Clearly we can assume $\localfield$ is non-archimedean.
If $\pi'$ is supercuspidal, we can take a number field $F$ with a finite place $\nu$ such that $F_{\nu}=\localfield$, and globalize $\pi'$ to
a cuspidal representation $\pi$ of $G(\A)$ such that $\pi_{\nu}=\pi'$ (see e.g., \cite[Appendice 1]{GH2}).
Let $\Pi$ be a weak functorial transfer of $\pi$. By Theorem~\ref{theorem:weak transfer fixes gamma}, $\Pi_{\nu}$ is a coarse transfer of $\pi'$.

In general write $\pi'$ as a constituent of $(\sigma_1\times\ldots\times\sigma_d)\rtimes\pi_0$, where $\sigma_i\in\Irr(\GL_{k_i}(\localfield))$,
$\pi_0\in\Irr(\mathcal{G}_{c-2l}(\localfield))$ with $l=k_1+\ldots+k_d$, and the representations $\sigma_1,\ldots,\sigma_d,\pi_0$ are all supercuspidal. Then if $\Pi_0$ is a coarse transfer of $\pi_0$, any irreducible subquotient of $\times_{i=1}^d\sigma_i\times\Pi_0\times_{i=1}^d\chi_{\pi}\sigma_i^{\vee}$ is a coarse transfer of $\pi'$. This follows immediately from \eqref{eq:multiplicativity I} and the multiplicativity of the Rankin--Selberg $\gamma$-factors (see \cite{JPSS}).
\end{proof}

\appendix

\section{The poles of the \texorpdfstring{$\GL_c\times\GL_k$}{GL(c) x GL(k)} integrals, by Eyal Kaplan}\label{appendix:GL2 poles}

We describe a new family of $\GL_c\times\GL_k$ integrals arising as inner integrals in the doubling construction, and develop the necessary local theory in order to produce the poles in Theorem~\ref{theorem:archimedean producing poles} in \S~\ref{Producing poles}. These integrals were independently constructed by Ginzburg \cite{G8,G7} (see also \cite{LapidMao2018}). (This appendix is independent of \cite{G8,LapidMao2018,G7}.) We use the notation and conventions of \S~\ref{the groups} and \S~\ref{local groups and notation}, except here $G$ is a general linear group.

Let $F$ be a local field (with characteristic $0$). For $\sigma\in\Irr(\GL_l)$ and $a\in F^*$, $\sigma(a)$ denotes the value of the central character of $\sigma$ on $aI_l$.
For our purposes here, a set of poles is a set of pairs $(s,m)$ where $s\in\C$ and $m>0$ is an integer (the multiplicity of $s$ as a pole), and if $F$ is non-archimedean we regard $s$ as a point in $\C/\frac{2\pi\sqrt{-1}}{\log q}\Z$.

Let $c\geq1$ and $G=\GL_c$. Let $\pi\in\Irr(G)$ and $\tau\in\IrrGen(\GL_k)$ where $k\geq1$, and set $H=\GL_{kc}$. As in \S~\ref{local groups and notation}, denote the $(k,c)$ model of $\rho_c(\tau)$ by $W_{\psi}(\rho_c(\tau))$. For $g\in G$, put $\appembeddingOne(g)=\diag(g,I_{(k-1)c})$ and $\appembeddingTwo(g)=\diag(I_c,g,\ldots,g)$, then $\appembeddingOne(g)\appembeddingTwo(g)$ is the
diagonal embedding of $G$ in $H$. Also let $\mathcal{S}(\Mat_{a\times b})$ denote the space of Schwartz functions on $\Mat_{a\times b}$ and for
$\phi\in\mathcal{S}(\Mat_{a\times b})$, $\widehat{\phi}$ will denote its Fourier transform with respect to $\psi$ or $\psi^{-1}$ (according to the context).

For a matrix coefficient $\omega$ of $\pi$, $W\in W_{\psi}(\rho_c(\tau))$, $\phi\in\mathcal{S}(\Mat_c)$ and $s\in\C$, consider the integral
\begin{align*}
Z(s,\omega,W,\phi)=\int\limits_GW(\appembeddingOne(g))\phi(g)\omega(g)|\det g|^{s-(kc-2c+1)/2}\,dg.
\end{align*}
This integral is absolutely convergent in $\Real(s)\gg0$ independently of the data $\omega,W$ and $\phi$, and for a given $s$ can be made nonzero and holomorphic in a small neighborhood of $s$ (or identically $1$ for all $s$, over non-archimedean fields). The convergence is proved using Lemma~\ref{lemma:main convergence} and a bound on the matrix coefficient, as in the case of the integral~\eqref{eq:mult inner start 2}.

The integral $Z(s,\omega,W,\phi)$ extends to a meromorphic function of $s$, which is continuous as a form on
$\pi\times\pi^{\vee}\times\rho_c(\tau)\times\mathcal{S}(\Mat_c)$ over archimedean fields. We will deduce the meromorphic continuation and continuity
(in the restricted setup of representations that are of interest here)
using a multiplicativity result, see the discussion before Theorem~\ref{theorem:gcd}, and we proceed assuming they hold.

When $k>1$, since $W(\left(\begin{smallmatrix}I_c&v\\&I_{(k-1)c}\end{smallmatrix}\right))=\psi(\tr v_1)W(I_{kc})$ where $v_1\in\Mat_c$ is the leftmost block of $v$, the Schwartz function can be omitted from the integrand (see \cite[Lemma~(4.1.5)]{JPSS2}, \cite[Lemma~2.7]{JPSS} and \cite[Proposition~6.1]{Jac5}).
Thus the integral becomes a (continuous) trilinear form on $\pi\times\pi^{\vee}\times\rho_c(\tau)$, which we denote by $Z(s,\omega,W)$.

These integrals interpolate the integrals of \cite{GJ,JPSS,JS3} in the sense that for $k=1$ we obtain the $\GL_c\times\GL_1$ integral of Godement and Jacquet \cite{GJ}, and for $c=1$ --- the $\GL_k\times\GL_1$ integral of \cite{JPSS,JS3}. Since the case of $k=1$ is well understood, henceforth take $k>1$.

The local theory has a global counterpart: A global integral introduced in \cite{G7} which in a right half plane unfolds to an Eulerian integral, whose components are the local integrals here.

Denote by $\gamma^{\mathrm{RS}}(s,\pi\times\tau,\psi)$ the Rankin--Selberg $\gamma$-factor of \cite{JPSS,JS3} (defined even if $\pi$ is not generic).
We start with proving the existence of a $\gamma$-factor arising from a functional equation involving the present integrals.
Let $U_{(c^k)}^{\emptyblockforGL}=\{u\in U_{(c^k)}:u_{1,2}=0\}$ ($u_{1,2}\in\Mat_c$, see \S~\ref{local groups and notation}) and  $\psi^{\emptyblockforGL}=\psi_{k}|_{U_{(c^k)}^{\emptyblockforGL}}$. We recall that by \cite[Lemma~12]{CFKmodels}, for any $W\in W_{\psi}(\rho_c(\tau))$,
\begin{align}\label{eq:additional equivariance of W}
W(\appembeddingOne(g)\appembeddingTwo(g))=\tau(\det g)W(I_{kc}),\qquad\forall g\in G.
\end{align}
\begin{apptheorem}\label{theorem:uniqueness for bilinear RS}
Except for finitely many values of $q^{-s}$ over a non-archimedean field, and except for a discrete set of values of $s$ over an archimedean field,
the space of continuous trilinear forms $\mathcal{B}$ on $\pi\times\pi^{\vee}\times\rho_c(\tau)$ such that
for all $(\varphi,\varphi^{\vee},\xi)$ in the space of $\pi\times\pi^{\vee}\times\rho_c(\tau)$, $g,g'\in G$ and $u\in U_{(c^k)}^{\emptyblockforGL}$,
\begin{align}\label{eq:equivariance for bilinear RS}
&\mathcal{B}(\pi(g)\varphi,\tau^{\vee}(\det g')\pi^{\vee}(g')\varphi^{\vee},\rho_c(\tau)(u\appembeddingOne(g)\appembeddingTwo(g'))\xi)
\\&\qquad=\psi^{\emptyblockforGL}(u)|\det gg'^{-1}|^{-s+(kc-2c+1)/2}\mathcal{B}(\varphi,\varphi^{\vee},\xi),\nonumber
\end{align}
is at most one dimensional.
\end{apptheorem}
\begin{proof}
The proof was given in \cite[Theorem~25]{me14} (when considering linear groups, the proof in \cite{me14} is valid for any local field of characteristic $0$).
\end{proof}
The integral $Z(s,\omega,W)$ satisfies the equivariance properties \eqref{eq:equivariance for bilinear RS}. To define the local functional equation, we introduce a second integral as in \cite{JPSS,JS3}. Define
\begin{align*}
Z^*(s,\omega,W)=\int\limits_G\int\limits_{U_{((k-2)c,c)}^-}W(\diag(I_{(k-1)c},g)\diag(I_c,z)w_{((k-1)c,c)})\omega(g)|\det g|^{
s-1+(kc-2c+1)/2}\,dz\,dg.
\end{align*}
Here for any integers $a,b\geq1$, $w_{(a,b)}=\left(\begin{smallmatrix}&I_a\\I_b\end{smallmatrix}\right)\in\GL_{a+b}$.
The analytic properties of this integral: absolutely convergent in a half plane independent of the data, can be made nonzero and holomorphic for a given $s$, and admits meromorphic continuation which is continuous over archimedean fields, are similar to those of $Z(s,\omega,W)$, except that its domain of absolute convergence is a left half plane. Indeed one can use conjugations by elements of
${}^{w_{((k-1)c,c)}^{-1}}(U_{((k-2)c,2c)}\cap U_{((k-1)c,c)})$ to
eliminate the $dz$-integral using \cite{DM}, independently of $g$ (see \cite[\S~2.6]{JPSS} and \cite[\S~6.1]{Jac5}), then replace $(\omega,W)$ with $(\omega^*,W^*)$ (where $\omega^*(g)=\omega(g^*)$ and $W^*(h)=W(h^*)$) and change $g\mapsto g^*$ in the integral. This argument also implies that the $dz$-integration does not contribute to the set of poles of the integrals $Z^*(s,\omega,W)$. The integral $Z^*(s,\omega,W)$ also satisfies \eqref{eq:equivariance for bilinear RS}. Hence by Theorem~\ref{theorem:uniqueness for bilinear RS} there is a $\gamma$-factor $\gamma(s,\pi\times\tau,\psi)$ such that for all $\omega$ and $W$,
\begin{align}\label{eq:gamma RS}
\gamma(s,\pi\times\tau,\psi)Z(s,\omega,W)=\pi(-1)^{k-1}Z^*(s,\omega,W).
\end{align}
This factor is a well defined, not identically zero meromorphic function of $s$. When $c=1$, $\gamma(s,\pi\times\tau,\psi)=\gamma^{\mathrm{RS}}(s,\pi\times\tau,\psi)$ (apply \cite[(11.7.1)]{Soudry} with the parameters $k=l=1$ in the notation of \textit{loc. cit.}).

Henceforth also assume $\tau$ is unitary. We recall the realization of the $(k,c)$ model of $\rho_c(\tau)$ using a composition of $c$, from \cite[\S~3.2]{CFKmodels}. Let $0<l<c$. Then
\begin{align*}
\rho_c(\tau)\subset\Ind_{R_{(kl,k(c-l))}}^{H}(\Big(W_{\psi}(\rho_l(\tau))\otimes W_{\psi}(\rho_{c-l}(\tau))\Big)\delta_{R_{(kl,k(c-l))}}^{-1/(2k)}).
\end{align*}
For $u\in\Mat_c$ write $u=\left(\begin{smallmatrix}u^1&u^2\\u^3&u^4\end{smallmatrix}\right)$ where $u^1\in\Mat_{l}$ and $u^4\in\Mat_{c-l}$, and if $v\in\Mat_{kc}$, set $v=(v_{i,j})$ with $v_{i,j}\in\Mat_c$. For $1\leq i<j\leq k$ (resp., $1\leq j<i\leq k$) and $1\leq t\leq 4$, let $V_{i,j}^t$ denote the subgroup of $v\in U_{(c^k)}$ (resp., $v\in U_{(c^k)}^-$) satisfying $v_{i',j'}^{t'}=0$ for all $(i',j',t')\ne(i,j,t)$ such that $i'\ne j'$. Identify $V_{i,j}^t$ with $\Mat_{a\times b}$, e.g., $V_{i,j}^1\cong\Mat_{l}$ and $V_{i,j}^3\cong\Mat_{(c-l)\times l}$. Also let $V=V_l<U_{(c^k)}$ be the group generated by the subgroups $V_{i,j}^3$ for all $1\leq i<j\leq k$. The group $V$ is abelian. Set
\begin{align}\label{kappa}
&\kappa=\kappa_{l,c-l}=\left(\begin{smallmatrix}I_l\\0&0&I_l\\0&0&0&0&I_l&\ddots\\&&&&&&I_l&0\\0&I_{c-l}\\0&0&0&I_{c-l}&&\ddots\\&&&&&&&I_{c-l}\end{smallmatrix}\right)\in H\qquad
\text{($I_l$ and $I_{c-l}$ each appears $k$ times).}
\end{align}
For any $\xi$ in the space $V_{\rho_c(\tau)}$ of $\rho_c(\tau)$, consider the integral
\begin{align}\label{eq:mnk functional using w}
\xi\mapsto\int\limits_{V}\xi(\kappa v)\,dv.
\end{align}
By \cite[Lemma~9]{CFKmodels}, this integral is absolutely convergent and realizes the $(k,c)$ functional on $\rho_c(\tau)$. Thus we can write
$W(h)=W_{\xi}(h)=\int_V\xi(\kappa vh)\,dv$ ($h\in H$).

\begin{apptheorem}\label{theorem:RS mult}
Let $0<l<c$, $R=R_{(l,c-l)}$ and $\sigma=\pi_1\otimes\pi_2\in\Irr(M_R)$.
If $\pi\subset\Ind_{R}^{G}(\sigma)$, $\gamma(s,\pi\times\tau,\psi)=\prod_{i=1}^2\gamma(s,\pi_i\times\tau,\psi)$.
Thus if $\pi$ is unramified or $F$ is archimedean, $\gamma(s,\pi\times\tau,\psi)=\gamma^{\mathrm{RS}}(s,\pi\times\tau,\psi)$.
\end{apptheorem}
\begin{proof}
The second assertion clearly follows from the first, since in these cases $\pi$ is embedded in a principal series and the $\gamma$-factors coincide for $\GL_1\times\GL_k$.

Assume for the moment $\pi=\Ind_{R}^{G}(\sigma)$; the integral $Z(s,\omega,W)$ is still absolutely convergent in $\Real(s)\gg0$ (though $\Ind_{R}^{G}(\sigma)$ may be reducible). Fix the canonical pairing $\langle,\rangle$ on
$\sigma\otimes\sigma^{\vee}$ and let $\varphi$ (resp., $\varphi^{\vee}$) belong to the space of $\pi$ (resp., $\pi^{\vee}$).
In $\Real(s)\gg0$, $Z(s,\omega,W)$ equals
\begin{align}
\label{eq:poles GL2 start 00}
&\int\limits_{G^{\triangle}\backslash G\times G}W(\appembeddingOne(g_1)\appembeddingTwo(g_2))|\det g_1g_2^{-1}|^{s-(kc-2c+1)/2}
\tau^{\vee}(\det g_2)\int\limits_{R\backslash G}\langle\varphi(g_0g_1),\varphi^{\vee}(g_0g_2)\rangle\,dg_0\,dg_1\,dg_2
\\&=\int\limits_{R\times R\backslash G\times G}\label{eq:poles GL2 start}
\int\limits_{M_R}\int\limits_{U_R}\delta_{R}^{-1/2}(m)
W(\appembeddingOne(bmg_1)\appembeddingTwo(g_2))|\det m|^{s-(kc-2c+1)/2}
\\&\quad\tau^{\vee}(\det g_2)
\langle\sigma(m)\varphi(g_1),\varphi^{\vee}(g_2)\rangle\,db\,dm\,d(g_1,g_2)\nonumber.
\end{align}
(Cf. \cite[\S~4]{LR}.) Note that $|\det g_1|=|\det g_2|=1$ in \eqref{eq:poles GL2 start} because the outer integral is over $K_G\times K_G$.
We compare this integral to $Z^*(s,\omega,W)$, computed in $\Real(s)\ll0$. Set $w=w_{((k-1)c,c)}$. Since $\diag(g_2,\ldots,g_2,g_1)$ normalizes $\diag(I_c,U_{((k-2)c,c)}^-)$ and ${}^{w^{-1}}\diag(g_2,\ldots,g_2,g_1)=\appembeddingOne(g_1)\appembeddingTwo(g_2)$, $Z^*(s,\omega,W)$ becomes
\begin{align}
&\int\limits_{R\times R\backslash G\times G}\label{eq:poles GL2 start second}
\int\limits_{M_R}\int\limits_{U_R}
\int\limits_{U_{((k-2)c,c)}^-}\delta_{R}^{-1/2}(m)
W(\diag(I_{(k-1)c},bm)\diag(I_c,z)w\appembeddingOne(g_1)\appembeddingTwo(g_2))\\&\quad|\det m|^{s-1+(kc-2c+1)/2}
\tau^{\vee}(\det g_2)
\langle\sigma(m)\varphi(g_1),\varphi^{\vee}(g_2)\rangle\,dz\,db\,dm\,d(g_1,g_2).\nonumber
\end{align}

Since $R\backslash G$ is compact, to compare between \eqref{eq:poles GL2 start} and \eqref{eq:poles GL2 start second} it suffices to
prove that the inner $db\,dm$-integral of \eqref{eq:poles GL2 start} is the product of integrals for
$\pi_1\times\tau$ and $\pi_2\times\tau$, and the $dz\,db\,dm$-integral of \eqref{eq:poles GL2 start second} is the product of proportional integrals for $\pi_1\times\tau$ and $\pi_2\times\tau$ (by the factor $\prod_{i=1}^2\pi_i(-1)^{1-k}\gamma(s,\pi_i\times\tau,\psi)$).
We start with the inner integral of \eqref{eq:poles GL2 start}, which takes the form
\begin{align}\label{eq:poles GL2 start 4}
&\int\limits_{M_R}\int\limits_{U_R}\delta_{R}^{-1/2}(m)W(\appembeddingOne(bm))|\det m|^{s-(kc-2c+1)/2}
\langle\sigma(m)\varphi(I_c),\varphi^{\vee}(I_c)\rangle\,db\,dm.
\end{align}
By \eqref{eq:mnk functional using w} with $l$ and writing $W(h)=\int_V\xi(\kappa vh)\,dv$, we obtain
\begin{align}\label{eq:poles GL2 start 2}
&\int\limits_{M_R}\int\limits_{U_R}
\int\limits_V\delta_{R}^{-1/2}(m)\xi(\kappa v\appembeddingOne(bm))|\det m|^{s-(kc-2c+1)/2}
\langle\sigma(m)\varphi(I_c),\varphi^{\vee}(I_c)\rangle\,dv\,db\,dm.
\end{align}
Formally changing the order $dv\,db$ to $db\,dv$, we obtain an inner integral
$\int_{U_R}\psi^{-1}(\tr{bv_{1,2}^3})\,db$ where we identify $U_R$ with $\Mat_{l\times (c-l)}$. This integral vanishes unless $v_{1,2}^3=0$, whence \eqref{eq:poles GL2 start 2} becomes
\begin{align}\label{eq:poles GL2 start 3}
&\int\limits_{M_R}
\int\limits_{\{v\in V:v_{1,2}^3=0\}}\delta_{R}^{-1/2}(m)\xi(\kappa v\appembeddingOne(m))|\det m|^{s-(kc-2c+1)/2}
\langle\sigma(m)\varphi(I_c),\varphi^{\vee}(I_c)\rangle\,dv\,dm.
\end{align}

To justify the passage \eqref{eq:poles GL2 start 2}--\eqref{eq:poles GL2 start 3}, replace $\xi$ with $\phi(\xi)(h)=\int_{V_{1,2}^3}x\cdot \xi(h)\phi(x)\,dx$ for $\phi\in\mathcal{S}(V_{1,2}^3)$. Put $m=\diag(m_1,m_2)$ and identify $m_2$ with $\diag(I_l,m_2)$. Observe that
$x\in V_{1,2}^3$ commutes with $m_1$ and with the elements of $V$ ($V_{1,2}^3<V$ and $V$ is abelian). Then we have the following ``Ghost Train"
\footnote{These convolutions are always there, hidden from sight, appearing one after the other as a train...} argument:
\begin{align*}
&\int\limits_{U_R}\int\limits_V\phi(\xi)(\kappa v\appembeddingOne(bm))\,dv\,db
=\int\limits_{U_R}\int\limits_V\int\limits_{V_{1,2}^3}\xi(\kappa v\appembeddingOne(bm))\phi(x)\psi(\tr bm_2x)\,dx\,dv\,db
\\&=\int\limits_{U_R}\int\limits_V\xi(\kappa v\appembeddingOne(bm))\widehat{\phi}(bm_2)\,dv\,db
=\int\limits_{U_R}\int\limits_V\xi(\kappa v\appembeddingOne(m))\psi^{-1}(\tr v_{1,2}^3b)\widehat{\phi}(bm_2)\,dv\,db
\\&=\int\limits_{U_R}\int\limits_{\{v\in V:v_{1,2}^3=0\}}\int\limits_{V_{1,2}^3}\xi(\kappa v\appembeddingOne(m)({}^{\appembeddingOne(m)^{-1}}v_{1,2}^3))\psi^{-1}(\tr v_{1,2}^3b)\widehat{\phi}(bm_2)\,dv_{1,2}^3\,dv\,db
\\&
=\int\limits_{U_R}\int\limits_{\{v\in V:v_{1,2}^3=0\}}\int\limits_{V_{1,2}^3}\xi(\kappa v\appembeddingOne(m)v_{1,2}^3)\psi^{-1}(\tr v_{1,2}^3b)\widehat{\phi}(b)\,dv_{1,2}^3\,dv\,db
\\&=\int\limits_{\{v\in V:v_{1,2}^3=0\}}\int\limits_{V_{1,2}^3}\xi(\kappa v\appembeddingOne(m)v_{1,2}^3)\left(\int\limits_{U_R}\psi^{-1}(\tr v_{1,2}^3b)\widehat{\phi}(b)\,db\right)\,dv_{1,2}^3\,dv
\\&=\int\limits_{\{v\in V:v_{1,2}^3=0\}}\int\limits_{V_{1,2}^3}\xi(\kappa v\appembeddingOne(m)v_{1,2}^3)\phi(v_{1,2}^3)\,dv_{1,2}^3\,dv
=\int\limits_{\{v\in V:v_{1,2}^3=0\}}\phi(\xi)(\kappa v\appembeddingOne(m))\,dv.
\end{align*}
Here two equalities before the last, we may change the order
$(dv_{1,2}^3\,dv)\,db\mapsto db\,(dv_{1,2}^3\,dv)$ because the $dv$-integral is absolutely convergent and
because of the Schwartz function $\widehat{\phi}$. Since any $\xi$ is a finite sum of $\phi_i(\xi_i)$ for some $\xi_i\in V_{\rho_c(\tau)}$ and
$\phi_i\in\mathcal{S}(V_{1,2}^3)$ (\cite{DM}), this identity applies in general.

We proceed with \eqref{eq:poles GL2 start 3}. Consider the abelian group
\begin{align*}
X=&\left\{\left(\begin{smallmatrix}
I_l&&&&0&0&\cdots&0\\
&I_l&&&\vdots&x_{1,2}&\ddots&\vdots\\
&&\ddots&&\vdots&\vdots&\ddots&0\\
&&&I_l&0&x_{1,k}&\cdots&x_{k-1,k}\\
&&&&I_{c-l}\\
&&&&&I_{c-l}\\
&&&&&&\ddots\\
&&&&&&&I_{c-l}
\end{smallmatrix}\right):x_{i,j}\in\Mat_{l\times(c-l)}\right\}<U_{(kl,k(c-l))}.
\end{align*}
Let $X_{i,j}<X$ be the subgroup such that $x_{i',j'}=0$ for all $(i',j')\ne(i,j)$, $X_{i,j}\cong\Mat_{l\times(c-l)}$, and
denote an element of $X_{i,j}$ by $x_{i,j}$.
One can argue as in \cite[Lemma~9]{CFKmodels} to replace the $dv$-integration of $\xi$ with a convolution of $\xi$ against Schwartz functions. That argument will already be independent of $m_1$, but not of $m_2$. Here because we already have $v_{1,2}^3=0$, we can make it independent of both $m_1$ and $m_2$, by modifying $X_{i,j}$ for $i=1$.
(We will not be using $X_{1,2}$.)

Recall the order $(k-1,k),(k-2,k-1),\ldots,(1,2)$, $(k-2,k),(k-3,k-1),\ldots,(1,3),\ldots,(1,k)$ of handling the blocks $V_{i,j}^3$ of $V$ in \cite[Lemma~9]{CFKmodels}. For $i>1$ put
$X'_{i,j}={}^{\kappa^{-1}}X_{i,j}$ and for $j>2$ let $X'_{1,j}=V_{j-1,1}^2$ ($V_{j-1,1}^2\cong\Mat_{l\times(c-l)}$). Observe that for $x'_{1,j}\in X'_{1,j}$, $v_{1,j}^3\in V_{1,j}^3$ and $j>2$, we have ${}^{\appembeddingOne(m)}x'_{1,j}=x'_{1,j}m_2^{-1}$,
$\xi(\kappa\,{}^{{x'_{1,j}}^{-1}}v_{1,j}^3)=\psi^{-1}(\tr x'_{1,j}v_{1,j}^3)\xi(\kappa v_{1,j}^3)$ and moreover, assuming $v$ does not contain coordinates $v_{i,j}^3$ which appear prior to $(1,j)$ (with respect to the above order), $\xi(\kappa\,\cdot)$ transforms trivially on the left under the coordinates outside $V$ that appear in the conjugation ${}^{{x'_{1,j}}^{-1}}v$. Also denote for $\xi\in V_{\rho_c(\tau)}$ and $\phi\in\mathcal{S}(X'_{i,j})$,
\begin{align}\label{eq:xi i j}
\xi_{i,j}(h)=\int\limits_{X'_{i,j}}x\cdot\xi(h)\phi(x)\,dx,\qquad \xi'_{i,j}(h)=\int\limits_{V_{i,j}^3}v\cdot \xi(h)\widehat{\phi}(v)\,dv.
\end{align}
First we change variables $v_{1,j}^3\mapsto m_2v_{1,j}^3$ and altogether multiply the integrand by $|\det m_2|^{(k-2)l}$.
We proceed as in the proof of \cite[Lemma~9]{CFKmodels} using the same order, except for $(i,j)=(1,2)$ which we already handled above (using $b$).
Briefly, by \cite{DM} any $\xi$ is a finite sum of $(\xi^{n})_{k-1,k}$ defined by \eqref{eq:xi i j} for some finite set of pairs $\{(\xi^{n},\phi^{n})\}_{n}$. Computing
\eqref{eq:poles GL2 start 3} for each $(\xi^{n})_{k-1,k}$, it becomes the similar integral
but with $(\xi^{n})'_{k-1,k}$ and the $dv$-integral is computed over $\{v\in V:v_{1,2}^3=v_{k-1,k}^3=0\}$. Since these manipulations do not affect the matrix coefficient, we can re-denote $\xi=\sum_n(\xi^{n})'_{k-1,k}$ then
\eqref{eq:poles GL2 start 3} equals
\begin{align}\label{eq:poles GL2 start 3.1}
&\int\limits_{M_R}
\int\limits_{\{v\in V:v_{1,2}^3=v_{k-1,k}^3=0\}}\delta_{R}^{-1/2}(m)\xi(\kappa v\appembeddingOne(m))|\det m_1|^{s-(kc-2c+1)/2}
|\det m_2|^{s+(k-2)l-(kc-2c+1)/2}
\\&\langle\sigma(m)\varphi(I_c),\varphi^{\vee}(I_c)\rangle\,dv\,dm.\nonumber
\end{align}
Again by \cite{DM}, $\xi=\sum_n(\xi^{n})_{k-2,k-1}$ for some $\{(\xi^{n},\phi^{n})\}_{n}$ and we can proceed similarly.
Altogether we deduce that the integrand is a Schwartz function of $v$ independently of $m$ (after the changes $v_{1,j}^3\mapsto m_2v_{1,j}^3$) and \eqref{eq:poles GL2 start 3.1} becomes
\begin{align}\label{int:GL Lemma after removing V}
&\int\limits_{M_R}
\xi(\kappa \appembeddingOne(m))\langle\sigma(m)\varphi(I_c),\varphi^{\vee}(I_c)\rangle|\det m_1|^{s-(kc-c-l+1)/2}
|\det m_2|^{s+(k-2)l-(kc-2c-l+1)/2}\,dm
\\&=\prod_{i=1}^2Z(s,\omega_i,W_i).\nonumber
\end{align}
Here $Z(s,\omega_i,W_i)$ is the integral for $\pi_i\times\tau$, and the functions $W_1\in W_{\psi}(\rho_{l}(\tau))$ and
$W_2\in W_{\psi}(\rho_{c-l}(\tau))$ are defined by
\begin{align*}
W_i(h_i)=\delta_{R_{(kl,k(c-l))}}^{(1-k)/(2k)}(h_i)\xi(h_i\kappa),
\end{align*}
where $h_1\in\GL_{kl}$ and $h_2\in\GL_{k(c-l)}$ are identified on the r.h.s.~ with their natural images in $M_{(kl,k(c-l))}$. Note that for each $\xi\in V_{\rho_c(\tau)}$ there is $\xi_0\in V_{\rho_c(\tau)}$ such that \eqref{eq:poles GL2 start 3} with $\xi_0$ equals \eqref{int:GL Lemma after removing V} with $\xi$, so that we can assume the functions $W_i$ are arbitrary.

Now consider the inner integral of \eqref{eq:poles GL2 start second},
\begin{align*}
&\int\limits_{M_R}\int\limits_{U_R}\int\limits_{U_{((k-2)c,c)}^-}
\int\limits_V\delta_{R}^{-1/2}(m)\xi(\kappa v\diag(I_{(k-1)c},bm)\diag(I_c,z)w)|\det m|^{s-1+(kc-2c+1)/2}
\\&\nonumber\langle\sigma(m)\varphi(I_c),\varphi^{\vee}(I_c)\rangle\,dv\,dz\,db\,dm.
\end{align*}
The step analogous to \eqref{eq:poles GL2 start 2}--\eqref{eq:poles GL2 start 3} but with $v_{k-1,k}^3$ implies that this integral equals
\begin{align}\label{eq:poles Z* GL2 start 2}
&\int\limits_{M_R}\int\limits_{U_{((k-2)c,c)}^-}
\int\limits_{\{v\in V:v_{k-1,k}^3=0\}}\delta_{R}^{-1/2}(m)\xi(\kappa v\diag(I_{(k-1)c},m)\diag(I_c,z)w)|\det m|^{s-1+(kc-2c+1)/2}
\\&\nonumber\langle\sigma(m)\varphi(I_c),\varphi^{\vee}(I_c)\rangle\,dv\,dz\,dm.
\end{align}
For the justification we use $\phi(\xi)(h)=\int_{V_{k-1,k}^3}x\cdot \xi(h)\phi(x)\,dx$, and as above prove that
\begin{align*}
&\int\limits_{U_R}\int\limits_V\phi(\xi)(\kappa v\diag(I_{(k-1)c},bm))\,dv\,db
=\int\limits_{\{v\in V:v_{k-1,k}^3=0\}}\phi(\xi)(\kappa v\diag(I_{(k-1)c},m))\,dv.
\end{align*}
This applies to all $\xi$ ($\xi$ is a sum of $\phi_i(\xi_i)$), in particular to $(\diag(I_c,z)w)\cdot\xi$ (we justify this step for each fixed $z$).

Set $Z=\diag(I_c,U_{((k-2)c,c)}^-)$. Using our notation for the blocks of $U_{(c^k)}$, $Z$ is the (direct) product of subgroups
$\{V_{k,j}^t\}_{1<j<k,1\leq t\leq 4}$ and we let $Z^t$ (resp., $Z_{k,j}^t$) denote the subgroup of $Z$ where all blocks other than
$\{V_{k,j}^t\}_{1<j<k}$ (resp., $V_{k,j}^t$) are zero. Then $z_{k,j}^t$ will denote an element of $Z_{k,j}^t$.
Note that $V_{k,j}^1\in\Mat_l$ and $V_{k,j}^4\in\Mat_{c-l}$.

We use $Z^2$ to eliminate the blocks $\{v_{d,k}^3\}_{1\leq d\leq k-2}$ of $v$. Proceeding in decreasing order $d=k-2,\ldots,1$ and assuming
$v$ does not contain the coordinates $\{v_{i,k}^3\}_{d<i\leq k-2}$,
\begin{align*}
\xi(\kappa v\diag(I_{(k-1)c},m)z_{k,d+1}^2)=\psi(\tr v_{d,k}^3m_1z_{k,d+1}^2)\xi(\kappa v\diag(I_{(k-1)c},m)).
\end{align*}
We change variables $v_{d,k}^3\mapsto v_{d,k}^3m_1^{-1}$. Then as above formally integrating over $Z_{k,d+1}^2$ first, we eliminate $v_{d,k}^3$ (note that $v$ appears to the left of $\diag(I_{(k-1)c},m)$!). For the justification repeat the Ghost Train argument above
using $\phi(\xi)(h)=\int_{V_{d,k}^3}x\cdot \xi(h)\phi(x)\,dx$. As with
$V_{k-1,k}^3$, for each $d$ this argument is carried out without the outer integration on the remaining coordinates $z'$ of $z$ other than
$z_{k,d+1}^3$, then the identity is valid for $(\diag(I_c,z')w)\cdot\xi$ as well. Following this step the integrand is multiplied by $|\det m_1|^{(2-k)(c-l)}$ and \eqref{eq:poles Z* GL2 start 2} becomes
\begin{align}\label{eq:poles Z* GL2 start 3}
&\int\limits_{M_R}\int\limits_{Z^1Z^3Z^4}
\int\limits_{\{v\in V:v_{i,k}^3=0,1\leq i<k\}}\delta_{R}^{-1/2}(m)\xi(\kappa v\diag(I_{(k-1)c},m)\diag(I_c,z)w)|\det m_1|^{s-1+(k-2)l+c-(kc-1)/2}
\\&\nonumber|\det m_2|^{s-1+(kc-2c+1)/2}\langle\sigma(m)\varphi(I_c),\varphi^{\vee}(I_c)\rangle\,dv\,dz\,dm.
\end{align}

Next we apply the same substitutions of $\xi$ as above (\eqref{eq:poles GL2 start 3}--\eqref{int:GL Lemma after removing V}), to handle the integration over $v$, except that when we translate $\xi$ on the right by $x'_{i,j}$,
$v$ is conjugated by ${{}^wx'_{i,j}}^{-1}$ instead of ${x'_{i,j}}^{-1}$, and we also have the integration over the remaining subgroups $Z^1$, $Z^3$ and $Z^4$ of $Z$. Observe that for $j>2$, ${}^wX'_{1,j}<U_{((k-3)c+l,3c-l)}\cap U_{((k-1)c,c)}\cap U_{((k-1)c+l,c-l)}$ and in particular ${}^wX'_{1,j} <U_{((k-2)c,2c)}\cap U_{((k-1)c,c)}$, whence conjugation by these elements eliminates $Z^3$: The conjugation of $Z$ by ${}^wX'_{1,d+2}$ will be used below to zero out $z_{k,d+1}^3$, $1\leq d\leq k-2$.

We started with writing $\xi=\sum_n(\xi^{n})_{k-1,k}$. Repeating this here, and since $v_{k-1,k}^3$ is already $0$ (using $b$),
the conjugation of $v$ by ${{}^wx'_{k-1,k}}^{-1}$ is used to zero out $v_{k-2,k-1}^3$, and because ${}^{w^{-1}}V_{k-2,k-1}^3=V_{k-1,k}^3$, the integral~\eqref{eq:poles Z* GL2 start 3} with $(\xi^{n})_{k-1,k}$ becomes the similar integral with $(\xi^{n})'_{k-1,k}$ and $v$ satisfies $v_{k-2,k-1}^3=0$. Re-denote
$\xi=\sum_n(\xi^{n})'_{k-1,k}$. Next write $\xi=\sum_n(\xi^{n})_{k-2,k-1}$ and proceed to eliminate $v_{k-3,k-2}^3$, etc., except ${}^wx'_{1,d+2}$ is used for $z_{k,d+1}^3$ (instead of $v_{1,d+2}^3$) and $v_{d,k}^3$ was already zeroed out by $z_{k,d+1}^2$.
Note that $x'_{i,j}$ with $i>1$ was used above to eliminate $v_{i,j}^3$, here ${}^wx'_{i,j}$ eliminates $v_{i-1,j-1}^3$ and
${}^{w^{-1}}V_{i-1,j-1}^3=V_{i,j}^3$, and for
$1\leq d\leq k-2$, ${}^{w^{-1}}Z_{k,d+1}^3=V_{1,d+2}^3$. This guarantees that the passage from $(\xi^{n})_{i,j}$ to
$(\xi^{n})'_{i,j}$ above will remain the same here, even though here the Fourier transform is over coordinates of $V_{i-1,j-1}^3$ or
$Z_{k,d+1}^3$.

The resulting unipotent subgroup of $Z$ is $Z^1\cdot Z^4$. We obtain the integral
\begin{align}\label{int:GL Lemma after removing V second}
&\int\limits_{M_R}\int\limits_{Z^1Z^4}
\xi(\kappa \diag(I_{(k-1)c},m)\diag(I_c,z)w)|\det m_1|^{s-1+kl-(kc-c+3l-1)/2}
|\det m_2|^{s-1+(kc-2c+l+1)/2}\\&
\langle\sigma(m)\varphi(I_c),\varphi^{\vee}(I_c)\rangle\,dz\,dm.\nonumber
\end{align}
Now observe that
\begin{align*}
&{}^{\kappa}\diag(I_{(k-1)c},m)=\diag(I_{(k-1)l},m_1,I_{(k-1)(c-l)},m_2),\\
&{}^{\kappa}(Z^1Z^4)=\diag(I_l,U_{((k-2)l,l)}^-,I_{c-l},U_{((k-2)(c-l),c-l)}^-),
\quad{}^{\kappa}w=\diag(w_{((k-1)l,l)},w_{((k-1)(c-l),c-l)}).
\end{align*}
Hence we reach
$\prod_{i=1}^2Z^*(s,\omega_i,W_i)$. Comparing this to \eqref{int:GL Lemma after removing V}
using \eqref{eq:gamma RS} shows that the inner integrals of
\eqref{eq:poles GL2 start} and \eqref{eq:poles GL2 start second} are proportional by
$\prod_{i=1}^2\pi_i(-1)^{1-k}\gamma(s,\pi_i\times\tau,\psi)$. This proves the identity
$\prod_{i=1}^2\gamma(s,\pi_i\times\tau,\psi)Z(s,\omega,W)=\pi(-1)^{k-1}Z^*(s,\omega,W)$.
Hence for each irreducible $\pi'\subset\Ind_R^G(\sigma)$,
$\gamma(s,\pi'\times\tau,\psi)=\prod_{i=1}^2\gamma(s,\pi_i\times\tau,\psi)$, and in particular for $\pi$.
\end{proof}
\begin{appexample}
When $c=k=2$,
\begin{align*}
\kappa=\left(\begin{smallmatrix}1&\\&&1\\&1\\&&&1\end{smallmatrix}\right),\quad
v=\left(\begin{smallmatrix}1&&\\&1&v_{1,2}^3&\\&&1\\&&&1\end{smallmatrix}\right),\quad
\appembeddingOne(b)=
\left(\begin{smallmatrix}1&b\\&1\\&&1\\&&&1\end{smallmatrix}\right),
\quad
\appembeddingOne(m)=
\left(\begin{smallmatrix}m_1&\\&m_2&\\&&1\\&&&1\end{smallmatrix}\right)
\end{align*}
and the justification to \eqref{eq:poles GL2 start 2}--\eqref{eq:poles GL2 start 3} is simple to see.
\end{appexample}
\begin{appremark}\label{remark:ind of t unr}
The identity $\int_{U_R}\int_V\xi(\kappa v \appembeddingOne(bm))\,dv\,db=\xi(\kappa \appembeddingOne(m))$ was first observed in the unramified case for $R=B_G$ in \cite{G8,G7}.
\end{appremark}
\begin{appexample}\label{example:removing v}
Consider $c=2$ and $k=4$. We have
\begin{align*}
&\kappa=\left(\begin{smallmatrix}
    1 &  &  &  &  &  &  &  \\
     & 0 & 1 &  &  &  &  &  \\
     &  &  & 0 & 1 &  &  &  \\
     &  &  &  &  & 0 & 1 &  \\
     0& 1 &  &  &  &  &  &  \\
     &  & 0 & 1 &  &  &  &  \\
     &  &  &  & 0 & 1 &  &  \\
     &  &  &  &  &  & 0 & 1
  \end{smallmatrix}\right),\qquad
v=\left(\begin{smallmatrix}
    1 &  &  &  &  &  &  &  \\
     & 1 & v_{1,2}^3  & & v_{1,3}^3 &  & v_{1,4}^3 &  \\
     &  & 1 &  &  &  &  &  \\
     &  &  & 1 & v_{2,3}^3  &  & v_{2,4}^3 &  \\
     &  &  &  & 1 &  &  & \\
     &  &  &  &  & 1 & v_{3,4}^3 &  \\
     &  &  &  &  &  & 1 &  \\
     &  &  &  &  &  &  & 1
  \end{smallmatrix}\right),
  \qquad x=\left(\begin{smallmatrix}
    1 &  &  &  &  &  &  &  \\
     & 1 &   & &  &  &  &  \\
     & x_{1,3} & 1 &  &  &  &  &  \\
     &  &  & 1 &   &  &  &  \\
     & x_{1,4} &  &  & 1 & x_{2,3} &  & \\
     &  &  &  &  & 1 &  &  \\
     &  &  &  &  & x_{2,4} & 1 & x_{3,4} \\
     &  &  &  &  &  &  & 1
  \end{smallmatrix}\right),
\\&w=\left(\begin{smallmatrix}&I_6\\I_2\end{smallmatrix}\right),\qquad{}^wx=\left(\begin{smallmatrix}
    1 &  &  &  &  &  &  & x_{1,3} \\
     & 1 &   & &  &  &  &  \\
     &  & 1 & x_{2,3} &  &  &  & x_{1,4} \\
     &  &  & 1 &   &  &  &  \\
     &  &  & x_{2,4} & 1 & x_{3,4} &  & \\
     &  &  &  &  & 1 &  &  \\
     &  &  &  &  &  & 1 &  \\
     &  &  &  &  &  &  & 1
  \end{smallmatrix}\right),\qquad
  z=\left(\begin{smallmatrix}
    1 &  &  &  &  &  &  &  \\
     & 1 &   & &  &  &  &  \\
     &  & 1 &  &  &  &  &  \\
     &  &  & 1 &   &  &  &  \\
     &  &  &  & 1 &  &  & \\
     &  &  &  &  & 1 &  &  \\
     &  & z_{4,2}^1 & z_{4,2}^2 & z_{4,3}^1 & z_{4,3}^2 & 1 &  \\
     &  & z_{4,2}^3 & z_{4,2}^4 & z_{4,3}^3 & z_{4,3}^4 &  &  1
  \end{smallmatrix}\right).
\end{align*}
Here the elements corresponding to $X'_{i,j}$ appear together as a matrix $x$. To handle the coordinates of $v$ and $z$ in \eqref{eq:poles Z* GL2 start 2}, first note that $v_{3,4}^3=0$. Then use $z_{4,3}^2$ to remove $v_{2,4}^3$, $z_{4,2}^2$ for $v_{1,4}^3$,
$x_{3,4}$ for $v_{2,3}^3$, $x_{2,3}$ for $v_{1,2}^3$,
$x_{2,4}$ for $v_{1,3}^3$, $x_{1,3}$ for $z_{4,2}^3$ and $x_{1,4}$ for $z_{4,3}^3$.
\end{appexample}

Recall that $\pi\in\Irr(G)$, $\tau\in\IrrGen(\GL_k)$, $\tau$ is also unitary and $k>1$. Further assume that
either $F$ is non-archimedean and $\pi$ is unramified (but $\tau$ is not assumed to be unramified), or $F$ is archimedean, until the end of the appendix.
In these cases the proof of Theorem~\ref{theorem:RS mult} reduces the properties of
meromorphic continuation and continuity of the continuation in the input data of both $Z(s,\omega,W)$ and $Z^*(s,\omega,W)$,
to those similar properties for Rankin--Selberg integrals for $\GL_1\times\GL_k$, which are known (\cite{JPSS,JS3,Jac5}).
This follows from \eqref{eq:poles GL2 start} and \eqref{eq:poles GL2 start second}, where we expressed $Z(s,\omega,W)$ and $Z^*(s,\omega,W)$ as
iterated integrals with a compact outer domain $R\times R\backslash G\times G$, and from \eqref{int:GL Lemma after removing V}
and \eqref{int:GL Lemma after removing V second} (the passage \eqref{eq:poles GL2 start 4}--\eqref{int:GL Lemma after removing V}
was justified in general, i.e., not for a specific choice of data). Refer to \cite[\S~5, Lemma~1]{Soudry3} and \cite[\S~5, Lemma~2, p.~199]{Soudry3} for more details. Of course over non-archimedean fields the meromorphic continuation already follows from
Theorem~\ref{theorem:uniqueness for bilinear RS}, the fact that the integral can be made constant, and Bernstein's continuation principle (\cite{Banks}), for all
irreducible representations $\pi$ and $\tau$ such that $\tau$ is generic.

For brevity let $\mathcal{F}$ be the ring of entire functions, $\mathcal{M}$ be its field of fractions, i.e., the field of meromorphic functions, and $\mathcal{F}^*$ be the group of entire nowhere vanishing functions (in the non-archimedean case
$\mathcal{F}=\C[q^{-s},q^s]$ and $\mathcal{F}^*=\C[q^{-s},q^s]^*$, see \S~\ref{local groups and notation}). For $f(s),f'(s)\in\mathcal{M}$ write
$f(s)\simeq f'(s)$ if $f(s)=e(s)f'(s)$ for some $e(s)\in\mathcal{F}^*$.
\begin{apptheorem}\label{theorem:gcd}
There is a nowhere vanishing function $\gcdzzz(s,\pi\times\tau)\in\mathcal{M}$ such that
$\gcdzzz(s,\pi\times\tau)^{-1}Z(s,\omega,W)\in\mathcal{F}$ for all data, and for each $s$ the product
$\gcdzzz(s,\pi\times\tau)^{-1}Z(s,\omega,W)$ can be made nonzero.
In the non-archimedean case we can assume $\gcdzzz(s,\pi\times\tau)=P(q^{-s})^{-1}$ where $P\in\C[X]$ satisfies $P(0)=1$; otherwise
$\gcdzzz(s,\pi\times\tau)$ is defined up to an element of $\mathcal{F}^*$.
In addition:
\begin{enumerate}[leftmargin=*]
\item An equivalent definition of $\gcdzzz(s,\pi\times\tau)$ can be given using the integrals $Z^*(1-s,\omega',W')$ where $\omega'$ is a matrix coefficient of $\pi^{\vee}$ and $W'\in W_{\psi}(\rho_c(\tau^{\vee}))$.
\item\label{it:theorem 3.1} With the notation of Theorem~\ref{theorem:RS mult}, $\gcdzzz(s,\pi\times\tau)=A(s)\prod_{i=1}^2\gcdzzz(s,\pi_i\times\tau)$ for $A(s)\in\mathcal{F}$.
\item $\gcdzzz(s,\pi\times\tau)$ is independent of $\psi$.
\item $\gcdzzz(s,\absdet^{s_0}\pi\times\tau)=\gcdzzz(s,\pi\times\absdet^{s_0}\tau)=\gcdzzz(s+s_0,\pi\times\tau)$.
\end{enumerate}
\end{apptheorem}
\begin{proof}
As observed in the beginning of the appendix, the additional unipotent integration in $Z^*(s,\omega,W)$ does not change the set of poles of the integrals, and for each $s$, $Z^*(s,\omega,W)$ can be made holomorphic and nonzero. (Recall that sets of poles also include multiplicities.) Thus for the purpose of defining $\gcdzzz(s,\pi\times\tau)$ one can consider either
family of integrals ($Z(s,\omega,W)$ or $Z^*(1-s,\omega',W')$). Once $\gcdzzz(s,\pi\times\tau)$ is defined, to see it is independent of $\psi$ use left translations of $W$ by $\diag(aI_c,\ldots,a^kI_c)$ for $a\in F^*$;
and the assertion on unramified twisting is also clear.

Consider \eqref{eq:poles GL2 start}. Since $R\backslash G$ is compact and the integrals \eqref{eq:poles GL2 start 3} are continuous (as forms), the set of poles of \eqref{eq:poles GL2 start} is contained in the set of poles of integrals of the form \eqref{eq:poles GL2 start 3} (\eqref{eq:poles GL2 start 2}--\eqref{eq:poles GL2 start 3} was justified in general). Now the step
\eqref{eq:poles GL2 start 3}--\eqref{int:GL Lemma after removing V} shows that the integrand of \eqref{eq:poles GL2 start 3} is a Schwartz function of $v$ independently of $m$, and the poles are contained in the poles of a product of
$\GL_l\times\GL_k$ and $\GL_{c-l}\times\GL_k$ integrals. This already implies \eqref{it:theorem 3.1} once $\gcdzzz(s,\pi\times\tau)$ is defined.
Applying this inductively, the poles are contained in the poles of a product of
Rankin--Selberg $\GL_1\times\GL_k$ integrals, and thereby in a product of Rankin--Selberg $L$-factors $L(s,\pi_i\times\tau)$ for some quasi-characters $\pi_i$ depending on $\pi$.

If $F$ is non-archimedean, the integrals are rational functions in $q^{-s}$. We deduce that the set of possible poles for the integrals is finite,
and in particular so are their multiplicities. It then follows that these integrals span a fractional ideal of $\C[q^{-s},q^s]$ containing $1$, and the existence of $\gcdzzz(s,\pi\times\tau)$ and its form are now clear.

Assume $F$ is archimedean. Now the argument above shows that the set of poles of the integrals is contained in the set of poles of finitely
many twisted classical Gamma functions ($\Gamma(r_is+d_i)$ for some constants $r_i\in\{1,1/2\}, d_i\in\C$). In addition, for each $s$ one can choose data such that the integral is finite and nonzero at $s$. This implies that we can define $\gcdzzz(s,\pi\times\tau)\in\mathcal{M}$ such that $\gcdzzz(s,\pi\times\tau)^{-1}Z(s,\omega,W)\in\mathcal{F}$ for all data, and for each $s$, nonvanishing for some choice of data. The uniqueness up to $\mathcal{F}^*$ also follows.
\end{proof}
\begin{appremark}
When $F$ is archimedean, using \eqref{eq:gamma RS}, the equality $\gamma(s,\pi\times\tau,\psi)=\gamma^{\mathrm{RS}}(s,\pi\times\tau,\psi)$ and the fact that there is a finite vertical strip outside of which $Z(s,\omega,W)$ and $Z^*(s,\omega,W)$ do not have common poles, one deduces that the set of poles of the integrals is obtained by removing finitely many poles from a set of poles of
classical Gamma functions, and adding finitely many poles in the strip.
\end{appremark}
By Theorems~\ref{theorem:uniqueness for bilinear RS} and \ref{theorem:gcd} there is a factor $\epsilon(s,\pi\times\tau,\psi)\in\mathcal{F}^*$ such that for all $\omega$ and $W$,
\begin{align}\label{eq:L RS}
\epsilon(s,\pi\times\tau,\psi)\frac{Z(s,\omega,W)}{\gcdzzz(s,\pi\times\tau)}=
\pi(-1)^{k-1}\frac{Z^*(s,\omega,W)}{\gcdzzz(1-s,\pi^{\vee}\times\tau^{\vee})}.
\end{align}
Combined with \eqref{eq:gamma RS} this gives the usual functional equation
\begin{align}\label{eq:L RS and Gamma}
\gamma(s,\pi\times\tau,\psi)=\epsilon(s,\pi\times\tau,\psi)\frac{\gcdzzz(1-s,\pi^{\vee}\times\tau^{\vee})}{\gcdzzz(s,\pi\times\tau)}.
\end{align}
\begin{applemma}\label{lemma:lemma poles RS general}
Assume that $\pi$ is the unique irreducible quotient of $\Ind_{R_{\beta}}^G(\sigma_{\beta})$ where
$\sigma_{\beta}=\otimes_{i=1}^{d}\sigma_i$, $\sigma_i=\absdet^{a_i}\sigma_{i,0}$, each
$\sigma_{i,0}$ is tempered and $a_1>\ldots>a_{d}$. Denote
$\beta'=(\beta_2,\ldots,\beta_d)$, $\sigma_{\beta'}=\otimes_{i=2}^d\sigma_{i}$ and let
$\pi'$ be the unique irreducible quotient of $\Ind_{R_{\beta'}}^{\GL_{c-\beta_1}}(\sigma_{\beta'})$. Then $\gcdzzz(s,\pi'\times\tau)=P(s)\gcdzzz(s,\pi\times\tau)$ for $P(s)\in\mathcal{F}$.
\end{applemma}
\begin{proof}
It suffices to show that for each $s$ which is a pole of $\gcdzzz(s,\pi'\times\tau)$, there is a choice of data for which
$Z(s,\omega,W)/\gcdzzz(s,\pi'\times\tau)$ is nonzero at $s$.
Put $l=\beta_1$, $G'=\GL_{c-l}$ and $R=R_{(l,c-l)}$ ($M_R\cong\GL_l\times G'$). The representation $\pi$ is the Langlands quotient
of $\Ind_R^G(\sigma_1\otimes\pi')$ and the image of the
convergent intertwining operator $\Ind_R^G(\sigma_1\otimes\pi')\rightarrow\Ind_{R^-}^G(\sigma_1\otimes\pi')$ given by
$I(\varphi)(g)=\int_{U_R^-}\varphi(ug)\,du$. Additionally $\pi^{\vee}$ is the Langlands quotient
of $\Ind_{R^-}^G(\sigma_1^{\vee}\otimes{\pi'}^{\vee})$. Let $\langle,\rangle$ be the canonical pairing on
$(\sigma_1\otimes\pi')\otimes(\sigma_1^{\vee}\otimes{\pi'}^{\vee})$, and let $\varphi^{\vee}$ be an arbitrary vector in the space of $\Ind_{R^-}^G(\sigma_1^{\vee}\otimes{\pi'}^{\vee})$. Then, the following function
\begin{align}\label{eq:strange coefficients}
\omega(g)=\int\limits_{R^-\backslash G}\langle I(\varphi)(g_0g),\varphi^{\vee}(g_0)\rangle\,dg_0=
\int\limits_{M_R\backslash G}\langle\varphi(g_0g),\varphi^{\vee}(g_0)\rangle\,dg_0
\end{align}
given by an absolutely convergent integral (use Lemma~\ref{lemma:intconvergence}) is a matrix coefficient of $\pi$.

Substituting \eqref{eq:strange coefficients} into the integral~\eqref{eq:poles GL2 start 00} (instead of $\int_{R\backslash G}$) and using the analog of
\eqref{eq:collapsing formula}, we formally obtain
\begin{align}\label{eq:poles GL2 start 01}
Z(s,\omega,W)=&\int\limits_{R\times R^-\backslash G\times G}\int\limits_{U_R^-}\int\limits_{M_R}\int\limits_{U_R}
\delta_{R}^{-1}(m)W(\appembeddingOne(y^{-1}bmg_1)\appembeddingTwo(g_2)
)|\det m|^{\complexSbrevity}
\\&\tau^{\vee}(\det g_2)\langle\varphi(mg_1),\varphi^{\vee}(g_2)\rangle\,db\,dm\,dy\,d(g_1,g_2).\notag
\end{align}
Here we set $\complexSbrevity=s+(kc-2c+1)/2$ for brevity and also used \eqref{eq:additional equivariance of W} with $g=y$.

To justify this formal step over archimedean fields, we observe that the r.h.s.~ of \eqref{eq:poles GL2 start 01} is absolutely convergent in $\Real(s)\gg0$.
Write $y=m_yb_yk_y$ with $m_y\in M_R$, $b_y\in U_R$ and $k_y\in K_G$ and note that we can take $g_1,g_2\in K_G$. After changing variables in $b$ and $m$, and using \eqref{eq:additional equivariance of W}, we need to show that the following integral is finite:
\begin{align*}
&\int\cdots\int\left|\delta_{R}^{-1}(m)W(\appembeddingOne(bmg_1)\appembeddingTwo(k_yg_2))\langle\varphi(m_ymg_1),\varphi^{\vee}(g_2)\rangle\right|
|\det m|^{\complexSbrevity}\,d(\cdots)\notag.
\end{align*}
By \eqref{eq:conv main N 2}, for some $\Phi\in\mathcal{S}(\Mat_c)$ and a constant $d_0>0$, this integral is bounded by
\begin{align*}
&\int\cdots\int\left|\delta_{R}^{-1}(m)\Phi(bm)\langle\varphi(m_ymg_1),\varphi^{\vee}(g_2)\rangle\right|
|\det{m}|^{\complexSbrevity-d_0}\,d(\cdots).\notag
\end{align*}
Now the proof of Lemma~\ref{lemma:intconvergence} shows that the $dy$-integral
is bounded by $C||m||^D$, where $||\cdot||$ is a norm on $G$, and $C$ and $D$ are constants. The convergence follows.

Next, as in the proof of Corollary~\ref{corollary:main poles} (and using \cite{DM} over archimedean fields), for any $\varrho$ (resp., $\varrho^{\vee}$) in the space of
$\sigma_1\otimes\pi'$ (resp., $\sigma_1^{\vee}\otimes{\pi'}^{\vee}$) and $W\in W_{\psi}(\rho_c(\tau))$, we can choose data $(\varphi_i,\varphi^{\vee}_i,W_i)$ such that a finite sum of integrals \eqref{eq:poles GL2 start 01} is equal to
\begin{align}\label{eq:poles GL2 start 02}
&\int\limits_{U_R^-}\int\limits_{M_R}\int\limits_{U_R}
\delta_{R}^{-1}(m)W(\appembeddingOne(y^{-1}bm))|\det m|^{\complexSbrevity}
\langle\varepsilon(m)\varrho,\varrho^{\vee}\rangle\,db\,dm\,dy.
\end{align}
If $F$ is archimedean, the absolute convergence of \eqref{eq:poles GL2 start 02} in $\Real(s)\gg0$ follows from the above proof of the similar convergence for \eqref{eq:poles GL2 start 01}. If $F$ is non-archimedean we assume this convergence for \eqref{eq:poles GL2 start 02}, momentarily, but note that for each $W$, one pair $(\varphi_1,\varphi_1^{\vee})$ is enough to pass from \eqref{eq:poles GL2 start 01} to \eqref{eq:poles GL2 start 02}. Thus the absolute convergence of \eqref{eq:poles GL2 start 02} in $\Real(s)\gg0$ will justify the passage from $Z(s,\omega,W)$ to \eqref{eq:poles GL2 start 01} (see the proof of Corollary~\ref{corollary:main poles}).

The integral~\eqref{eq:poles GL2 start 02} admits meromorphic continuation, because $Z(s,\omega,W)$ does.
We have thus reduced the proof to the statement that for each $s$ which is a pole of $\gcdzzz(s,\pi'\times\tau)$, there exist data
$(\varrho,\varrho^{\vee},W)$ for which \eqref{eq:poles GL2 start 02} divided by $\gcdzzz(s,\pi'\times\tau)$ is nonzero at $s$.

Consider the integral~\eqref{eq:poles GL2 start 02}. Put $m=\diag(m_1,m')\in M_R$ and identify $m_1=\diag(m_1,I_{c-l})$ and
$m'=\diag(I_l,m')$. Each $x\in V_{1,2}^1$ commutes with $\appembeddingOne(m'b)$, and
\begin{align*}
W(\appembeddingOne(y^{-1}m_1)x)=\psi(\tr m_1x)W(\appembeddingOne(y^{-1}m_1)).
\end{align*}
Hence if $\phi(W)(h)=\int_{V_{1,2}^1}x\cdot W(h)\phi(x)\,dx$ where $\phi\in\mathcal{S}(\Mat_l)$, the integral~\eqref{eq:poles GL2 start 02} equals
\begin{align}\label{eq:poles GL2 start 03}
&\int\limits_{\GL_l}\delta_R^{-1/2}(m_1)Z^1(s,\sigma_1(m_1)\otimes\pi'(I_{c-l})\varrho,\varrho^{\vee},\appembeddingOne(m_1)\cdot W)\widehat{\phi}(m_1)|\det m_1|^{\complexSbrevity}\,dm_1,
\end{align}
where
\begin{align*}
Z^1(s,\varrho,\varrho^{\vee},W)=&\int\limits_{U_R^-}\int\limits_{G'}\int\limits_{U_R}
\delta_{R}^{-1/2}(m')W(\appembeddingOne(y^{-1}bm'))|\det m'|^{\complexSbrevity}
\langle\sigma_1(I_l)\otimes\pi'(m')\varrho,\varrho^{\vee}\rangle\,db\,dm'\,dy.\nonumber
\end{align*}

Over non-archimedean fields it is clear that if $Z^1(s,\varrho,\varrho^{\vee},W)/\gcdzzz(s,\pi'\times\tau)$ is nonzero at $s$, then by choosing $\widehat{\phi}$ compactly supported near $I_l$, the integral~\eqref{eq:poles GL2 start 03}, and thereby \eqref{eq:poles GL2 start 02}, divided by $\gcdzzz(s,\pi'\times\tau)$ is also nonzero at $s$.

Over archimedean fields we will show that $Z^1(s,\varrho,\varrho^{\vee},W)$ is a meromorphic function which is continuous in the data $\varrho$, $\varrho^{\vee}$ and $W$.
It then follows that, when $\widehat{\phi}$ is compactly supported near $I_l$, the integral~\eqref{eq:poles GL2 start 03} is a meromorphic function when we regard $Z^1(s,\varrho,\varrho^{\vee},W)$ as a meromorphic function, and as meromorphic functions \eqref{eq:poles GL2 start 03} and \eqref{eq:poles GL2 start 02} are equal. Assume $Z^1(s,\varrho,\varrho^{\vee},W)/\gcdzzz(s,\pi'\times\tau)$ is nonzero at $s$. By the continuity of $Z^1(s,\varrho,\varrho^{\vee},W)$, we can take $\widehat{\phi}$ compactly supported near $I_l$ such that \eqref{eq:poles GL2 start 03}, and thereby \eqref{eq:poles GL2 start 02}, divided by $\gcdzzz(s,\pi'\times\tau)$ is nonzero at $s$.

We turn to $Z^1(s,\varrho,\varrho^{\vee},W)$. Since $\appembeddingOne(bm')$ centralizes $V_{1,2}^1\cdot V_{1,2}^2$, and for
$x\in V_{1,2}^2$, ${}^{\appembeddingTwo(y)}x\in V_{1,2}^1\cdot V_{1,2}^2$, we can use \eqref{eq:additional equivariance of W} again with $g=y^{-1}$
and obtain
\begin{align*}
W(\appembeddingOne(y^{-1}bm')x)=W(\appembeddingOne(bm')\appembeddingTwo(y)x)=\psi^{-1}(\tr xy)\,\appembeddingTwo(y)\cdot W(\appembeddingOne(bm')).
\end{align*}
Thus by taking a finite sum $\sum_i\phi_i(W_i)$, where
$\phi_i(W_i)(h)=\int_{V_{1,2}^2}x\cdot W_i(h)\phi_i(x)\,dx$ with
$W_i\in W_{\psi}(\rho_c(\tau))$ and $\phi_i\in\mathcal{S}(\Mat_{l\times c-l})$, and using \cite{DM}, we can obtain an arbitrary integral
of the form
\begin{align}\label{int:the last one}
&\int\limits_{G'}\int\limits_{U_R}\delta_{R}^{-1/2}(m')W(\appembeddingOne(bm'))|\det m'|^{\complexSbrevity}\langle\sigma_1(I_l)\otimes\pi'(m')\varrho,\varrho^{\vee}\rangle\,db\,dm'.
\end{align}
This integral is absolutely convergent in $\Real(s)\gg0$, by Lemma~\ref{lemma:main convergence}.
Over non-archimedean fields this already implies the absolute convergence of the integral~\eqref{eq:poles GL2 start 02}.

We continue as in the proof of Theorem~\ref{theorem:RS mult} and apply the arguments taking \eqref{eq:poles GL2 start 4} to \eqref{int:GL Lemma after removing V}, with
$G'$ instead of $M_R$, to obtain the integral $Z(s,\omega',W')$ for arbitrary $\omega'$ and $W'$, where $\omega'$ is a
matrix coefficient of $\pi'$ and $W'\in W_{\psi}(\rho_{c-l}(\tau))$.
The analytic properties and continuity properties of $Z^1(s,\varrho,\varrho^{\vee},W)$ now follow from those of $Z(s,\omega',W')$, for arbitrary data. Since any pole of $\gcdzzz(s,\pi'\times\tau)$ is obtained by some $Z(s,\omega',W')$, the proof is complete for both non-archimedean and archimedean fields.
\end{proof}
\begin{appremark}
For the similar result for Rankin--Selberg integrals over non-archimedean fields see \cite[Lemma~9.3]{JPSS}.
\end{appremark}
Let $L(s,\pi\times\tau)$ be the Rankin--Selberg $L$-factor of $\pi\times\tau$ (\cite{JPSS,JS3}).
\begin{appproposition}\label{appproposition:GL2 poles}
$\gcdzzz(s,\pi\times\tau)^{-1}L(s,\pi\times\tau)\in\mathcal{F}$.
\end{appproposition}
\begin{proof}
Write $\pi$ and $\pi'$ using the notation of Lemma~\ref{lemma:lemma poles RS general}.
Denote $\beta^*=(\beta_d,\ldots,\beta_1)$ and $(\beta^*)'=(\beta_{d-1},\ldots,\beta_1)$. Then
$\pi^{\vee}$ is the unique irreducible quotient of $\Ind_{R_{\beta^*}}^G(\otimes_{i=d}^1\sigma_i^{\vee})$ and $\Ind_{R_{(\beta_d,c-\beta_d)}}^G(\sigma_d^{\vee}\otimes(\pi^{\vee})')$, where
$(\pi^{\vee})'$ is the unique irreducible quotient of $\Ind_{R_{(\beta^*)'}}^{\GL_{c-\beta_d}}(\otimes_{i=d-1}^1\sigma_i^{\vee})$.

We argue using induction on $d$. When $d=1$, $\pi=\sigma_1=\absdet^{a_1}\sigma_{1,0}$ is essentially tempered.
Then
\begin{align*}\gamma^{\mathrm{RS}}(s,\sigma_1\times\tau,\psi)\simeq L(1-s-a_1,\sigma_{1,0}^{\vee}\times\tau^{\vee})/L(s+a_1,\sigma_{1,0}\times\tau).
\end{align*}
Here there are no cancellations between the $L$-factors --- the poles of $L(s+a_1,\sigma_{1,0}\times\tau)$ are contained in $\Real(s)+a_1<1/2$ while those of $L(1-s-a_1,\sigma_{1,0}^{\vee}\times\tau^{\vee})$ are in $\Real(s)+a_1>1/2$ ($\tau$ is generic and unitary). Hence
\eqref{eq:L RS and Gamma} and Theorem~\ref{theorem:RS mult} imply
\begin{align*}
\frac{\gcdzzz(1-s,\sigma_1^{\vee}\times\tau^{\vee})}{\gcdzzz(s,\sigma_1\times\tau)}\simeq\frac{L(1-s,\sigma_1^{\vee}\times\tau^{\vee})}{L(s,\sigma_1\times\tau)}.
\end{align*}
Therefore $\gcdzzz(s,\sigma_1\times\tau)^{-1}L(s,\sigma_1\times\tau)\in\mathcal{F}$, though it might not be in $\mathcal{F}^*$ (precisely when the l.h.s.~ has common poles).

Assume $d>1$. We proceed as in \cite[Proposition~9.4]{JPSS} (but our statement is weaker).
There (always) exist $A(s),B(s)\in\mathcal{M}$ such that
\begin{align*}
\gcdzzz(s,\pi\times\tau)=A(s)\prod_{i=1}^{d}L(s,\sigma_i\times\tau),\quad
\gcdzzz(1-s,\pi^{\vee}\times\tau^{\vee})=B(1-s)\prod_{i=1}^{d}L(1-s,\sigma_i^{\vee}\times\tau^{\vee}).
\end{align*}
Now according to \eqref{eq:L RS and Gamma} and Theorem~\ref{theorem:RS mult},
\begin{align*}
\frac{\gcdzzz(1-s,\pi^{\vee}\times\tau^{\vee})}{\gcdzzz(s,\pi\times\tau)}
\simeq\gamma^{\mathrm{RS}}(s,\pi\times\tau,\psi)\simeq\prod_{i=1}^d\frac{L(1-s,\sigma_i^{\vee}\times\tau^{\vee})}{L(s,\sigma_i\times\tau)}.
\end{align*}
Hence $A(s)/B(1-s)$ belongs to $\mathcal{F}^*$.

Applying the induction hypothesis to $\gcdzzz(s,\pi'\times\tau)$ and since $L(s,\pi'\times\tau)=\prod_{i=2}^dL(s,\sigma_i\times\tau)$, there is $Q(s)\in\mathcal{F}$ satisfying
\begin{align*}
\prod_{i=2}^dL(s,\sigma_i\times\tau)=Q(s)\gcdzzz(s,\pi'\times\tau).
\end{align*}
Combining this with Lemma~\ref{lemma:lemma poles RS general}, there is $P(s)\in\mathcal{F}$ such that
\begin{align*}
\gcdzzz(s,\pi\times\tau)&=A(s)\prod_{i=1}^{d}L(s,\sigma_i\times\tau)=A(s)L(s,\sigma_1\times\tau)Q(s)\gcdzzz(s,\pi'\times\tau)\\&=
A(s)Q(s)P(s)L(s,\sigma_1\times\tau)\gcdzzz(s,\pi\times\tau).
\end{align*}
Thus $A(s)Q(s)P(s)L(s,\sigma_1\times\tau)\simeq1$ and $A(s)L(s,\sigma_1\times\tau)\in\mathcal{M}$ is nowhere vanishing. By the similar argument applied to $\gcdzzz(1-s,(\pi^{\vee})'\times\tau)$ we deduce that
$B(1-s)L(1-s,\sigma_d^{\vee}\times\tau^{\vee})\in\mathcal{M}$ is nowhere vanishing. Since $a_1\geq a_d$, the poles of $L(s,\sigma_1\times\tau)$ and $L(1-s,\sigma_d^{\vee}\times\tau^{\vee})$ are disjoint whence $A(s)$ and $B(1-s)$ do not have common zeros, and because $A(s)/B(1-s)\in\mathcal{F}^*$, both $A(s)$ and $B(1-s)$ are nowhere vanishing. Since
$L(s,\pi\times\tau)=\prod_{i=1}^dL(s,\sigma_i\times\tau)$, we conclude that
$\gcdzzz(s,\pi\times\tau)=A(s)L(s,\pi\times\tau)$ whence
$\gcdzzz(s,\pi\times\tau)^{-1}L(s,\pi\times\tau)\in\mathcal{F}$.
\end{proof}
\begin{apptheorem}\label{apptheorem:GL2 poles}
Let $\pi\in\Irr(G)$. If $F$ is non-archimedean, suppose $\pi$ is unramified. Let
$\tau\in\IrrGen(\GL_k)$ be unitary, $k>1$.
For each $s$, there is a matrix coefficient $\omega$ of $\pi$ and a function $W\in W_{\psi}(\rho_c(\tau))$ such that
$Z(s,\omega,W)/L(s,\pi\times\tau)$
is nonzero (but may have a pole). If $F$ is archimedean, we can further assume $\omega$ is $K_G$-finite and $W$ is $K_H$-finite.
\end{apptheorem}
\begin{proof}
Direct application of Proposition~\ref{appproposition:GL2 poles}.
The passage to $K_G$- and $K_H$-finite data follows because the meromorphic continuation of the integrals is continuous.
\end{proof}

\def\cprime{$'$} \def\cprime{$'$} \def\cprime{$'$}

\end{document}